\documentclass[11pt]{amsart}

\usepackage{amsthm, amsfonts, graphicx, pinlabel, amssymb}
\usepackage{enumitem}	
\usepackage[all,import]{xy}
\usepackage[usenames,dvipsnames,table]{xcolor}
\usepackage[color=Thistle,textsize=small]{todonotes}
\usepackage[normalem]{ulem}

\newcommand{\e}{\ensuremath{\eta}} 
\newcommand{\te}{\tilde{\e}}

\newcommand{\pd}[2]{\ensuremath{{\partial_{#2} #1}}}

\def\TM+{T^*(\rr_+ \times M)}
\def\rp{\mathbb R_{>0}}

\newcommand{\rr}{\ensuremath{\mathbb{R}}}
\newcommand{\zz}{\ensuremath{\mathbb{Z}}}
\newcommand{\ZZ}{\ensuremath{\mathbb{Z}}}

\newcommand{\Sphere}{\ensuremath{\mathbb{S}}}

\theoremstyle{plain}
\newtheorem{thm}{Theorem}[section]
\newtheorem{cor}[thm]{Corollary}
\newtheorem{lem}[thm]{Lemma}

\newtheorem{prop}[thm]{Proposition}

\newtheorem{ques}[thm]{Question}

\newtheorem{choice}[thm]{Choice}

\theoremstyle{definition}
\newtheorem{defn}[thm]{Definition}
\newtheorem{notation}[thm]{Notation}

\newtheorem{assume}[thm]{Assumption}

\newtheorem{construction}[thm]{Construction}

\theoremstyle{remark}
\newtheorem{rem}[thm]{Remark}
\newtheorem{ex}[thm]{Example}

\numberwithin{equation}{section} 
\usepackage{hyperref}

\newcommand{\dfn}[1]{{\textbf {#1}}}
\newcommand{\leg}{\ensuremath{\Lambda}}

\newcommand{\nc}{\newcommand}
\nc{\eqnn}{\begin{equation*}}
\nc{\eqnnd}{\end{equation*}}
\nc{\eqnd}{\end{equation}}
\nc{\eqn}{\begin{equation}}
\nc{\RR}{\mathbb{R}}
\nc{\CC}{\mathbb{C}}
\nc{\HH}{\mathbb{H}}
\nc{\bk}{{\bf k}}
\nc{\lmax}{\overline{\ell}}
\nc{\lmin}{\underline{\ell}}
\nc{\co}{\colon}

\def\clag{\mathcal L} 
\def\sclag{L} 
\def\Infin{\Omega}
\def\infin{\omega}
\def\col{\operatorname{colim}}
\def\Co{\operatorname{Cone}}
\def\id{\operatorname{id}}
\def\grad{\operatorname{grad}}
\def\FF{\mathbb F}

\nc{\DMO}{\DeclareMathOperator}
\def\claghiro{\overline{\mathcal L}} 

\nc{\LLeg}{\mathcal{L}\!\operatorname{eg}}
\nc{\GGen}{\mathcal{G}\!\operatorname{en}}
\DMO{\gen}{gen}
\DMO{\aut}{Aut}
\DMO{\swap}{swap}
\DMO{\flow}{Flow}
\DMO{\diff}{Diff}
\DMO{\del}{\partial}
\DMO{\colim}{colim}
\DMO{\hocolim}{hocolim}
\DMO{\cone}{Cone}
\nc{\Top}{\mathcal{T}\!\operatorname{op}}
\nc{\spectra}{\mathcal{S}\!\operatorname{pectra}}
\nc{\inftycat}{\mathcal{C}\!\operatorname{at}_\infty}
\nc{\GF}{\mathcal{GF}}
\nc{\immto}{\looparrowright}
\nc{\gfc}{C}

\begin{document}
 
\title[{A Stable Homotopy Invariant for Legendrians}]
	{A stable homotopy invariant for Legendrians with generating families}
\author{Hiro Lee Tanaka}
\author{Lisa Traynor}
\date{\today}

\begin{abstract}  
We construct a stable homotopy type invariant
for any Legendrian submanifold in a jet bundle equipped with a  linear-at-infinity generating family. We show that this spectrum lifts the generating family homology groups.
When the generating family extends to a generating family for an embedded Lagrangian filling, we lift the Seidel isomorphism to the spectrum level. As applications, we establish topological constraints on Lagrangian fillings arising from generating families, algebraic constraints on whether generating families admit fillings, and lower bounds on how many fiber dimensions are needed to construct a generating family for a Legendrian.
 \end{abstract}
\maketitle

\tableofcontents

\section{Introduction}

A central problem in contact topology is  the search for invariants of Legendrian submanifolds of contact manifolds.
 Given $\rr^{2n+1}$ with its standard contact structure $\xi_0 = \ker (dz - y\,dx)$, 
 there are classical integer-valued invariants of a closed Legendrian $\leg \subset \RR^{2n+1}$ known as the Thurston-Bennequin and rotation numbers; for definitions, see \cite{ees:high-d-geometry}.  For a Legendrian with vanishing rotation number, there exist categorifications: homological invariants of the Legendrian that can recover the Thurston-Bennequin number; \cite[Proposition 5.7] {chv},
 \cite[Proposition 3.3]{ees:high-d-geometry}.

Indeed, given a Legendrian $\leg$ with an augmentation $\epsilon$,  it is 
possible to define {\it linearized contact homology}   $LCH_*(\leg, \epsilon; R)$ \cite{chv, ees:high-d-analysis, Karlsson} via the theory of holomorphic curves.
When  $\leg$ is equipped with a generating family $f$, 
building on work in \cite{traynor:shomology, traynor:gf-polys}, Fuchs and Rutherford~\cite{f-r} defined {\em generating family homology}, $GFH_*(\leg, f; R)$. For both these homology theories $R$ denotes a ring of coefficients.  
 It has been established~\cite{f-r}
 that for a $1$-dimensional Legendrian
  the existence of a linear-at-infinity $f$ implies the existence of an augmentation $\epsilon_f$ and an isomorphism
	\eqn\label{eqn. gfh is lch}
	GFH_{*+c}(\leg, f;\ZZ/2\ZZ) \cong LCH_*(\leg, \epsilon_f;\ZZ/2\ZZ).
	\eqnd
Here, $c$ is a constant equalling either 0 or 1 depending on the convention. In this and further work, we will  employ the convention that $c=1$ (see Remark~\ref{remark. grading shift in LCH} and Remark~\ref{rem:GFH-indexing}). The $c=0$ convention will only appear briefly in Section~\ref{section. GFH from GS spectra}, with explicit notation indicating the convention shift.

Homological invariants can admit stable homotopy refinements. 
 Indeed, such refinements have a rich history, some of which we summarize in Section~\ref{section. history}. Thus, it is natural to ask:
 \begin{ques} Are there spectral lifts of linearized contact homology or of generating family homology?
 \end{ques}
For linearized contact homology, the answer is widely expected to be yes under favorable circumstances, and the (still open) construction of such a lift falls under the purview of an active field, often called Floer homotopy theory. 
For generating family homology, we provide an affirmative answer in this work. 
More precisely,
fix a smooth manifold $B$, a closed Legendrian $\Lambda \subset J^1 B$ in the $1$-jet bundle of $B$, and a linear-at-infinity generating family 
	\eqnn
	f: B \times \RR^N \to \RR
	\eqnnd 
for $\leg$ (see Definition~\ref{defn. linear-at-infinity}). From this data we define a stable homotopy type (Definition~\ref{defn:leg-spectrum})
	\eqn\label{eqn. gfc over sphere}
	\gfc(\leg,f;\Sphere)
	\eqnd
that we call the {\em generating family spectrum} associated to the pair $(\leg,f)$.  We prove that our spectral lifts are invariants of the pair $(\Lambda,f)$ (Theorem~\ref{theorem. invariance}) and that the spectrum recovers generating family homology 
 (Theorem~\ref{thm:spec-lift}). We also establish a highly useful structural result: The Seidel isomorphism lifts to  stable homotopy (Theorem~\ref{thm:suspension computation}). 

\subsection{Main results}

\begin{thm}[Proposition~\ref{prop:equiv-class-specta} and Theorem~\ref{thm:iso-spectra}.]
\label{theorem. invariance}
The generating family spectrum of $(\leg,f)$, $C(\leg, f; \Sphere)$, is invariant under Legendrian isotopy and equivalence of generating families (Definition~\ref{defn:gf-equiv}). 
\end{thm}

\begin{rem}
$\gfc(\leg,f;\Sphere)$ is an invariant not of $\leg$ alone, but of the pair $(\leg,f)$. 
One expects that the (Spanier-Whitehead duals to) the collection of spectra $\{\gfc(\leg,f;\Sphere)\}_f$ form endomorphisms in a non-unital $\infty$-category associated to $\Lambda$ -- see Section~\ref{section. comultiplications}.
\end{rem}

The following is proven in Section~\ref{section. GFH from GS spectra}. See Definition~\ref{defn:GF-homology} for the definition of $GFH$ and for an explicit description of the grading convention used in this work.
\begin{thm} \label{thm:spec-lift} 
The generating family spectrum is a lift of generating family homology. That is, 
for any coefficient abelian group $A$,
the generating family homology of $(\leg,f)$ is isomorphic to the homology of the generating family spectrum:
	\eqnn
	\forall k \in \ZZ, \qquad
	H_{k}(\gfc(\leg, f; \Sphere) ; A)\cong {GFH}_k(\leg, f;A).
	\eqnnd
\end{thm}

\begin{rem}
Theorem~\ref{thm:spec-lift} in part explains the appearance of the sphere spectrum $\Sphere$ in our notation. Indeed, the notation~\eqref{eqn. gfc over sphere} is meant to evoke ``generating family chains with coefficients in the sphere spectrum.'' We view classical $GFH$ as a linear invariant computed using $\ZZ$-linear coefficients, while the generating family spectrum is a lift to sphere-spectrum-linear coefficients.
\end{rem}

Generating families pose interesting geometric questions of their own. For example, given a Legendrian $\leg$ equipped with a linear-at-infinity generating family $f:B \times \rr^N \to \rr$, one can define the {\it dimension} of $f$, $\dim f$,  to be the fiber dimension $N$, and ask to reduce $\dim f$.  More precisely, let $[(\leg, f)]$ denote the equivalence class of $(\leg, f)$ generated by stabilization, fiberwise diffeomorphism, and Legendrian isotopy (Section~\ref{section. equivalent gfs} and Proposition~\ref{prop:leg-persist}). We can ask: What is
 $$
 N_{\min}[(\leg, f)] := \min_{(\leg',f') \in [(\leg, f)]} \{ \dim f' \}?
 $$ 

For a fixed 1-dimensional Legendrian equipped with a graded, normal ruling, Fuchs and Rutherford  
gave an algorithm to construct a generating family that would induce this ruling~\cite[Section 3]{f-r}.
 Thus for a Legendrian equipped with  a ruling, if we consider the generating family $f$ obtained through
the Fuchs-Rutherford algorithm, one obtains an upper bound on $N_{\min}[(\leg, f)]$.

 In all dimensions, the next theorem shows that the generating family spectrum produces a lower bound on $N_{\min}[(\leg, f)]$. It is proven in Section~\ref{ssec:dim-bound}; we review the definition of suspension spectra of a pointed space in Definition~\ref{defn:suspension-spec} and the suspension of spectra in Notation~\ref{notation:shift-spec} of the Appendix.

\begin{thm}
\label{theorem. bound on family dimension}
Suppose $\leg \subset J^1B$, and 
let $N$ be the minimal non-negative integer for which $\Sigma^{N} \gfc(\leg,f;\Sphere)$ is equivalent to a suspension spectrum  -- i.e., for which there exists a pointed space $A$ and an equivalence of spectra
$$\gfc(\leg,f;\Sphere) \simeq  \Sigma^{-N}( \Sigma^\infty(A)).$$
Then
$$N \leq N_{\min}[(\leg, f)].$$
In particular, $f$ itself cannot arise as a stabilization of an $(N-1)$-dimensional generating family.
\end{thm}

\noindent

The stable-homotopy bounds of Theorem~\ref{theorem. bound on family dimension} immediately give bounds using homology. In the following, and as we do throughout this work, we use the grading convention $c=1$ in~\eqref{eqn. gfh is lch} -- i.e., the convention in Definition~\ref{defn:GF-homology}.

\begin{cor}\label{cor. gfh degrees bound dimension} 
Fix a linear-at-infinity generating family $f$ for an $m$-dimensional {connected} Legendrian $\leg \subset J^1B$. 
Suppose, for some coefficient group,  the generating family homology $GFH_{-k}(\leg,f)$ is non-zero for some $k \geq 0$.
Suppose also that for some coefficient group (not necessarily equal to the previous coefficient group) $GFH_{\ell}(\leg,f)$ is non-zero for some $\ell \geq 0$.  Then $N_{\min} \geq \max\{k+1 , \ell  - m\}$.
 \end{cor}

\begin{proof} Suppose $\leg \subset J^1B$ is an $m$-dimensional {connected} Legendrian. 
{Because $\leg$ is connected, we may as well assume $B$ is connected.}
Given a generating family $f: B \times \RR^N \to \RR$ for $\leg$, the difference function has domain $B \times \RR^{2N}$, so for any real number $\epsilon$, the associated sublevel quotient $B \times \RR^{2N} / \{f \leq \epsilon\}$ is homotopy equivalent to some pointed CW complex $A$ with cells of dimension at most $2N + \dim B = 2N + m$. By Definition~\ref{defn:leg-spectrum}, the generating family spectrum $C(\leg,f;\Sphere)$ is equivalent to $\Sigma^{-N}\Sigma^\infty A$  -- the suspension spectrum of $A$, shifted $-N$ times (for $\epsilon$ positive and small enough). So   the generating family spectrum is generated by spheres of at least degree $-N$ and at most degree $N + m$.  In particular, generating family homology can only be non-zero in degrees between $-N$ and $N+ m$, inclusive. 

If $f$ is linear-at-infinity, the space $ \{ f \leq \epsilon\} \subset B \times \RR^{2N}$ is non-empty. (See also Remark~\ref{remark. dim f at least 1}.) 
 Because $B \times \RR^{2N}$ is path-connected,
it follows that  the quotient space $A$ is path-connected. In particular, $A$ has reduced homology $\widetilde{H}_0(A) \cong 0$. Applying the shift by $-N$, we see that $GFH_{*}(\leg, f) \neq 0$ implies
that $-N+1 \leq * \leq N+m$.  In particular, defining $k$ and $\ell$ as above, we see that $N \geq k+1$ and $N \geq \ell - m$. This shows $N_{\min} \geq \max\{k+1 , \ell  - m\}$.
\end{proof}

\begin{rem} 
If one chooses the coefficient group to be $\ZZ/2\ZZ$, then $GFH_{-k}(\leg,f)$ is non-zero if and only if $GFH_{m+1+k}(\leg,f)$ is non-zero by duality of $GFH_*(\leg, f)$ \cite[Theorem 1.1]{BST:geography}.   
In particular,  $\ell - m =  (m+1 + k) - m =  k+1$.  
\end{rem} 

\begin{rem}
We caution that duality does not hold for all $(\leg,f)$ in general with arbitrary coefficient groups for homology. For example,  the statement of \cite[Theorem 6.1(1)]{BST:geography} utilizes Poincar\'e duality for $\leg$, and indeed the proof relies on \cite[Lemma~7.1]{S-T:obstruct} which in turn utilizes Alexander/Lefschetz/Poincar\'e duality (which requires orientation hypotheses on $\leg$).
It is not hard to find a counter-example to duality if one violates orientation hypotheses. For example, take $B = \RR P^{2n}$, and let $f: B \times \RR \to \RR$ be a cubic function, constant in the $B$-variable, with two critical points in each fiber of distinct critical values. Then $f$ generates a Legendrian $\leg \cong B \coprod B \subset J^1(B)$, and the generating family spectrum of $(\leg,f)$ is computed in Example~\ref{example. cubic gf spectrum over B} below. Because the homology of $\Sigma^\infty (B_+)$ is the (unreduced) homology of $B$, incorporating the appropriate shifts, one finds (for any coefficient group):
	\eqnn
	GFH_a(\leg,f) \cong H_{a+1}(B),
	\qquad
	\forall a \in \ZZ.
	\eqnnd
Then over $\ZZ$ or over a finite field $\FF_p$ with odd $p$, one can check that the extremal values of $k+1$ and $\ell-m$ do not agree.
\end{rem}

\begin{rem}\label{remark. spectrum level statements are more powerful than homology level statements}
{
One can in fact prove
Corollary~\ref{cor. gfh degrees bound dimension} without ever knowing about generating family spectra -- after all, one does not need to know about spectra to study (a shift of) the cellular chain complex.
Note also that the corollary is  a weaker conclusion than Theorem~\ref{theorem. bound on family dimension}, since there exist spectra whose homology range behaves like a suspension spectrum (of a pointed, connected space), but which are not suspension spectra. One example is $X = \Sigma^{-1}\Sigma^\infty \CC P^N$ for $N \geq 2$. Because the homology of $X$ is only non-zero in degrees $1,3,\ldots,2N-1$, its homology alone leaves open the possibility that $X$ is the suspension spectrum of a connected pointed CW complex}
 -- 
but $X$ cannot be a suspension spectrum because the Steenrod operation $Sq^2$ does not vanish on $H^1(X) \cong H^2(\CC P^N)$. This is one example of spectra harboring more information than chain complexes.
\end{rem}

\begin{ex} \label{ex:FR-ub} 
 Consider the Legendrian $m(5_2)$  knot $\leg_3 \subset \rr^3 = J^1\rr$ (Figure~\ref{fig:stab-5}) with its unique graded, normal ruling.  By the Fuchs-Rutherford algorithm in \cite[Section 3]{f-r}, $\leg$ admits
a generating family $f: \rr \times \rr^7 \to \rr$  that will induce this ruling. Assuming that the index of each strand is at least $2$, the algorithm produces a generating family with dimension equaling two greater than the maximum of the
switch indices, which is $2 + \max\{3, 4, 5\} = 7$ in this example.  Thus $N_{\min}[(\leg_3, f)] \leq  7$.
 Calculations of $LCH(\leg_3, \epsilon)$ from \cite[Section 3]{melvin-shrestha} combined with~\eqref{eqn. gfh is lch} show that $GFH_k(\leg_3, f)$  
 is non-zero in degree $-1,2,3$ with our grading convention of $c = 1$ -- see Equation~\eqref{eqn. gfh is lch}.  Thus, by Corollary~\ref{cor. gfh degrees bound dimension},   $N_{\min}[(\leg_3, f)] \geq 2$.    Since $\leg_3$ can be obtained by  attaching two $0$-handles to an appropriately indexed Hopf link and  unknot that has undergone three Legendrian Reidemeister type-1 moves, one can construct a generating family with $N = 4$; (see the constructions in 
 \cite[Lemma 6.8]{BST:geography} and the corrected statement of \cite[Theorem 4.2]{BST:geography} in~\cite{BST:geography-erratum}).
 To the authors' knowledge, it is unknown whether there exists $(\leg', f') \in [(\leg_3, f)]$ with 
 $\dim f' < 4$. 
 
 More generally, we can generalize the above example and arguments to construct $\leg_d$, having $d \geq 4$ crossings in the right-hand column; see the rightmost front projection in Figure~\ref{fig:stab-5}. 
 For the generating family $f$ associated 
 to this ruling, $GFH_k(\leg_d, f; \zz)$ is non-zero in dimensions $-d+2, 2, d$.  Thus the lower bound given by Corollary~\ref{cor. gfh degrees bound dimension} and the upper bound given by the Fuchs-Rutherford algorithm tell us that   $d-1 \leq N_{\min}[(\leg_d, f)] \leq  d+4$.   Constructions yield a generating family with $N = d+1$.

 \begin{figure}
 \centerline{\includegraphics[height=1in]{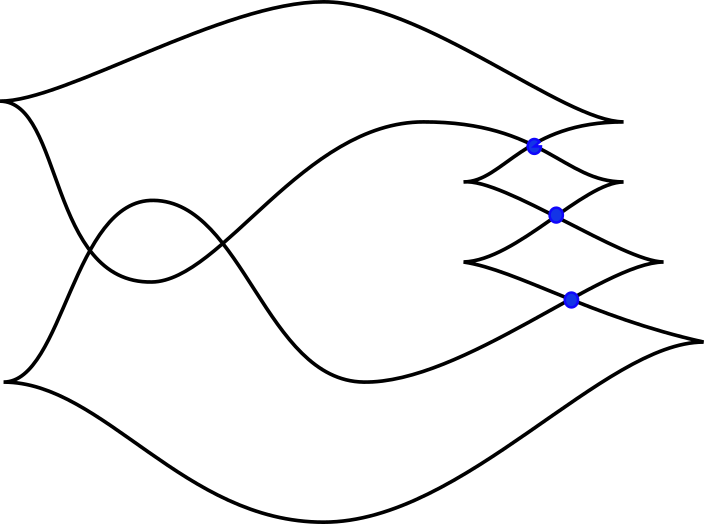} \qquad\qquad  \includegraphics[height=1in]{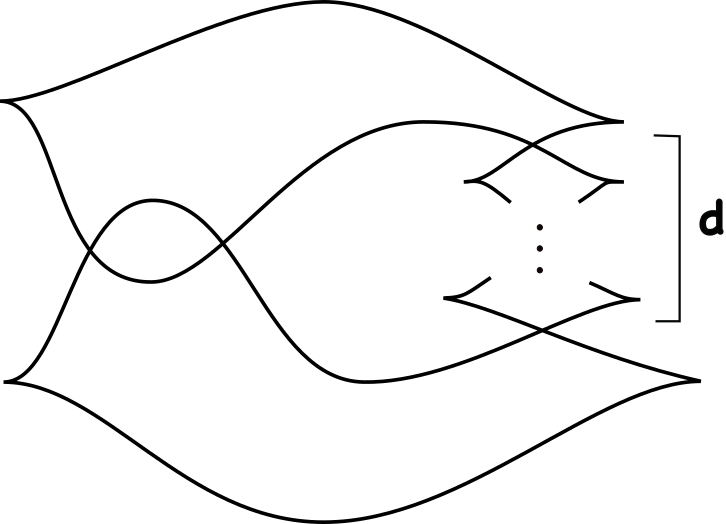}}
  \caption{  On the left, a Legendrian $m(5_2)$ knot $\leg_3$ with generating family $f$ corresponding to the ruling arising from switches at the $3$ marked crossings; the associated
  $GFH_k(\leg_3, f; \zz)$ is non-zero when $k$ is $-1, 2$, or $3$, using our grading convention.
  On the right is a Legendrian twist knot $\leg_d$ that has a ruling
  from the switches at the $d \geq 3$ crossings in the rightmost column.  The associated generating family $f$ will
  satisfy $GFH_k(\leg_d, f; \zz) \not \cong 0$ for $k$ values $-d+2, 2$, and $d$. 
  }
   \label{fig:stab-5}
 \end{figure}
\end{ex}
 
When a generating family $f$ for a Legendrian $\leg$ extends to a ``linearly-controlled" 
generating family $F$ for a Lagrangian filling $\sclag$ of $\leg$ (Definition~\ref{defn:compatible}), 
we may apply our invariant to obstruct the topology of Lagrangian fillings of Legendrian submanifolds. In what follows, we let 
	\eqnn
	L_0 \subset L
	\eqnnd 
be the complement of an open collared end of $L$. In particular, $L_0$ is a compact, codimension zero submanifold of $L$ with boundary diffeomorphic to $\leg$. We prove the following in Section~\ref{section. proof of main theorem}.

\begin{thm}
[The spectral Seidel isomorphism for generating families]
\label{thm:suspension computation}
Let $(L,F)$ be a filling  of $(\Lambda,f)$ in the sense of Definition~\ref{defn:compatible}. Then there is an equivalence of spectra
	\eqnn
	\gfc(\leg,f;\Sphere) \simeq \Sigma^\infty(L_0/\leg),
	\eqnnd
where $\Sigma^\infty(L_0/\leg)$ denotes the suspension spectrum of the quotient $L_0/\leg$.
\end{thm}

\begin{rem}
One can identify the pointed space $L_0/\leg$ with the one-point compactification of $L$.  
\end{rem}

\begin{rem}
Theorem~\ref{thm:suspension computation} is a spectral lift of the 
the Seidel isomorphism for generating family (co)homology~\cite[Theorem 1.5]{S-T:obstruct}. This is the isomorphism between 
 the generating family cohomology groups of the Legendrian boundary and the relative singular cohomology groups of the filling, and thus via Lefschetz duality the homology groups
 of the filling: 
 $$GFH^{k+1}(\leg, f; \mathbb Z/2\mathbb Z) \cong H^{k+1}(L,\partial L; \mathbb Z/2\mathbb Z) \cong H_{n-k}(L^{n+1}; \mathbb Z/2\mathbb Z).
 $$
An analogous isomorphism for linearized contact homology, in line with~\eqref{eqn. gfh is lch}, was established by Ekholm~\cite{ekholm:lagr-cob} and Dimitriglou-Rizell~\cite{DR:polygons}.
 \end{rem}

\begin{rem} For one-dimensional examples of  $\leg$, Theorem~\ref{thm:suspension computation} does not give new invariants — this is because when $L_0$ is two-dimensional, the suspension spectrum $\Sigma^\infty(L_0/\leg)$ is always a wedge sum of spheres, of dimensions that can be read off from the Betti numbers of the (possibly nodal) surface $L_0/\leg$. 
In particular, for connected,  $1$-dimensional Legendrians, the spectrum in Theorem~\ref{thm:suspension computation} carries the same information as the Thurston-Bennequin invariant of $\leg$ and the genus of the filling.  \end{rem} 
Theorem~\ref{thm:suspension computation} gives an alternate path to prove the following previously known result:

\begin{cor}\label{cor. no embedded GFs in Euclidean space}
For any smooth manifold $B$, there does not exist a non-empty, compact embedded Lagrangian in $T^*(B \times \RR_{>0})$ admitting a linear-at-infinity generating family.  In fact,  for all $n \geq 1$, there is no non-empty compact embedded Lagrangian in $T^*\RR^n \cong T^*(\RR^{n-1} \times \RR_{>0})$ admitting a linear-at-infinity generating family.
\end{cor}

\begin{proof}[Proof of Corollary~\ref{cor. no embedded GFs in Euclidean space}.]
First observe that any non-zero linear function, $\ell$,  is a generating family for  $\leg = \emptyset$. 
A simple computation shows $\gfc(\emptyset,\ell;\Sphere)$ is equivalent to the zero spectrum, otherwise known as the suspension spectrum of a point.
 On the other hand any pair $(L,F)$ of a non-empty, compact embedded Lagrangian $L=L_0 \subset T^*(M \times \RR_{0})$
with a linear-at-infinity generating family $F$ is a filling of $(\leg, \ell)$. 
 Thus we can apply Theorem~\ref{thm:suspension computation} to find
	\eqnn
	0 \simeq \Sigma^\infty\left(L_0/\emptyset\right)\simeq \Sigma^\infty\left(L \coprod \ast \right).
	\footnote{ In point-set topology, $X/\emptyset$ is identified with 
$X$, but  the quotient can alternatively be defined via a pushout diagram and this  categorical definition yields $X$ with a disjoint point:
	\eqnn
	X/\emptyset \cong X \coprod \ast.
	\eqnnd
	}
	\eqnnd
But the suspension spectrum of a non-empty space given a disjoint basepoint $\ast$ is never trivial -- its zeroth homology has rank equal to the number of path-connected components of the space. We conclude $L$ must be empty.
\end{proof}

\begin{rem}
\label{remark. no compact exacts in cotangents of open manifolds}
We present Corollary~\ref{cor. no embedded GFs in Euclidean space} and its proof mainly for the novelty of avoiding any holomorphic curve techniques. Indeed, a stronger form of the corollary is well-known to experts: For any open manifold $X$, the cotangent bundle $T^*X$ admits no compact exact Lagrangian. (Note any Lagrangian admitting a generating family is exact.) This stronger form follows from the fact that $T^*X$ is a subcritical Weinstein manifold: Any compact Lagrangian can be made close to a skeleton of $T^*X$ by the Liouville flow, while the skeleton (and hence a neighborhood of itself) is self-displaceable by virtue of being isotropic. On the other hand, no exact compact Lagrangian is self-displaceable.
 t is this self-displaceability that is usually proven using holomorphic curve techniques. 
Here are the details.     Following Gromov's original arguments -- see $2.3.B_1$ and $2.3.B_2$ of~\cite{gromov-1985} -- one deforms the moduli of constant disks via a Hamiltonian isotopy to study a moduli of solutions to a PDE with large inhomogeneous term. Concluding there must be no such solutions, one infers bubbling -- which is not possible for exact Lagrangians. We note that, now-a-days, sheaf-theoretic techniques in cotangent bundles also prove non-displaceability results~\cite{tamarkin-microlocal-condition}.
 \end{rem}

\begin{rem}
Corollary~\ref{cor. no embedded GFs in Euclidean space} can also be proven without a spectral lift of $GFH$ -- one computes $H_0$ of $L_0/\emptyset$  to witness a $\ZZ$ summand when $L_0$ is non-empty.
\end{rem}

Theorem~\ref{thm:suspension computation} also implies the following. Again, we use the $c=1$ grading convention -- i.e., the grading convention in Definition~\ref{defn:GF-homology} -- for generating family homology.
 
\begin{cor}\label{cor. fillings and suspension spectra}
Fix a non-empty $\leg$. If $C(\Lambda,f;\Sphere)$ is not equivalent to a suspension spectrum,
then $(\Lambda,f)$ admits no filling $(L,F)$. In particular, if the generating family homology of $(\Lambda,f)$ does not vanish in all negative degrees $k < 0$ (for all coefficient abelian groups) then the pair $(\Lambda,f)$ does not admit a filling. Further, if the generating family homology is nonzero in degree $k = 0$ and $\leg$ is connected, then $(\Lambda, f)$ admits no connected 
filling $(L, f)$.
\end{cor}

\begin{proof}
By Theorem~\ref{thm:suspension computation}, a filling exhibits $C(\leg,f;\Sphere)$ as a suspension spectrum of a topological space -- thus all $k^{th}$ homology groups vanish for $k < 0$.  For the second claim, if $L$ is connected,
then $(L_0/\leg)$ has trivial reduced $0^{th}$ homology.  
\end{proof}

\begin{rem}
The homology-level statement of Corollary~\ref{cor. fillings and suspension spectra} was already utilized in~\cite{S-T:obstruct} (see Example~1.6 and Theorem~1.7(2) of ibid). As explained in Remark~\ref{remark. spectrum level statements are more powerful than homology level statements}, the spectrum-level statement in Corollary~\ref{cor. fillings and suspension spectra} is more powerful than the homology-level implication.
\end{rem}
\begin{rem}
In our definition of fillings, we demand that $L$ not only be a Lagrangian filling of $\leg$, but that the generating family $f$ extends to a well-behaved generating family $F$ for $L$ (Definition~\ref{defn:compatible}). This style of filling condition -- not just of the manifolds $L$ and $\leg$, but incorporating extra structures -- is already familiar from Floer theory: Such an extension is necessary if we are to compare Floer-type invariants of $\leg$ defined over some dga or ring spectrum $R$ to Floer-type invariants of $L$ defined over the same $R$. In the present paper, $R$ is the sphere spectrum, and we view the existence of a generating family $F$ extending $f$ as analogous to a grading or a null-homotopy of the stable Gauss map extending from $\leg$ to $L$.  See also Remark~\ref{remark. tangential structures on branes are natural}.  
\end{rem}

Utilizing the stable homotopy type, one also has ``generating family homotopy groups" associated to a pair $(\leg, f)$:
\begin{notation}
Given $(\leg,f)$, we let 
	\eqnn
	\pi_k(\leg,f) := \pi_k(\gfc(\leg, f;\Sphere)) 
		\eqnnd
denote the homotopy groups of the generating family spectrum.\end{notation}

We review homotopy groups of spectra in the Section~\ref{section. homotopy groups via mapping spaces} of the Appendix.

\subsection{Examples}
 
 \begin{ex} Suppose $(U^n, f)$ is the $n$-dimensional Legendrian sphere obtained by applying the spinning procedure to the max-$tb$ Legendrian unknot~\cite{BST:geography};   this Legendrian has Lagrangian projection equal to the standard Whitney sphere. Then $(U^n,f)$ admits a filling $(L, F)$  where
 $L$ is a  Lagrangian disk. 
 By Theorem~\ref{thm:suspension computation},
 	\eqnn
	\gfc(U^n,f;\Sphere) \simeq \Sigma^\infty(S^{n+1}).
	\eqnnd
So the homotopy groups of the generating family spectrum are computed as the stable homotopy groups $\pi^s$ of spheres; see Example~\ref{ex:homotopy-suspension-stabilized-homotopy}.
For those unfamiliar with the beautiful and unpredictable nature of stable homotopy groups, we refer to Example~\ref{ex:homotopy-calculations} where we display the first eleven values of $\pi_k^s(S^3) \cong \pi_k(\gfc(U^2,f;\Sphere))$. 
\end{ex}
 \begin{ex} 
 Suppose $(\leg_3, f)$ is the max-$tb$ positive
 Legendrian trefoil with generating family $f$  corresponding to the ruling in Figure~\ref{fig:trefoil} following the Fuchs-Rutherford algorithm \cite{f-r}.  Then it is known that $\Lambda_3$ admits a generating family  filling with $L \cong T^2 \setminus D^2$, so by Example~\ref{example. suspension of torus},
 $$\pi_k(\leg_3, f) \cong \pi_k^{s}(T^2) \cong \pi_k^s(S^1 \vee S^1 \vee S^2) \cong \pi_k^{s}(S^1) \oplus \pi_k^s(S^1) \oplus \pi_k^s(S^2).$$
\end{ex} 
 
\begin{rem}
We would like to both caution and intrigue the reader unfamiliar with stable homotopy groups of spheres. The intrigue: It is impossible to recover/compute the homotopy groups of a spectrum merely from its homology. The other direction is also true -- homotopy groups alone cannot recover the homology groups of a spectrum. The caution: because generating family spectra are typically finite spectra, it turns out that their homotopy groups are notoriously difficult to compute. So instead of using homotopy groups, in practice, one often tries to distinguish spectra using auxiliary invariants -- e.g., Steenrod operations on cohomology, and homotopy groups after applying a smashing localization. These are the invariants one hopes can hit a sweet spot between sensitivity (e.g., they are more powerful than homology groups) and computability (e.g., not as difficult to compute as homotopy groups). 
\end{rem}

 \begin{figure}
  \centerline{\includegraphics[height=.7in]{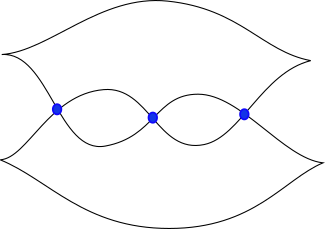}}
  \caption{The fillable max $tb$ Legendrian positive trefoil $\leg_3$, which has a generating family $f$ corresponding to the above ruling with all three crossing switches, ~\cite{f-r}.}
  \label{fig:trefoil}
  \end{figure}

\begin{ex}
\label{example. cubic gf spectrum over B}
Let $f_0 : \RR \to \RR$ be a cubic function with two critical points of distinct critical values. For any smooth compact manifold $B$, let 
	\eqnn
	f: B \times \RR \to \RR,
	\qquad
	(b,t) \mapsto f_0(t)
	\eqnnd
be the $B$-independent cubic. Because $f$ has fiber derivative with constant sign outside a compact subset of $B \times \RR$, $f$ is linear-at-infinity. (See also Example~\ref{example. cubic over a point}.) Then $f$ is a generating family for a Legendrian $\leg \subset J^1B$ diffeomorphic to $B \coprod B$. One may compute its generating family spectrum in two distinct ways.

One may compute by hand the sublevel set of $\delta_{f_0} \leq \epsilon$ for small positive $\epsilon$ -- one finds a homotopy equivalence of pairs
	\eqnn
	(\RR^2 , \{\delta_{f_0} \leq \epsilon\} ) \simeq (D^2, S^1).
	\eqnnd
As a result, we find a homotopy equivalence of pairs
	\eqnn
	(B \times \RR^2, \{\delta_{f} \leq \epsilon\} )
	\simeq
	(B \times D^2, B \times S^1).
	\eqnnd
By definition of generating family spectrum, we thus find 
	\eqnn
	C(\leg, f; \Sphere)
	\simeq
	\Sigma^{-1}  \Sigma^\infty (B \times D^2)/( B \times S^1).
	\eqnnd
(Here the $-1$ is the shift of degree $\dim f = 1$, which is baked into the definition of generating family spectra -- see Definition~\ref{defn:leg-spectrum}.) On the other hand, the quotient $(B \times D^2)/( B \times S^1)$ is homotopy equivalent to the two-fold reduced suspension of $B_+$, where $B_+$ is $B$ with a disjoint basepoint. Thus	
	\eqn\label{eqn. cubic computation}
	C(\leg, f; \Sphere)
	\simeq
	\Sigma^\infty \Sigma^{1} (B_+).
	\eqnd

 The second way to compute the generating family spectrum is to note that $(\leg = B \coprod B, f)$ admits a Lagrangian filling $(L,F)$, compatible with $f$, and with $L$ diffeomorphic to $B \times \RR$. By our spectral Seidel isomorphism (Theorem~\ref{thm:suspension computation}) we conclude
	\eqnn
	C(\leg,f; \Sphere)
	\simeq
	\Sigma^\infty (L_0 / \leg).
	\eqnnd
On the other hand, we have homotopy equivalences of pointed spaces
	\eqnn
	 L_0 / \leg \simeq 
	(B \times [0,1] / B \times \{0,1\})
	\simeq
	\Sigma^1 (B_+).
	\eqnnd
This agrees with our first computation~\eqref{eqn. cubic computation}.
\end{ex}

\begin{rem}
If one chooses a basepoint for $B$, then $\Sigma^1(B_+) \simeq \Sigma^1 B \vee S^1$. Because $\Sigma^\infty$ sends wedge sums of pointed spaces to direct sums of spectra, we thus find that $\Sigma^\infty(\Sigma^1 (B_+)) \simeq \Sigma^\infty B \oplus \Sphere^1$, where $\Sphere^1$ is the one-dimensional sphere spectrum (i.e., the sphere spectrum shifted by 1). Thus, up to an $\Sphere^1$ summand, and by choosing an appropriate $B$, we see that any suspension spectrum of a compact manifold (orientable or otherwise) is realized as the generating family spectrum of some compact Legendrian. 
\end{rem}

\subsection{A next step: More computations}
As mentioned already, spectral invariants are often much more powerful than chain-complex or homological invariants. 
This is because there are many inequivalent spectra with isomorphic homology and cohomology. (For example, the suspension spectra of $\CC P^3$ and of $S^2 \times S^4$ have isomorphic homology and cohomology yet the spectra are inequivalent:   $\Sigma^\infty(\CC P^3)$ has non-trivial Steenrod operations while   $\Sigma^\infty(S^2 \times S^4)$ does not.) In the absence of spectrum-level comparison results (Section~\ref{section. comparisons}), it is thus highly desirable to produce computational techniques for generating family spectra straight from their definition.

As later work will show, for many examples, generating family spectra are highly computable because of the local-to-global (in $B$) properties inherent in the definition, reducing computations to Mayer-Vietoris type arguments for spectra. (As far as we know, this observation provides a new computational technique for generating family homology as well.)

\begin{rem}
In fact, generating family homology can be defined not for a single generating family, but for ordered pairs $(f_0,f_1)$ of generating families (each $f_i$ may potentially generate a different Legendrian). This was exploited in~\cite{traynor:gf-polys,lisa-jill} to produce Legendrian link invariants. (In fact, the generating family homology of pairs was the first generating family homology to be constructed in the Legendrian setting.) The same techniques of this paper also lift generating family homology groups of pairs to the spectral level. Even when $B = \ast$ is a point, it seems any finite stable homotopy type can be constructed as the generating family spectrum of a pair of Legendrians in $J^1(\ast) \cong \RR$ -- in other words, the local classification of pairs of generating families is at least as rich as the classification of finite spectra (e.g., suspension spectra of finite CW complexes and their shifts).
\end{rem}

\subsection{A next step: An $\infty$-category of generating families}
\label{section. comultiplications}
In later work, we will construct a spectrally enriched, non-unital  $\infty$-category whose objects are generating families. (The spectra in the present work are, after taking Spanier-Whitehead duals, endomorphisms in this category.) The composition in this category is constructed roughly as follows: Given $f_0, f_1, f_2: B \times \RR^N \to \RR$, the diagonal embedding of spaces
	\begin{align}
	(B \times \RR^N_2 \times \RR^N_0) \times \RR^N
	& \to
	(B \times \RR^N_2 \times \RR^N_1) \times
	(B \times \RR^N_1 \times \RR^N_0), \label{eqn. coproduct map}
	\end{align}
after an appropriate homotopy, induces a map from the spectrum associated to $(f_0,f_2)$ to the smash of the spectra associated to $(f_1,f_2)$ and $(f_0,f_1)$. Taking Spanier-Whitehead duals, one obtains the composition product. We claim these maps cohere to form an $\infty$-category enriched in spectra. 
In particular, taking $f_0=f_1=f_2=f$, one obtains a not-necessarily-unital, $A_\infty$-algebra structure on the dual of $\gfc(\leg,f;\Sphere)$. 

\begin{rem}
By applying the singular chains functor, one obtains an $A_\infty$-category whose objects are generating families (and by passing to homology over a field, one obtains a category whose objects are generating families).
To the best of our knowledge, this would be the first demonstration of an $A_\infty$-structure on generating family invariants that does not invoke an isomorphism to another invariant. We also expect this product structure to be a spectrum-level lift of the $m_2$ product constructed by Ziva Myer~\cite{ziva} at the level of homology.
\end{rem}

\begin{rem}[Invariants of Legendrians]
The spectrally enriched $\infty$-category of those generating families that generate $\leg$ is a Legendrian isotopy invariant of $\leg$. While it was known that ``the collection of all $f$ and all $\gfc(\leg,f)$'' was an invariant of $\leg$, the compositions in the $\infty$-category give an algebraic structure to this collection. On the other hand, it is rather difficult to use just two pairs $(\leg,f)$ and $(\leg',f')$ to distinguish the Legendrian isotopy type of $\leg$ from that of $\leg'$. 

To ward off discouragement, let us assure the reader that this situation is completely parallel to that of Fukaya categories as a tool for distinguishing Lagrangians (see also Remark~\ref{remark. tangential structures on branes are natural}). If two objects of a Fukaya category have different endomorphisms, one cannot immediately conclude that the underlying Lagrangians of the two objects must not be Hamiltonian isotopic -- it may be that the two objects are simply the same Lagrangian equipped with different brane structures.  
And, it is already rather powerful to be able to distinguish equivalence classes of pairs $(\leg,f)$. (A priori, it is completely non-obvious whether $(\leg',f')$ is Legendrian isotopic to $\leg$ in a way relating $f$ and $f'$ via stabilization and fiberwise diffeomorphism!) 

The upshot is: If one truly wants to use linear-at-infinity generating families to distinguish Legendrian isotopy types, it seems one should understand the full subcategories of generating families consisting of those $f$ generating a single $\leg$, and the full subcategory of those $f'$ generating $\leg'$. By distinguishing two such subcategories, one may conclude $\leg$ and $\leg'$ are not Legendrian isotopic.
\end{rem}

\begin{rem}
\label{remark. grading shift in LCH}
The existence of the above $\infty$-category is a compelling reason to use the ``natural'' grading of $GFH$ we use in this work -- i.e., the $c=1$ grading convention in~\eqref{eqn. gfh is lch}.
Indeed, without the present grading convention, multiplication/composition would not be a degree zero operation.

These grading differences have appeared in previous works. The $A_\infty$ grading conventions in Civan-Etnyre-Koprowski-Sabloff-Walker~\cite{CEKSW} are shifted from many other works dealing with $A_\infty$-structures, as their $m_2$ product is a map of degree 1 (not degree 0). This is because the authors use the grading on $LCH$ -- i.e., the grading on $GFH^{c=0}$ (see Notation~\ref{notation. c =0 grading}). To have a product of degree 0, one must use $c=1$ grading convention for $GFH$.  The $c=1$ grading convention is also used by Myer~\cite{ziva} to construct a product of degree zero, albeit for generating family {\em cohomology}. 

As further motivation for the importance of our convention, we caution that for spectra, one cannot simply ``shift the signs in a formula'' to verify $A_\infty$-relations; instead, one must often exhibit higher and higher homotopies, and an incorrect degree may doom such efforts. Getting the ``correct shift'' from the outset is critical.

\end{rem}

\begin{rem}
We also note the natural appearance of Atiyah duality. The dual to~\eqref{eqn. coproduct map} is naturally a Thom collapse map sensitive to the diagonal embedding of $B$ inside $B \times B$.
\end{rem}

\subsection{Future direction: Comparisons passing through sheaves}
Let us organize the landscape of Legendrian invariants -- wrapped Fukaya categories stopped at $\leg$, the Chekanov-Eliashberg dga of $\leg$, the linearized contact homologies of $\leg$, categories of sheaves with microsupport at infinity contained in $\leg$, and of course, generating family invariants. We will try to summarize what is known and what is, to us at least, unknown.

\begin{rem}[Generating family spectra are computable using sheaves]
It is known that generating family {\em cohomology} (which is isomorphic to the cohomology of the generating family spectra we construct here) is isomorphic to the cohomology of morphism complexes between certain sheaves with prescribed singular supports -- see for example Theorem~8 of the withdrawn work~\cite{shende-withdrawn} (the proof of Theorem~8 is in fact correct, though the author of ibid.~points out that there are other portions of the work which are not correct). 

Though a fully satisfactory six-functor formalism for sheaves of spectra seems not yet in the literature, there is enough written formalism to launch off the ground the microlocal theory of  sheaves with values in spectra. (See~\cite{volpe-six-functors}, and Section 2 of~\cite{jin-treumann}.) In particular, the same proof techniques as in Theorem~8 of~\cite{shende-withdrawn} shows that the Spanier-Whitehead duals of generating family spectra of pairs can be computed as morphisms complexes between their induced sheaves on $B \times \RR$. (In fact, one does not even need the full assortment of microlocal foundations for this computation!) Thus, one has at their disposal both sheaf-theoretic techniques and generating-family techniques for computing generating family spectra. 
\end{rem}

\begin{rem}[Product structures on other invariants]
Given the isomorphism to linearized contact homology~\eqref{eqn. gfh is lch}, it is natural to conjecture that the spectrum-level products for (the Spanier-Whitehead duals to) generating family spectra lift the chain-level products on other Legendrian invariants (such as LCH and microlocal invariants).

As far as we know, an endomorphism algebra for linearized contact homology manifests as an $A_\infty$-algebra first identifiable 
in work of Civan-Etnyre-Koprowski-Sabloff-Walker~\cite{CEKSW} (defined for Legendrians in $\RR^3$) -- see also Remark~\ref{remark. grading shift in LCH}. Generalizing this structure, Bourgeois-Chantraine~\cite{bourgeois-chantraine} constructed an $A_\infty$-category of augmentations and bimodules, where the algebra from~\cite{CEKSW} arises as endomorphisms of a single augmentation -- and this was done for Legendrians in $J^1B$ for any dimension of $B$. A microlocal version was produced, and its subcategory of microstalk-rank-1 objects was conjectured to be equivalent to the augmentation category, in~\cite{stz}. In fact, the conjecture holds for a modified version of the augmentation category, as shown in~\cite{nrssz} when Legendrians are 1-dimensional.  

We note the generating family $\infty$-category has no dimension constraints on $\leg$.
We also note that the generating family $\infty$-category is not restricted to a single Legendrian -- that is, the morphism spectrum of the pair $(f_0,f_1)$ can be defined even when $f_0$ and $f_1$ do not generate the same (isotopy class of) Legendrian (one need only fix the base manifold $B$ of $J^1B$).

We have already mentioned sheaf-theoretic and the augmentation-bimodule categories. We would like to note that this categorical structure is also expected to be captured by Fukaya categories. Indeed, it was asked in Akaho-Joyce's work~\cite[Section 13.5]{akaho-joyce} on immersed Lagrangian Floer cohomology whether their theory can recover Floer theory for Legendrians. Fukaya's $A_\infty$-category of immersed Lagrangians~\cite{fukaya-immersed} (which, according to personal communications with Fukaya, can actually be constructed immediately from~\cite{akaho-joyce} if one only considers a full subcategory with {\em finitely} many objects) when applied to immersed Lagrangians in $T^*B$, is thus expected to be a Fukaya-categorical counterpart.
\end{rem}

\begin{rem}[Conjectures about spectral lifts of the Chekanov-Eliashberg dga and other Fukaya-categorical invariants] \label{rem:spectral lift conjectures} 
Work of Ekholm-Lekili~\cite{ekholm-lekili} has shown that, over chains on the based loop space of $\leg$, Koszul duality for $A_\infty$-algebras and coalgebras relate the Chekanov-Eliashberg dga $CE^*(\leg)$ to an $A_\infty$-coalgebra structure on a kind of linearized contact homology -- this coalgebra is denoted $LC_* = LC_*(\leg,\epsilon)$ in their work. While ibid. constructs $LC_*$ a posteriori from an augmentation $\epsilon$ on the Legendrian Chekanov-Eliashberg dga, the generating family spectrum theory seems to instead take (a spectral, generating-family version of) $LC_*$ as the starting point. Instead of an action of the based loop space $\Omega \leg$ on the Chekanov-Eliashberg algebra, there is naturally a comodule action of the coalgebra $\Sigma^\infty \leg$ on the generating family spectrum $C(\leg,f;\Sphere)$. 

This conjecturally gives two frameworks for trying to define a generating-family, spectral analogue of the Chekanov-Eliashberg dga for $\leg$. One method is to construct the Koszul dual to the generating family spectrum of $(\leg,f)$. One would hope that the Morse filtration of a difference function (or, a filtration of the spectrum by Reeb chord lengths) will allow one to recover a version of the Chekanov-Eliashberg dga as an $A_\infty$-algebra completed with respect to a filtration by quantitative invariants (such as Reeb chord length), and spectrally so. Likewise, the comodule action of $\Sigma^\infty \leg$ would be Koszul dual to a module action from (a completed version of) $\Sigma^\infty \Omega \leg$.

It is yet unproved that these  Koszul dual algebras for a single $\leg$ should all be equivalent regardless of $f$. However, as suggested by an anonymous referee, one may imagine that generating families for $\Lambda$ are like twists of bounding cochains for a curved coalgebra (whose Koszul dual is the  Chekanov-Eliashberg dga). Then indeed the Koszul duals of the linearized coalgebras should all give rise to the same algebra.

The other framework is to consider the generating family category as a model for an infinitesimally wrapped Fukaya category, and localize with respect to positive wrappings. At a naive level, the Chekanov-Eliashberg algebra is  similar to the kinds of bar constructions familiar from localizations, where words of high length contain letters representing the morphisms one seeks to invert -- in this case, positive Reeb chords, which one might think of as a composition of morphisms in the infinitesimal category arising from many positive wrappings. Indeed, it is worth noting that quadratic-near-infinity generating families can lift the kinds of non-compact branes one finds in the theory of infinitesimal Fukaya categories of cotangent bundles, and generating family spectra seem to compute the correct (infinitesimal) morphisms between such objects. 
\end{rem}

\subsection{Future direction: Finite-dimensional approximations and comparison with Floer theory}
\label{section. comparisons}
\label{section. finite dim approx}
The present work's generating family spectrum is the first to explicitly encode a spectral lift for generating-family invariants for Legendrians. While Floer homotopy theory has yet to produce a spectral lift of linearized contact homology, nor of infinitesimally wrapped Fukaya categories --  so we cannot yet produce a rigorous comparison -- it is regardless highly desirable to have a strategy of how one {\em would} compare generating family spectra to Floer-theoretic invariants. We share here a speculative analogy that may give some insight.

Given a generating family, one uses difference functions to naturally associate a pointed space to serve as an invariant of the generating family. Spectra are forced upon us when we try to prove that these spaces are invariant under {\em Legendrian} isotopies. Explicitly: When a Legendrian is isotoped, the associated generating families may naturally acquire a higher-dimensional domain (Proposition~\ref{prop:leg-persist}), and this accounts for the appearance of suspensions of spaces (Proposition~\ref{prop:leg-stab}),  whence spectra emerge. 

That invariance necessitates suspension seems to be a motif. Spectra emerge with an inevitable air for the same reason in Seiberg-Witten/Bauer-Furuta invariants. Thus, in the same way that finite-dimensional approximations of PDEs resulted in Bauer-Furuta's stable homotopy lifts, one may wonder the extent to which difference functions lead to finite-dimensional approximations to holomorphic curve equations in symplectizations.
Indeed, a version of finite-dimensional approximation seems to have been the original motivation for Viterbo's work probing Hamiltonian dynamics using generating families~\cite{viterbo-1992-symplectic-topology-as-generating-functions}.

One might hope for such a philosophy -- that gradients of difference functions yield finite-dimensional approximations to holomorphic curve equations -- to guide future work. However, in Seiberg-Witten theory (where the analysis is expected to be considerably simpler), the comparison between finite-dimensional approximation invariants (which are spectral in nature) and the original Seiberg-Witten Floer invariants (which, at present, are homological) involves highly non-trivial arguments as in the work of Lidman-Manolescu~\cite{lidman-manolescu}. 
 
A direct comparison between generating family invariants and  holomorphic-curve invariants has only been attained in a few limited examples -- 
for example, Viterbo~\cite{viterbo-functors-applications-ii} exhibits an equivalence with the Floer cohomology of two Lagrangians that are both Hamiltonian isotopic to the zero section.

\subsection{Context and history}
\label{section. history}
Let us first motivate the problem of lifting chain complex and homological invariants to the level of stable homotopy types (i.e., spectra). 
In homotopy theory, the stable homotopy type of a space is a refinement of the singular chain complex of the space.
In low-dimensional topology, 
Lipshitz and Sarkar~\cite{L-S:homotopy-type} showed
that the chain complex underlying Khovanov homology has a refinement to a CW spectrum.  These
spectra have homology groups that recover Khovanov homology 
and contain more information: 
By studying Steenrod operations, Seed \cite{Seed} found examples of smooth knots with the same Khovanov homology
but nonequivalent spectra. 

In the realm of geometric invariants constructed from Floer theory, Floer noted in their original works~\cite{floer:witten} that stable homotopy refinements of Floer homology groups should be present. Since then, many stable homotopy types lifting Floer-type homology groups have been constructed. In symplectic geometry,
Cohen-Jones-Segal outlined an approach using flow categories~\cite{cohen-jones-segal},
Kragh lifted both symplectic homology and Viterbo's transfer maps~\cite{Kragh:Viterbo-transfer},
and Abouzaid-Blumberg have lifted Hamiltonian Floer homology~\cite{abouzaid-blumberg}. 
In Seiberg-Witten theory, one has 
Bauer-Furuta's invariants~\cite{bauer-furuta} and Manolescu's equivariant stable homotopy type~\cite{manolescu-pin2}. 

More recent foundational works in this direction include the work of Large~\cite{large-thesis} and the work of Abouzaid-Blumberg~\cite{abouzaid-blumberg-foundations} proving a categorification of the Pontrjagin-Thom isomorphism -- so that the kinds of flow categories appearing in Floer theory tautologically give rise to stable homotopy types.    The work of Porcelli-Smith  build on this idea to construct ``Donaldson'' type Fukaya categories over spectra~\cite{porcelli-smith-bordism,porcelli-smith-tangential}. (By ``Donaldson" type we mean that one sidesteps the higher coherences required for an $\infty$-category enriched in spectra, and in fact only remembers the homotopy groups.) We refer also to the work of  Blakey~\cite{blakey}, Porcelli~\cite{porcelli-selecta}, and Hirschi-Porcelli~\cite{hirschi-porcelli} for applications. Importantly, all the works in this paragraph attempt to extract stable homotopy types directly from moduli spaces of pseudoholomorphic objects -- in particular, the categorical coherence, the tangential obstructions, and the analytical difficulties are far more intricate than when dealing with generating families.

Below, we survey some uses of  generating families to tackle problems in symplectic topology. We thank the referee for suggesting that we exposit these works. 
We have also been informed that forthcoming work of Lazarev is expected to produce spectral lifts of wrapped  Floer cohomology in a cotangent bundle, via generating families.

\subsubsection{Bounds on intersection numbers}
\label{section. intersection number bounds by counting crit points of gfs}
A now-classical problem of symplectic geometry is to establish lower bounds on the number of intersection points between the zero section $M \subset T^*M$ and its image under a Hamiltonian isotopy. Cup-length lower bounds were established by Hofer~\cite{hofer-lagrangian-embeddings-AIHP-1985} (using variational methods) and by Laudenbach-Sikorav~\cite{laudenbach-sikorav-inventiones-1985} (using broken geodesic methods). Generating family finally appeared in later work of Sikorav~\cite{sikorav:gen-fn} -- once one knows that Lagrangians Hamiltonian isotopic to the zero section admit generating families, a simple count of critical points of a generating family yields the desired bounds.  Lalonde-Sikorav~\cite[Theorem 3(i)]{lalonde-sikorav} built upon this insight to yield bounds on (and in particular, the non-zero-ness of) the number of intersection points between any $L$ Hamiltonian isotopic to the zero section, and the conormal $T^*_K M$ to a submanifold $K \subset M$.

\subsubsection{Works of Viterbo}
In~\cite{viterbo-functors-applications-ii}, it is shown that the cohomology of sublevel sets of a difference function recovers Floer cohomology (for pairs of Lagrangians Hamiltonian-isotopic to the zero section).\footnote{We note there is a minor error, in that the cohomology of the sublevel sets must be shifted to match (choices of) gradings on the Lagrangians; this is roughly the same reason for the $N$-fold shift in Definition~\ref{defn:leg-spectrum} below.} In contrast to the works cited in Section~\ref{section. intersection number bounds by counting crit points of gfs}, this work made use of the underlying topology of sublevel sets arising from generating families. A powerful outcome is that this isomorphism holds as objects filtered by action, but little is mentioned about the (stable) homotopy type defined by the sublevel sets.

In Section~2 of~\cite{viterbo-1992-symplectic-topology-as-generating-functions}, the Conley index for a quadratic-at-infinity generating family is seen as computing a Thom space up to homotopy equivalence. And in~\cite{viterbo-jdg}, the importance of the suspension-invariance of the Conley indices related to sublevel spaces is obviously recognized 
(see \cite[Proposition~4.1]{viterbo-jdg}  and the commentary surrounding it); so the seeds for defining a stable homotopy type in this way were planted at least as early as this work, though no stable homotopy type computations (and no mention of stable homotopy types) appeared. It is our understanding that the first explicit use of spectra using this method first appeared Kragh's work~\cite{Kragh:Viterbo-transfer}, following inspiration from the work of Viterbo.

\subsubsection{Arnold conjecture for tori}
It was noticed early on that (the topological invariants of) flows induced by generating families contain symplectic and dynamical information.   For example, the proof of the Arnold Conjecture for the $2n$-dimensional torus can be given by  first writing a Hamiltonian symplectomorphism as a
composition of  a finite number of tiny symplectomorphisms admitting generating families. 
Then, studying the number of fix points of the original symplectomorphism amounts to counting critical points of a Morse function constructed out of a sequence of difference functions associated to the generating families (this Morse function also goes by the name of the discrete action functional). The number of critical points of a Morse function is bounded by the Betti numbers of (the relative homology groups of) a Conley index pair. Moreover, it turns out to be easy to see that the Conley pair is the smash product of $T^n \coprod pt$ with a sphere. For details see the proof of  Theorem~11.1.9 in~\cite{mcduff-salamon-intro}. Of course, in this situation there is little to be gained from the stable homotopy type (as opposed to the co/homology) -- after all, the suspension of $T^n$ is homotopy equivalent to a wedge of spheres by Whitehead's theorem for simply-connected CW complexes.

\subsubsection{Relations to capacities}
Work of the second author~\cite{traynor:shomology} constructed an analogue of symplectic homology (together with its filtration by capacities) using generating families (and the homologies of certain super/sub level sets). Later, Sheila Sandon~\cite{sandon-capacity-annales-fourier-2011} extended the construction to the contact setting. In fact, capacity-like invariants from filtered generating family homology were utilized by Mohan Bhupal~\cite{bhupal-partial-order-2001} (taking inspiration from Viterbo's~\cite{viterbo-1992-symplectic-topology-as-generating-functions}) to define a partial order on the group of compactly supported contactomorphisms of Euclidean space.

\subsubsection{Tube generation}
In~\cite{abouzaid-courte-guillermou-kragh}, a homological criterion for a generating family locally (in $B$) tube-generating ``half the double'' of the zero section is expressed through the homology of the sublevel set of a difference function. To be more specific, in ibid., a function $D$\footnote{Up to fiber diffeomorphism, this is the generating family from Example~\ref{example. cubic gf spectrum over B}.} is used to model the attachment of one trivial handle to affine half-space. (See the start of Section~3.4 of ibid.) By definition, their notion of tube type and tube generation systematically concentrates on the single handle, and not its necessarily cancelling partner\footnote{Any linear-at-infinity generating family has sublevel spaces that begin as a half-space and end as a half-space, so the only two handles must cancel.}.

The homological criterion we have referred to is Proposition~3.25 of ibid., which detects when a stabilization of a function is of tube type. Crucially, the authors pass to the difference function of $f$ to compute the homology of $f$ (specifically, they combine their Lemma~3.22 and Lemma~3.28 -- the former of which computes a particular sublevel set homology of the difference function in terms of that of $f$ under a mild assumption on the critical values of $f$). This is what allows the authors to conclude their main tube generation result (Theorem D of ibid.). The stable homotopy types of sublevel sets of difference functions are not utilized.\footnote{However, we do point out that Lemma~3.22 holds at the level of spectra by replacing $C_*$ by the stable homotopy type and $\otimes_\ZZ$ by the smash product of spectra.} We also point out that the sublevel set computed in Lemma~3.22 is particularly simple; the difficulty of generating family homology computations lies in computing the sublevel set $\delta_f^{\leq a}$ for very small positive values of $a$, which in particular precludes the helpfulness of the kinds of critical-value bounds in the hypothesis of Lemma~3.22.

\subsection{Acknowledgements}
 This material is based upon work supported by the National Science Foundation under Grant No. 1440140, while the authors were in residence at SLMath in Berkeley, California, during the Floer Homotopy Theory program in Fall 2022. HLT was supported by a Sloan Research Fellowship, an NSF CAREER grant DMS-2044557, and the Texas State University Presidential Seminar Award and Valero Award. 
We thank
Mohammed Abouzaid,
Daniel \'Alvarez-Gavela,
Denis Auroux,
Roger Casals,
Thomas Kragh,
Oleg Lazarev,
Robert Lipshitz,
Lenny Ng, 
Dan Rutherford,
Josh Sabloff,
Paul Seidel,
and
Claude Viterbo
for helpful communication regarding this work.  The authors also thank the anonymous referee for exceptionally useful comments, suggestions, and corrections.

\section{Generating families and difference functions} \label{sec:background}

\subsection{Generating family background} \label{ssec:gf-background}
We recommend~\cite{theret:viterbo, traynor:gf-polys, viterbo-1992-symplectic-topology-as-generating-functions, S-T:obstruct} 
for further reading.  
Let $B$ be a smooth manifold (not necessarily compact). Given a smooth function $f\co B \to \rr$,
the graph of $df$ in $T^*B$ is a Lagrangian submanifold, and the $1$-jet of $f$  in $J^1B$ is a Legendrian submanifold.  Generating families can further produce ``non-graphical'' Legendrian submanifolds by expanding the domain of the function to, for example, the trivial vector bundle $B \times \rr^N$ for some potentially large $N$.

\begin{notation}[$\eta$]
\label{notation. eta}
We will denote the fiber coordinates (i.e., the coordinates of $\RR^N$) by $\e = (\e_1, \ldots, \e_N)$. 
\end{notation}
 
\begin{assume}[Genericity of $f$]
\label{assumption. f generic}
Throughout this section, $f$ denotes a smooth function 
	\eqnn
	f\colon B \times \rr^N \to \rr
	\eqnnd
such that $\mathbf{0}$ is a regular value of the map $\partial_{\e} f\co  B \times \rr^N \to \rr^N$. 
\end{assume} 

A generic $f$ yields a Legendrian as follows.
 The graph of $df$, $\Gamma_{df}$ is an embedded Lagrangian submanifold of $T^*(B \times \rr^N)$. A coisotropic reduction, as described in \cite[Section 5.4]{mcduff-salamon}, gives
rise to an immersed, exact Lagrangian $L$ in $T^*B$, which lifts to an immersed  Legendrian $\leg$ in $J^1B$.

Alternatively, 
using the perspective of Weinstein's category, 
\cite{Weinstein-cat}, one can view $\Gamma_{df}$ as a Lagrangian 
correspondence (also known as a canonical relation) between $T^*B$ and $ T^*\rr^N$.   The zero-section of $T^*\rr^N$ can also be viewed as a Lagrangian correspondence
between $T^*\rr^N$ and  the trivial space $pt$.
	\eqnn
	\xymatrix{
		& \Gamma_{df} \ar[dr] \ar[dl]  && \RR^N \ar[dr] \ar[dl] \\
	T^*B && T^*\RR^N && pt
	}
	\eqnnd
It is well known that a  composition of a Lagrangian (Legendrian) correspondence with a transverse Lagrangian correspondence yields another Lagrangian (Legendrian). Thus, the fiber products $\Sigma_f$ below
	\eqnn
	\xymatrix{
		&& \Sigma_f \ar[dr] \ar[dl] \\
		& \Gamma_{df} \ar[dr] \ar[dl]  && \RR^N \ar[dr] \ar[dl] \\
	T^*B && T^*\RR^N && pt
	}
	\eqnnd
	\eqnn
	\xymatrix{
		&& \Sigma_f \ar[dr] \ar[dl] \\
		& \Gamma_{df} \ar[dr] \ar[dl]  && \RR^N \ar[dr] \ar[dl] \\
	J^1B && T^*\RR^N && pt
	}
	\eqnnd
are (diffeomorphic) smooth manifolds; we call $\Sigma_f$ the \dfn{fiber critical set}. Via the diffeomorphism of $\Gamma_{df}$ with $B \times \RR^N$, and the fact that $\Sigma_f \to \Gamma_{df}$ is necessarily an embedding, we may naturally identify $\Sigma_f$ with the subset
	\eqn\label{eqn. sigma f as a subset}
	\{
	(q,\eta) \, | \, \del_\eta f(q,\eta) = 0
	\} \subset B \times \RR^N.
	\eqnd
The maps from $\Sigma_f$ to $T^*B$ and $J^1B$ define an immersed Lagrangian and an immersed Legendrian, respectively; these immersions have formulas
\eqn\label{eqn. immersions of sigma f}
  (q,\e) \mapsto (q,\partial_qf(q,\e)), \qquad
  (q,\e) \mapsto (q, \partial_q f(q,\e), f(q,\e)).
\eqnd 
We let $L$ denote the immersed Lagrangian and $\leg$ denote the immersed Legendrian.
We say
that $f$ \dfn{generates} $L$ and $\leg$, or that $f$ is a
\dfn{generating family (of functions)} for $L$ and $\leg$.

\begin{rem} Sometimes the term ``generating function" is used instead of ``generating family".  Both terms are a shortening of the longer phrase ``generating family of functions."  Due to the
 common use of ``generating function" in physics and combinatorics, some authors prefer the term ``generating family'' in the present context, to avoid confusion when communicating beyond symplectic and contact topology.
\end{rem}

\subsection{Equivalent generating families}
\label{section. equivalent gfs}

Given a generating family $f$ for a Lagrangian/Legendrian, the following two operations produce more generating families. These operations generate an equivalence relation on the collection of generating families for a fixed Lagrangian/Legendrian.

\begin{defn}  \label{defn:gf-equiv}
Fix a smooth function $f: B \times \RR^N \to \RR$.
\begin{enumerate} 
\item 
	A {\bf rank $K$
  stabilization} of $f$ of \dfn{index} $i$ is a function
  \eqnn
  	f \oplus Q\co B \times \rr^N \times \rr^K \to \rr,
	\qquad
	(q,
  \e, \e') \mapsto f(q, \e) + Q(\e'),
  \eqnnd
  where  $Q\co
  \rr^K \to \rr$ is a non-degenerate quadratic form of index $i$.
\item
	A \dfn{fiber-preserving diffeomorphism} is a diffeomorphism
  	\eqnn
	\Phi\co B \times \rr^N \to B \times \rr^N , 
	\qquad
	(q,\eta) \mapsto (q, \phi_q(\e))
	\eqnnd
	for some smooth family of 
  diffeomorphisms $\phi_q$. Then $f \circ \Phi$ is said to be
  obtained from $f$ by fiber-preserving diffeomorphism.
\end{enumerate}
 \end{defn}
 
\begin{rem} 
If $f$ generates a Lagrangian, the addition of a constant to $f$
will not change the Lagrangian generated but will change the Legendrian generated. 
 We will later be considering generating families for Lagrangian fillings  that are an ``extension" of the generating family
on the cylindrical end formed from the Legendrian, and so this addition of a constant will not arise.   See Definition~\ref{defn:compatible}. 
\end{rem}

\subsection{Difference functions} \label{ssec:diff}

\begin{defn} \label{defn:leg-diff}
Suppose that $f\co B \times \rr^N \to \rr$ is a generating family for a Legendrian $\leg \subset J^1B$.  The {\bf difference function} of $f$ is the function 
	\eqnn
	\delta_f \co B \times \rr^N \times \rr^N \to \rr,
	\qquad
	 \delta_f(q,\e,\te) := f(q,\te) - f(q,\e). 
	 \eqnnd
\end{defn}

\begin{rem}
More generally, given generating families $f_0, f_1$ that generate Legendrians $\leg_0,\leg_1 \subset J^1B$, one may assume that $f_0$ and $f_1$ have the same domain (by stabilizing if necessary). Then we may define a difference function for this pair as follows:
	\eqnn
	\delta_{f_0,f_1} \co B \times \rr^N \times \rr^N \to \rr,
	\qquad
	 \delta_{f_0,f_1}(q,\e,\te) := f_1(q,\te) - f_0(q,\e). 
	\eqnnd
Such a difference function was already used in~\cite{traynor:gf-polys,lisa-jill} to detect Legendrian linking phenomena between $\Lambda_0$ and $\Lambda_1$.
\end{rem}

Difference functions are at the core of our invariants, so we take some time to explicate their properties.

 \begin{rem} \label{rem:stab-choice-difference}
Let $f_\pm: B \times \rr^N \times \rr \to \rr$ be
 a stabilization of $f$ by  a non-degenerate quadratic form $Q_\pm(\e') = \pm (\e')^2$. 
Even though the indices of $Q_+$ and $Q_-$ differ, the associated difference
functions $\delta_{f_{\pm}} \co B \times \rr^{2N +2} \to \rr$ both differ from $\delta_f \co B \times \rr^{2N} \to \rr$ by a stabilization by a non-degenerate quadratic form of index $1$:
$$ \begin{aligned}
\delta_{f_+}(x, \e, \e', \te, \te') &=
( f(x, \te) + (\te')^2) -  (f(x, \e) + (\e')^2)\\ 
&= 
 \delta_f(x, \e, \te) + (\te')^2  - (\e')^2, \\
 \delta_{f_-}(x, \e, \e', \te, \te') &= 
 ( f(x, \te) - (\te')^2) -  (f(x, \e) - (\e')^2) \\
 &=
 \delta_f(x, \e, \te) - (\te')^2  +(\e') ^2.
 \end{aligned}$$
Moreover, $\delta_+$ and $\delta_-$ are related by a fiber-preserving diffeomorphism that swaps $\te$ and $\e$.
 \end{rem}

In the standard contact structure on the jet bundle $J^1B$, {\bf Reeb chords} 
	 \eqnn
	 \gamma\co [a,b] \to J^1B
	 \eqnnd
of a Legendrian $\leg$ are   trajectories of $\frac{\partial}{\partial z}$ whose endpoints lie on $\leg$.  Under the projection of the 
Legendrian $\leg \subset J^1B$ generated by $f$  to the immersed Lagrangian  $L \subset T^*B$ generated by $f$ the Reeb chords of $\leg$  are in one-to-one correspondence with double points of $L$.

\begin{notation}[$\ell$] \label{notation:length} 
Given a Reeb chord $\gamma: [a,b] \to J^1B$ of an embedded Legendrian, 
we let $\ell(\gamma) > 0$ be its \dfn{length} -- that is, the integral $\int_{[a,b]} \gamma^*\theta$ where $\theta = dz-ydx$ is the standard contact 1-form on $J^1B$.
\end{notation}

The following shows that the critical locus of the difference function is sensitive to the topology and some of the Reeb dynamics of $\leg$. See \cite[Lemma 3.3]{josh-lisa:cap} and~\cite{f-r}.
  
\begin{prop}
  \label{prop:leg-crit-point} 
  Suppose $f$ is a generating family for an embedded Legendrian $\leg \subset J^1B$.  Then 
  \begin{enumerate}  
  \item The critical locus with $\delta_f \equiv 0$ is the locus 	
  	$$\left\{ (q,\e,\e) \,:\, (q,\e) \in \Sigma_f\right\}$$
	and hence is naturally diffeomorphic to the critical submanifold $\Sigma_f$ (and thus is diffeomorphic to $\leg$).
  \item The critical locus with $\delta_f  \neq 0$ is identified with a 2-to-1 cover of the set of Reeb chords of $\leg$.
  Specifically, for each Reeb chord $\gamma$ of $\leg$, there are two
    critical points $(q,\e,\te)$ and $(q,\te,\e)$ of $\delta_f$ with
    nonzero critical values $\pm \ell(\gamma)$.
\end{enumerate}
\end{prop}

 It is convenient to consider the  {\bf length spectrum} of a Legendrian submanifold $\leg$, defined as 
 $$\ell(\leg) := \{\ell(\gamma)\mid  \gamma \text{ a Reeb chord of } \leg\}.$$
 
 For later discussions, it will be useful to keep in mind the following lemma, which tells us that for a $1$-parameter family of Legendrians, the length spectra
 will be uniformly bounded away from $0$.

 \begin{lem} \label{lem:uniform-non0} If $\leg_t$, $t \in [0,1]$, is a 1-parameter family of compact, embedded Legendrian submanifolds in $J^1B$, then there exists an $\epsilon > 0$ such that 
 $$\ell(\leg_t) \cap (0,\epsilon) = \emptyset, \quad \text{ for all } t \in [0,1].$$ 
 \end{lem}

\begin{proof}
Let $X = \Lambda_0$ and fix a smooth map $\lambda: X \times[0,1] \to J^1B$ such that for all $t \in [0,1]$, the map $\lambda_t: X \cong X \times \{t\} \to J^1B$ is a Legendrian embedding with image $\Lambda_t$. A standard exercise in symplectic geometry shows that the induced map $X \xrightarrow{\lambda_t} J^1B \xrightarrow{\pi} T^*B$ is an immersion for all $t$. On the other hand, every immersion is locally an embedding. Hence, for every $(x,t) \in X \times [0,1]$ there exists an open neighborhood $U \subset X$ of $x$ and an open interval $I \subset [0,1]$ containing $t$ so that, for every $t' \in I$, the composition
	\eqnn
	U \subset X \cong X \times \{t'\} \subset X \times [0,1] \xrightarrow{\lambda} J^1 B \xrightarrow{\pi} T^*B
	\eqnnd
is a smooth embedding. Thus, for every $t$ we have produced an open cover $\mathcal{U}_t$ of $X$ such that $\pi \circ \lambda_t$ is an injection along each $U \in \mathcal{U}_t$. By compactness (and refining $\mathcal{U}_t$ if necessary) we may choose the cover $\mathcal{U}$ to be independent of $t$ (while still satisfying the property that $U \in \mathcal{U}$ implies $(\pi \circ \lambda_t)|_U$ is an injection, for all $t$).
 On the other hand, the Reeb chords (possibly constant, possibly backward) of $\Lambda_{t}$  correspond exactly to pairs $x',x'' \in X$ having equal image in $T^*B$ under $\pi\circ\lambda_{t}$. So we have covered $X$ by open sets $U$ such that, for all $t$, no $U$ contains the endpoints of a non-constant Reeb chord of $\Lambda_t$.
 
In $X \times X \times [0,1] $, consider the closed (hence compact) subspace 
	\eqnn
	T:= \{  (x,x',t) \, | \, \pi \circ \lambda_t(x) = \pi \circ \lambda_t(x')\}.
	\eqnnd 
Letting $\pi_z: J^1B \to \RR$ denote the projection to the $z$ coordinate of the jet bundle, we see that the function 
	\eqnn
	\ell: 
	T \to \RR,
	\qquad
	(x,x',t) \mapsto 
	\pi_z(\lambda_t (x)) - \pi_z(\lambda_t (x'))
	\eqnnd 
equals $0$ only along the diagonal points -- i.e., along those $(x,x',t)$ for which $x=x'$; here, we have used that each $\lambda_t$ is an embedding. Note 
	\eqnn
	T^{\geq 0} := \ell^{-1}( [0,\infty))
	\eqnnd 
is a closed (hence compact) subspace. One identifies $T^{\geq 0}$ with the space of pairs $(\gamma,t)$ where $\gamma$ is a  (possibly constant, but not backward) Reeb chord with endpoints on $\Lambda_t$. Now consider the closed (hence compact) subspace
	\eqnn
	T^+ := T^{\geq 0} \bigcap \left( X \times X \times [0,1] \setminus
	\bigcup_{U \in \mathcal{U}} U \times U \times [0,1]
	\right).
	\eqnnd
By design, $U \times U \times [0,1]$ does not contain (the end points of) any non-constant Reeb chord, while of course the union $\bigcup_{U \in \mathcal{U}} U \times U \times [0,1]$ contains the entire diagonal. So the above intersection $T^+$ is identified with the space of pairs $(\gamma,t)$ where $\gamma$ is a (non-constant, non-backward) Reeb chord with endpoints on $\Lambda_t$. By the extreme value theorem $\ell$ must attain a minimum on $T^+$, but because $T^+$ does not intersect $\ell^{-1}(0)$, this minimum $\mu$ must be a positive real number. Choosing $\epsilon$ to be any real number in the interval $(0,\mu)$, the result follows.
\end{proof}

\subsection{Linearity-at-infinity} 
A generating family is defined on the non-compact space $B \times \rr^N$. The family's behavior outside a compact set must be sufficiently well-behaved in order to apply the Morse-theoretic lemmas mentioned in 
Section~\ref{ssec:morse-theory}. So henceforth in this paper, we will assume that generating families
 for Legendrian submanifolds satisfy a ``linear-at-infinity" condition, similar to that used in, for example, 
 \cite{f-r,S-T:obstruct}.

Recall the following classical definition:

\begin{defn}[Classical]
\label{defn. classical linear-at-infinity}
A function $f: B \times \RR^N \to \RR$ is called linear-at-infinity if there exists 
a non-zero linear functional
$A: \RR^N \to \rr$ and constant $c \in \rr$ such that, outside a compact subset of $B \times \rr^N$, $f$ takes the form 
	\eqnn
	(x,\eta) \mapsto A\eta + c.
	\eqnnd
\end{defn}

Definition~\ref{defn. classical linear-at-infinity} is not preserved by the two natural notions of equivalence for generating families: stabilization and $B$-parametrized diffeomorphisms of $\RR^N$. Some authors have often thought of linear-at-infinity to mean that after a fiber-preserving diffeomorphism the generating family takes on the classical form.  
To make this thought process more transparent, we introduce the following slight generalization, 
using the same terminology.

\begin{defn}[For this paper]
\label{defn. linear-at-infinity}
A smooth function $f: B \times \RR^N \to \RR$ is called {\em linear-at-infinity} if there exists a diffeomorphism $\phi: B \times \RR_t \times \RR^{N-1} \to B \times \RR^N$ such that
\begin{enumerate}
	\item $\phi$ respects the projection to $B$ (meaning $\phi$ is a $B$-parametrized family of diffeomorphisms from $\RR_t \times \RR^{N-1}$ to $\RR^N$), and
	\item outside a compact set, $f \circ \phi = t$, where $t$ is the projection $B \times \RR_t \times \RR^{N-1} \to \RR_t$.
\end{enumerate}
\end{defn}

\begin{rem}\label{remark. gradient flow complete for linear-at-infinity function}
Note that if $f: B \times \RR_t \times \RR^{N-1} \to \rr$ is equal to $t$ outside a compact subset, then there necessarily exists a Riemannian metric on $B \times \RR_t \times \RR^{N-1}$ for which the gradient flow of $f$ is complete. 
\end{rem}

\begin{rem}
\label{remark. dim f at least 1}
Definitions~\ref{defn. classical linear-at-infinity}  and~\ref{defn. linear-at-infinity} both imply that $N$ is at least 1. 
\end{rem}

\begin{rem}
Definition~\ref{defn. linear-at-infinity} also codifies the utility of a generating family being linear-at-infinity: One can parametrize $B \times \RR^N$ so that the dynamics of the generating family (outside a compact set) is simply translation along some direction in $\RR^N$.
\end{rem}

\begin{prop}
\label{prop. linear at infinity defns are equiv}
If $f$ is linear-at-infinity in the classical sense (Definition~\ref{defn. classical linear-at-infinity}), then it is linear-at-infinity in the sense of Definition~\ref{defn. linear-at-infinity}.
Conversely, if $f$ is linear-at-infinity in the sense of Definition~\ref{defn. linear-at-infinity}, there exists a fiberwise diffeomorphism $B \times \RR^N \to B \times \RR^N$ transforming $f$ to a function that is linear-at-infinity in the classical sense.
\end{prop}
 
\begin{proof} 
First assume that we have a generating family $f: B \times \rr^N \to \rr$ that is linear-at-infinity in the classical sense: outside
a compact set $f(x, \eta) = A \eta + c$. We will construct the desired $\phi: B \times \rr_t \times \rr^{N-1} \to B \times \rr^N$ so that,
outside of a compact set,
$f \circ \phi (x, t, \tilde \eta) = t$.
To construct $\phi$, first
observe that there is a special orthogonal transformation $M$ of $\rr_t \times \rr^{N-1}$ to  $\rr^N$, 
that maps $\rr^{N-1}$ to $\ker A$ and $\rr_t \times \{0\}$ to the $1$-dimensional vector subspace perpendicular
 to the vector subspace $\ker A$.  Applying this linear
map in each fiber gives rise to a diffeomorphism
$$\phi_{M}:  B \times \rr_t \times \rr^{N-1} \to B \times \rr^N,$$
such that, for all $x \in B$,  
$$\phi_M\left(\{x\} \times \{ t = b \}  \times \rr^{N-1} \right) = \left(\{x\} \times \{ A = b\}\right).  $$
Now at each $x \in B$,  in each $\rr^N$ we can perform translation in the direction perpendicular to $\ker A$, which defines 
 a diffeomorphism
$$\phi_{-c}: B \times \rr^N \to B \times \rr^N,$$
such that, for all $x \in B$, 
$$\phi_{-c}\left( \{x\} \times  \{A = b\}  \right) = \left( \{x \} \times \{A = b - c\}\right).$$
By construction, for $\phi = \phi_{-c} \circ \phi_M$, 
$$f \circ \phi \left( \{x\} \times \{t = b \} \times \rr^{N-1} \right) = f\left(  \{x \} \times \{A = b - c\} \right) = (b-c) + c = b,$$
 thus showing that, outside of a compact set, 
$$f \circ \phi(x, t, \tilde \eta) = t,$$
 as desired.

The converse is immediate from the definitions.
\end{proof}

\begin{prop}
\label{prop. stabilization preserves linearity}
A stabilization of a linear-at-infinity function is linear-at-infinity.
\end{prop}

\begin{proof}
Before stabilizing $f$, apply the diffeomorphism guaranteed by Definition~\ref{defn. linear-at-infinity} so that, outside a compact set, $f$ equals the function
	\eqnn
	 B \times \RR_t \times \RR^{N-1}  \to \RR,
	\qquad
	(x,t,u)
	\mapsto 
	t.
	\eqnnd
 The remainder of the proof follows from the  proof of \cite[Lemma 3.8]{S-T:obstruct}.
\end{proof}

\subsection{Paths of generating families}
\label{section. paths of gfs}

\begin{defn}
A smooth function $f: [0,1] \times B \times \RR^N \to \RR$  is called a {\em smooth path of generating families} if for every $t \in [0,1]$, $f_t = f(t,-,-): B \times \RR^N \to \RR$ is a generating family for some ($t$-dependent) compact, embedded Legendrian in $J^1 B$. (That is, $f_t$ satisfies Assumption~\ref{assumption. f generic} and the Legendrian immersion from~\eqref{eqn. immersions of sigma f} is an injection.) 
\end{defn}

\begin{rem}
Given a path of generating families, let $\widetilde{\Sigma}$ denote the fiber critical set -- that is, the subset of $[0,1] \times B \times \RR^N$ along which $\del_\eta = 0$ (Notation~\ref{notation. eta}). 
Then the induced map $\widetilde{\Sigma} \to [0,1]$ is automatically a submersion. Indeed, by Assumption~\ref{assumption. f generic}, the tangent space of $\widetilde{\Sigma}$ at a point $(t,b,\eta)$ is the ($(1+\dim B)$-dimensional) kernel of the linear map
	\eqnn
	(T_{t}[0,1] \oplus T_x B \oplus T_\eta \RR^N) \oplus T_0 \RR^N \to T_{0} (T^*\RR^N).
	\eqnnd
As each $f_t$ is a generating family, the vectors with trivial $T_t[0,1]$ component form a $(\dim B)$-dimensional subspace. In particular, the projection to $T_t[0,1]$ has 1-dimensional image. 
Since the map $\widetilde{\Sigma} \to [0,1]$ is a submersion, 
by Ehresmman's lemma, it is a trivial fiber bundle. So the embedding $\widetilde{\Sigma} \to [0,1] \times J^1 B$ has image given by the trace of an isotopy. We conclude that the Legendrians $\{\Lambda_t\}$ generated by $\{f_t\}$ define a smooth Legendrian isotopy of $\Lambda_0$.
\end{rem}

\begin{defn}\label{defn. path of linear at infty gfs}
A smooth function $f: [0,1] \times B \times \RR^N \to \RR$ is called a {\em smooth path of linear-at-infinity generating families} if 
\begin{enumerate}
\item[(i)]  $f$ is a smooth path of generating families, 
\item[(ii)] there exists a smooth function $\phi: [0,1] \times B \times \RR_t \times \RR^{N-1} \to \RR^N$ such that
\begin{enumerate}
\item[(a)] for every $(t,b) \in [0,1] \times B$, the function $\phi(t,b,-)$ is a diffeomorphism from $\RR_t \times \RR^{N-1}$ to $\RR^N$, and 
\item[(b)] outside a compact subset of $[0,1] \times B \times \RR^N$, we have $f_t \circ \phi_t = t$ for every $t \in [0,1]$.
\end{enumerate}
\end{enumerate}\end{defn}

\begin{rem}
It is straightforward to adapt the proof of Proposition~\ref{prop. stabilization preserves linearity} so that given a path of linear-at-infinity generating families, the stabilization of the path (by, say, a $t$-independent quadratic form) is also a path of linear-at-infinity generating families.
\end{rem}

\begin{rem}
There are in fact some subtleties to defining the space of linear-at-infinity generating families, depending on how one incorporates the data of $\phi$. These subtleties will not matter for us in this work, as the paths $\{f_t\}$ we utilize will actually be constant outside a fixed compact subset of $B \times \RR^N$ (see Proposition~\ref{prop:leg-persist}). 

Proposition~\ref{prop. linear at infinity defns are equiv} generalizes to paths of generating families: A path of generating families $f_t$ satisfying Definition~\ref{defn. path of linear at infty gfs} may be transformed to equal a fixed non-zero linear function $A\eta + c$ outside a compact set.
\end{rem}
 
\subsection{Sublevel sets} The linear-at-infinity condition for our generating families allows us to do Morse theoretic
constructions.  Recall that we are restricting our attention to compact, embedded Legendrian submanifolds.  Lemma~\ref{lem:uniform-non0}
guarantees that the length spectrum of  $\leg$ is bounded away from $0$ for either a single Legendrian or for a 1-parameter family of Legendrians.

\begin{notation}[$\lmin,\lmax$]
\label{notation. lmin lmax}
Let
\begin{equation}\label{defn:ReebMinMax}
0< \lmin  \leq \lmax < \infty
\end{equation} 
denote the minimum and maximum lengths of all the Reeb chords of $\leg$.
\end{notation}
 
Proposition~\ref{prop:leg-crit-point}  implies that  
all positive critical values of $\delta_f: B \times \rr^{2N} \to \rr$ are contained in $[\lmin, \lmax]$. 
Given  the geometric importance of the  critical points of $\delta_f$,  
 Morse theory motivates us to study sublevel sets of $\delta_f$. 

\begin{notation}[Sublevel sets $\delta_f^{\leq a}, h^{\leq a}$]
For any real number $a$, we let 
	\eqnn
	\delta_f^{\leq a} := \{ p \in B \times \RR^N \times \RR^N : \delta_f(p) \in (-\infty, a]\}.
	\eqnnd 
More generally, given any function $h$, we let $h^{\leq a}$ denote the $a$-sublevel set,
 (i.e., the subset of the domain along which $h$ has values in $(-\infty,a])$).
\end{notation} 
 
\begin{prop} 
\label{prop. derivative of difference function is well-behaved}
Fix $\lmax<\infin < \infin'$. Then 
the total derivative of $\delta_f$ is bounded away from zero along the preimage $(\delta_f)^{-1}[\infin,\infin']$. 
Likewise, fix $0 < \epsilon<\epsilon'<\lmin$.  Then 
the total derivative of $\delta_f$ is bounded away from zero along the preimage $(\delta_f)^{-1}[\epsilon,\epsilon']$. 
\end{prop}

\begin{proof}
By the assumption that $f$ is linear-at-infinity, the $\eta$ component of the derivative of $\delta_f$ only approaches zero in a compact region of $B \times \RR^N_{\eta}$. Likewise, the $\te$ component of the derivative only approaches zero in a compact region of $B \times \RR^N_{\te}$. In particular, the total derivative of $\delta_f$ only approaches zero in a compact region of $B \times \RR^N \times \RR^N$.
By Proposition~\ref{prop:leg-crit-point}, we know that the critical values of $\delta_f$ are constrained to $0$ and the intervals $[\lmin,\lmax], [-\lmax,-\lmin]$.
\end{proof}

\begin{choice}
\label{choice:epsilon-omega} 
Given a generating family $f$ for $\leg$, we choose
$\epsilon$ and $ \infin $ such that   
\begin{equation} \label{ineq:epsilon-infin}
  0 < \epsilon < \lmin \leq \lmax < \infin.
\end{equation}
\end{choice}

  In the remainder of this section and paper, we  assume the reader is familiar with basic ideas from the homotopy theory of spaces.  References can be found at the start of Appendix~\ref{ssec:spectra}.

\begin{lem}\label{lem:gf-cofibration}  
Fix a linear-at-infinity generating family $f$ for $\leg$, and $\epsilon,\infin$ as in Choice~\ref{choice:epsilon-omega}.
The inclusion 
 $$i\co \delta_f^{\leq \epsilon} \hookrightarrow \delta_f^{\leq \infin}$$ 
 is a   cofibration.  
\end{lem}

\begin{proof}
To simplify notation, let us write $A = \delta_f^{\leq \epsilon}$ and $\del A = \delta_f^{-1}(\epsilon)$. We likewise write $B = \delta_f^{\leq \infin}$. We must show that the inclusion $A \to B$ satisfies the homotopy extension property.  By choice, $\epsilon$ is a regular value of $\delta_f$. Thus (by reparametrizing a gradient flow as necessary) there is a neighborhood $U$ of $A$ that one may write as
	\eqnn
	U = A \bigcup (\del A \times [0,2]),
	\qquad
	\del U := \del A \times \{2\} \subset B,
	\eqnnd
where $\del U$ is meant to be suggestive notation (rather than conform to a particular definition of boundary).
Suppose one is given a topological space $X$, a family of continuous maps $\{f_t: A \to X\}_{t \in [0,1]}$, and an extension of $f_0$ to a map $g_0: B \to X$. We will first construct
a homotopy extension of $f_t$ to $A \cup (\partial A \times [0,2])$ such that, for all $t$, the homotopy agrees with $g_0$ on $\partial A \times \{2\}$.  Consider the  map
	$$
	 h: (\del A \times [0,2]_s) \times [0,1]_t \to X,
	\qquad
	(a,s,t) \mapsto
	\begin{cases}
	g_0( (a,{\frac {2}{2-t}}(s-t))), &  s \geq t \\
	f_{t-s}(a), & s \leq t.
	\end{cases}
	$$
Observe that $h$ is continuous, and  $h_t: \partial A \times [0,2]_{s}  \to X$  satisfies 
\begin{enumerate}
\item $h_t(a,0) = f_t(a)$, for all $t \in [0,1]$,
\item $h_t(a,2) = g_0(a,2)$,  for all $t \in [0,1]$, where $(a,2) \in \partial A \times \{2\} \subset B$.
\end{enumerate}
Thus we see that $h$ extends, via $g_0$ for all $t$, to   $B \times [0,1]_t$. This proves the inclusion $A \to B$ is a cofibration, as desired.
\end{proof}

  A basic result from the homotopy theory of spaces is that if  $i\co A \to X$ is a cofibration, then the map
$$\Co(i)  \to X/A$$ is a homotopy equivalence.   Thus we have:
\begin{cor}
Fix a linear-at-infinity generating family $f$ for $\leg$ and $\epsilon,\infin$ (Choice~\ref{choice:epsilon-omega}).
Then we have a natural homotopy equivalence from the mapping cone to the quotient:
$$\Co\left( \delta_f^{\leq \epsilon} \hookrightarrow \delta_f^{\leq \infin} \right)
\stackrel{\simeq}{\longrightarrow}
\delta_f^{\leq \infin}/\delta_f^{\leq \epsilon}.$$
\end{cor}

 \begin{rem}\label{rem:old-lem-B9}    Another useful basic result from the homotopy theory of spaces is that  a  pushout square 
\begin{equation*} 
	\xymatrix{
	W \ar[r] \ar[d] & B \ar[d]
	\\
	A \ar[r] & C
	}
\end{equation*}
such that  $W \to B$  or $W \to A$ is a cofibration is a homotopy pushout square.
\end{rem}

We then see that the associated quotient sublevel sets are invariant outside of $[\lmin,\lmax]$:
 
\begin{prop}\label{prop. omega epsilon independence}
Fix $f: B \times \RR^N \to \RR$. We assume $f$ is linear-at-infinity and generates a Legendrian $\leg \subset J^1B$. Then 
\begin{enumerate}
	\item\label{item. omega invariance} For any $\infin' > \infin > \lmax$,  the inclusion
		$
		\delta_f^{\leq \infin} \to
		\delta_f^{\leq \infin'} 
		$
	is a homotopy equivalence. 
	\item\label{item. epsilon invariance} For any $\lmin > \epsilon' > \epsilon > 0$,  the inclusion
		$
		\delta_f^{\leq \epsilon} \to
		\delta_f^{\leq \epsilon'} 
		$
	is a homotopy equivalence. 
	\item\label{item. omega epsilon invariance} The induced map
		\eqnn
		(\delta_f^{\leq \infin} / \delta_f^{\leq \epsilon})
		\to
		(\delta_f^{\leq \infin'} / \delta_f^{\leq \epsilon'})
		\eqnnd
	is a homotopy equivalence.
\end{enumerate}
\end{prop}

\begin{proof}
The proof of \eqref{item. omega invariance} follows from Lemma~\ref{lem:retract} setting $a = \infin$ and $b = \infin'$; for~\eqref{item. epsilon invariance} we set $a = \epsilon$ and $b = \epsilon'$.

For \eqref{item. omega epsilon invariance}, by  Lemma~\ref{lem:gf-cofibration}  the inclusions in \eqref{item. omega invariance} and \eqref{item. epsilon invariance} are cofibrations.  Thus, by Remark~\ref{rem:old-lem-B9},    the quotients are homotopy pushouts (and in particular, homotopy invariant).
\end{proof}
  
\begin{rem}
Proposition~\ref{prop. omega epsilon independence}\eqref{item. omega invariance} implies that for $\infin$ sufficiently large, the inclusion $\delta_f^{\leq \infin} \to B \times \RR^N$ is a homotopy equivalence.
\end{rem}

\begin{rem}
\label{rem. sublevel pair functoriality}
Proposition~\ref{prop. omega epsilon independence} may be interpreted as follows. Consider the partially ordered set
	\eqn\label{eqn. epsilon omega constraints}
	\{(\epsilon,\infin) \, | \,
	0 < \epsilon < \lmin < \lmax < \infin \}
	\subset \RR \times \RR
	\eqnd
ordered by $\leq$ in each factor. Then the assignment
	\eqnn
	(\epsilon,\infin) 
	\mapsto 
	\delta_f^{\leq \infin} / \delta_f^{\leq \epsilon}
	\eqnnd
defines a functor from this poset to the category of pointed topological spaces, and in particular to the $\infty$-category of topological spaces. By Proposition~\ref{prop. omega epsilon independence}\eqref{item. omega epsilon invariance}, this is an essentially constant functor.
\end{rem}

\subsection{Legendrian isotopies and the appearance of stabilizations}
The linear-at-infinity condition survives Legendrian isotopies:
 
\begin{prop} [Path Lifting for Generating
  Families] \label{prop:leg-persist} Suppose $B$ is compact. 
  For $t \in [0,1]$, let $\leg_t \subset J^1B$ be an isotopy of Legendrian submanifolds.  If $\leg_0$ has a linear-at-infinity generating family $f$, then there exists a smooth path of linear-at-infinity generating families $f_t\co B \times \rr^N \to \rr$ for $\leg_t $  (see Section~\ref{section. paths of gfs})
   such that $f_0$ is a stabilization of $f$, and $f_t = f_0$ outside a compact set.
\end{prop}

\begin{rem}
Proposition~\ref{prop:leg-persist}  is the only place that stabilizing of generating families is necessary for creating an invariant. Put a different way, Proposition~\ref{prop:leg-persist}  illustrates that stabilization is a useful equivalence relation on generating families (aside from the obvious fact that stabilizing preserves the underlying Legendrian on the nose).
\end{rem}

\begin{proof}
The argument can be proved via
Chekanov's ``composition formula'' \cite{chv:quasi-fns}; see, for example, \cite[Appendix]{traynor:helix}.  
\end{proof}
  
\begin{rem} 
\label{remark. compactifying Rn}
We will often be considering generating families for compact Legendrians in $J^1\rr^n$ -- that is, for non-compact $B$.  Proposition~\ref{prop:leg-persist} will still apply when $B = \RR^n$, as any Legendrian isotopy in $\RR^n$ automatically takes place in $J^1S^n$.
\end{rem}

\subsection{Stabilization and suspension} \label{ssec:morse-theory}
We identified a constant (up to homotopy equivalence)  family of pairs of spaces in Proposition~\ref{prop. omega epsilon independence}. 
We now study the invariance of these pairs with respect to the equivalence operations of Definition~\ref{defn:gf-equiv}.

It is clear that the fiber-preserving diffeomorphisms preserve pairs up to diffeomorphism. We will see that stabilization only preserves pairs up to homotopy equivalence and suspension. (Versions of these statements at the level of homology are proven in \cite[Lemma 4.7]{josh-lisa:cap}.)
Before proceeding, it will be useful to recall a homotopy equivalence of the stabilization of the quotient.

\begin{rem}
\label{remark. sublevel set pair suspension} 
  Given a cofibration $i \co A \hookrightarrow X$, 
 for any closed interval $I$ of positive length, there is a natural homotopy equivalence
$$\Sigma(X/A) \stackrel{\simeq}{\longrightarrow}  X \times I /  \left((A \times I) \cup (X \times \partial I)\right).$$
Furthermore, for any $a< b$, the map
$$\Sigma(X/A) \stackrel{\simeq}{\longrightarrow}  X \times \RR /  \left((A \times \RR) \cup (X \times (-\infty,a] \cup [b,\infty) \right).$$
-- obtained, for example, by choosing a homeomorphism $I \cong [a,b]$ and accordingly including $I$ into $\RR$ -- is a homotopy equivalence.
\end{rem}

\begin{rem}
We have seen (Proposition~\ref{prop:leg-persist}) that Legendrian isotopies necessitate the appearance of stabilizations of generating families. 
Thus, the appearance of suspensions is a hint that spectra are the ``correct'' category in which these invariants take values. See Section~\ref{section. finite dim approx} for related remarks.
\end{rem}

\begin{notation}[$g_{\pm}$]
Let $V$ be any set and let $g: V \to \RR$ be a function.
We let 
	\eqnn
	g_+  \co V \times \rr \to \rr,
	\qquad
	(x,\eta) \mapsto g(x) + \eta^2
	\eqnnd
and
	\eqnn
	g_-  \co V \times \rr \to \rr,
	\qquad
	(x,\eta) \mapsto g(x) - \eta^2
	\eqnnd
denote the stabilizations of $g$. 
\end{notation}

For every pair of real numbers $a<b$, one has the following maps of pairs:
	\begin{align}
	( g^{\leq b}  , g^{\leq a} )
		&\xrightarrow{x \mapsto (x,0)} 
		&(g_+^{\leq b},g_+^{\leq a})
		\label{eqn. positive stabilization map}, \\
	( g^{\leq b} \times \RR  
		, g^{\leq a} \times \RR \bigcup g^{\leq b} \times \{|\eta| \geq \sqrt{b-a}\} )
		&\xrightarrow{x \mapsto x}  
		&(g_-^{\leq b},g_-^{\leq a}).  \label{eqn. negative stabilization map} 
	\end{align}
  Observe that the domain of~\eqref{eqn. negative stabilization map} models the reduced suspension of the pair $(g^{\leq b},g^{\leq a})$; see Remark~\ref{remark. sublevel set pair suspension}.

\begin{lem}\label{lem:stab2-suspension}
\label{lemma. invariance of sublevel spaces}
Fix a function $g: V \to \RR$.
    \begin{enumerate}
        \item\label{item. positive stab does not change sublevel pair} For all $a < b$, the inclusion~\eqref{eqn. positive stabilization map} is a homotopy equivalence of pairs.
        \item\label{item. negative stab is suspension} Now assume the domain of $g$ is a smooth manifold. Further assume that there exists some gradient-like vector field $X$ of $g$ for which 
        	\begin{enumerate}
			\item $X$ is complete,
			\item $X$ is bounded away from zero on $g^{-1}([b,\infty))$, and
			\item $X$ is bounded away from zero on $g^{-1}(a-\varepsilon, a+\varepsilon)$, for some $\varepsilon > 0$. 
			\end{enumerate}
        Then the map~\eqref{eqn. negative stabilization map} 
        is a homotopy equivalence.
	\end{enumerate}
\end{lem}
\begin{rem}  In particular, Lemma~\ref{lem:stab2-suspension} states that positive stabilization never changes the homotopy type of a sublevel set pair, while (when $b$ is sufficiently large) negative stabilization suspends the homotopy type of a sublevel set pair.
\end{rem}

\begin{proof}[Proof of Lemma~\ref{lem:stab2-suspension}.]
We first prove~\eqref{item. positive stab does not change sublevel pair}. Note that $V \times \RR$ has a strong deformation retraction to $V \times \{0\}$, for example by the straight-line homotopy  $\eta \mapsto (1-t)\eta$ in the $\RR$ coordinate. For $t \in [0,1]$, and for any $c$, we clearly have that $(x,\eta) \in g_+^{\leq c} \implies (x,(1-t)\eta) \in g_+^{\leq c}$. This homotopy retracts the pair $(g_+^{\leq b},g_+^{\leq a})$ to the desired image; see Figure~\ref{fig:retract-plus}.

Now we prove~\eqref{item. negative stab is suspension}. 
Let us first note that $V \times \RR$ strongly deformation retracts to $g^{\leq b} \times \RR$. Here is one construction of the retraction: By the assumption that $\nabla g$ is complete, we can flow by $-\nabla g$ in the $V$ component (while leaving the $\RR$ component fixed), and by the assumption on critical values of $g$, any $x$ with $g(x)>b$ flows to an element $x'$ with $g(x')=b$. An appropriate time- and $g$-dependent flow map, glued to the constant map along $g^{\leq b} \times \RR$, achieves the retraction.
Next, we note that the space
	\eqnn
	g_-^{\leq a} 
	\bigcap
	g^{\leq b} \times \RR
	\eqnnd
deformation retracts to the space
	\eqnn
	g^{\leq a} \times \RR \bigcup g^{\leq b} \times \{|\eta| \geq \sqrt{b-a}\}.
	\eqnnd
Indeed, fix some small $\epsilon >0$ -- then for those $(x,\eta)$ where $g(x) \geq a + \epsilon$, one can expand the interval $[-\sqrt{g(x) - a},\sqrt{g(x)-a}]$ to the interval $[-\sqrt{b-a},\sqrt{b-a}]$; if $\epsilon$ is a priori chosen small enough so there are no critical values near $a$ (which is possible by hypothesis), we may then retract to $g^{\leq a} \times \RR$; see Figure~\ref{fig:retract-minus}.
\end{proof}

\begin{figure}
		\[
			\xy
			\xyimport(8,8)(0,0)
			{
			\includegraphics[width=2in]{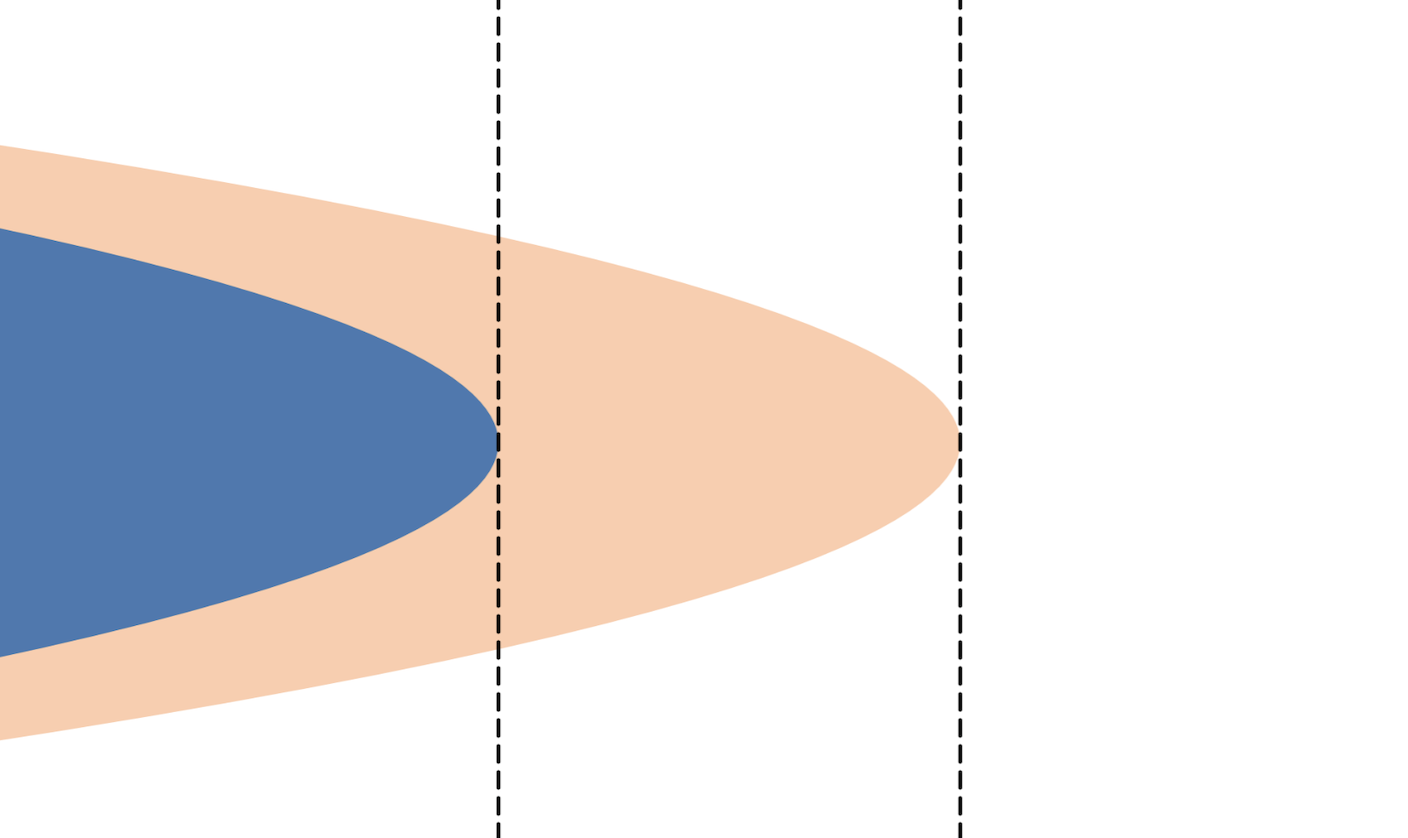}
			\qquad
			\includegraphics[width=2in]{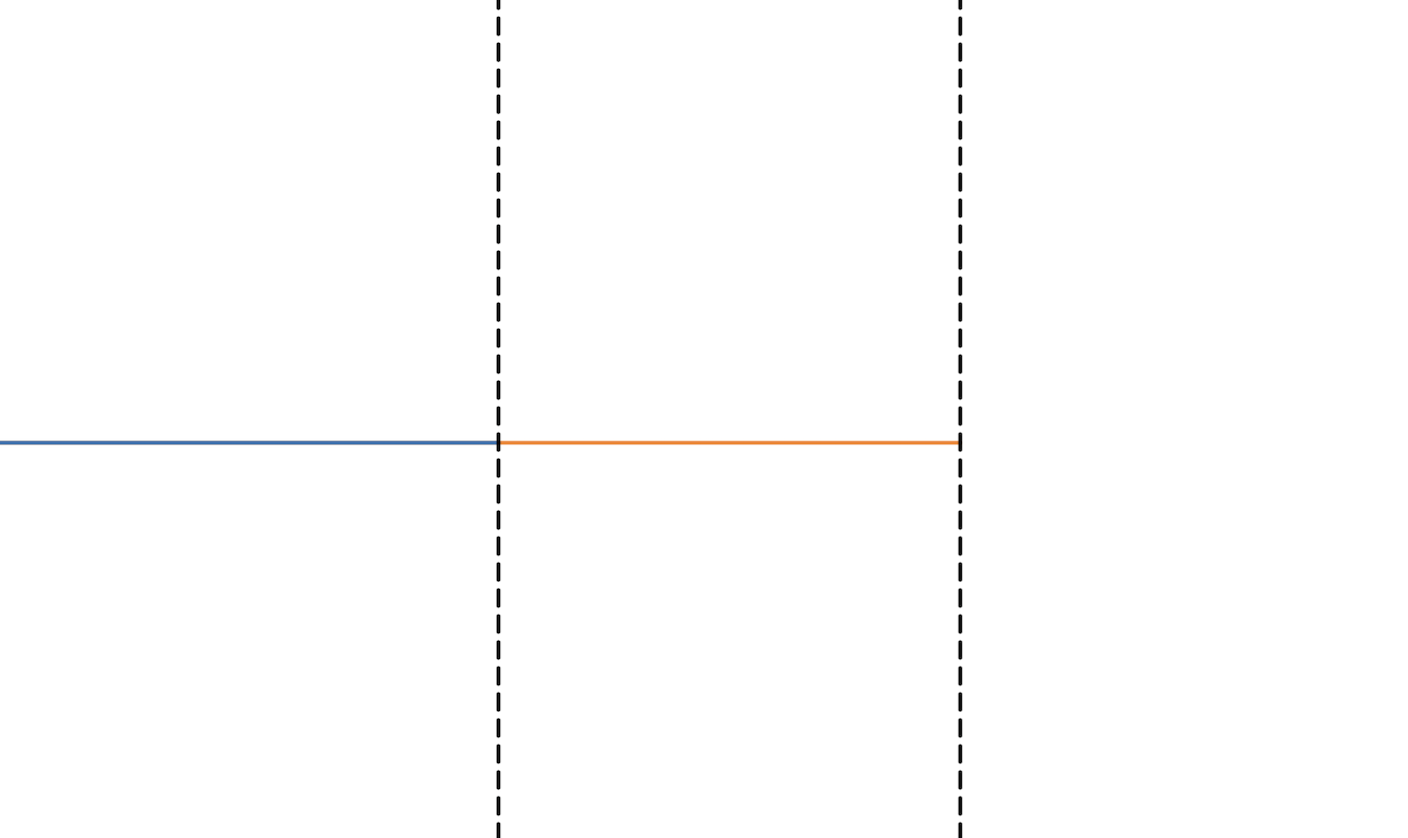}
			}
			,(1,8)*+{g=a}
			,(2.2,8)*+{g=b}
			,(5.2,8)*+{g=a}
			,(6.4,8)*+{g=b}
			\endxy
		\]
	\caption{
	A depiction of the pair $(g_+^{\leq b},g_+^{\leq a})$ -- indicated by the shaded regions on the left -- and of the pair $(g^{\leq b},g^{\leq a})$, indicated by the horizontal lines on the left. Note these are the images of the two pairs under the map $(g,\eta)$ to $\RR^2$, and the vertical dashed lines indicate the loci where $g=a$ and $g=b$. The deformation retraction from the left image to the right image is obtained by retracting the vertical $\eta$ coordinate to zero.
	}
	\label{fig:retract-plus}
\end{figure}

\begin{figure}
		\[
			\xy
			\xyimport(8,8)(0,0)
			{
			\includegraphics[width=1.3in]{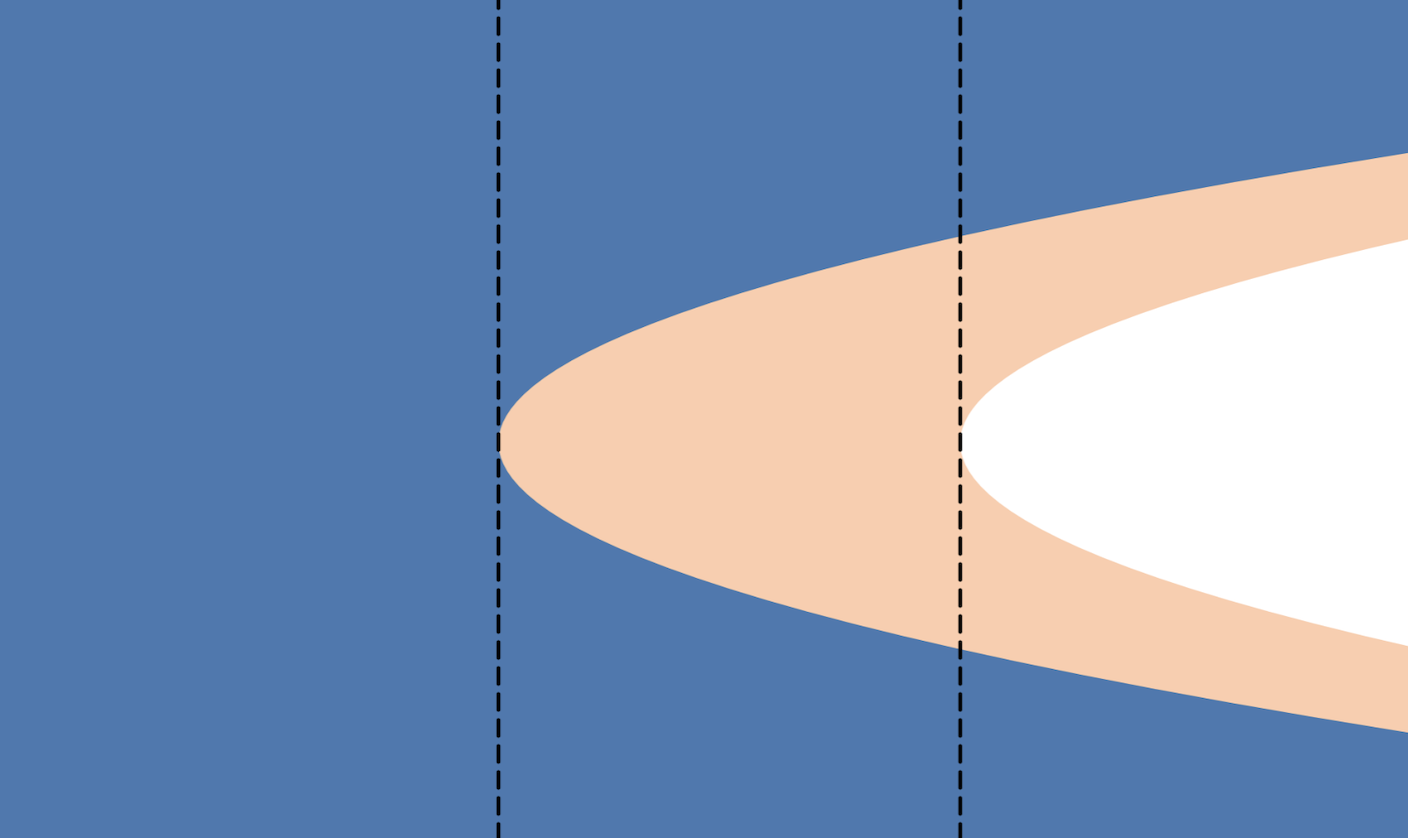}
			\qquad
			\includegraphics[width=1.3in]{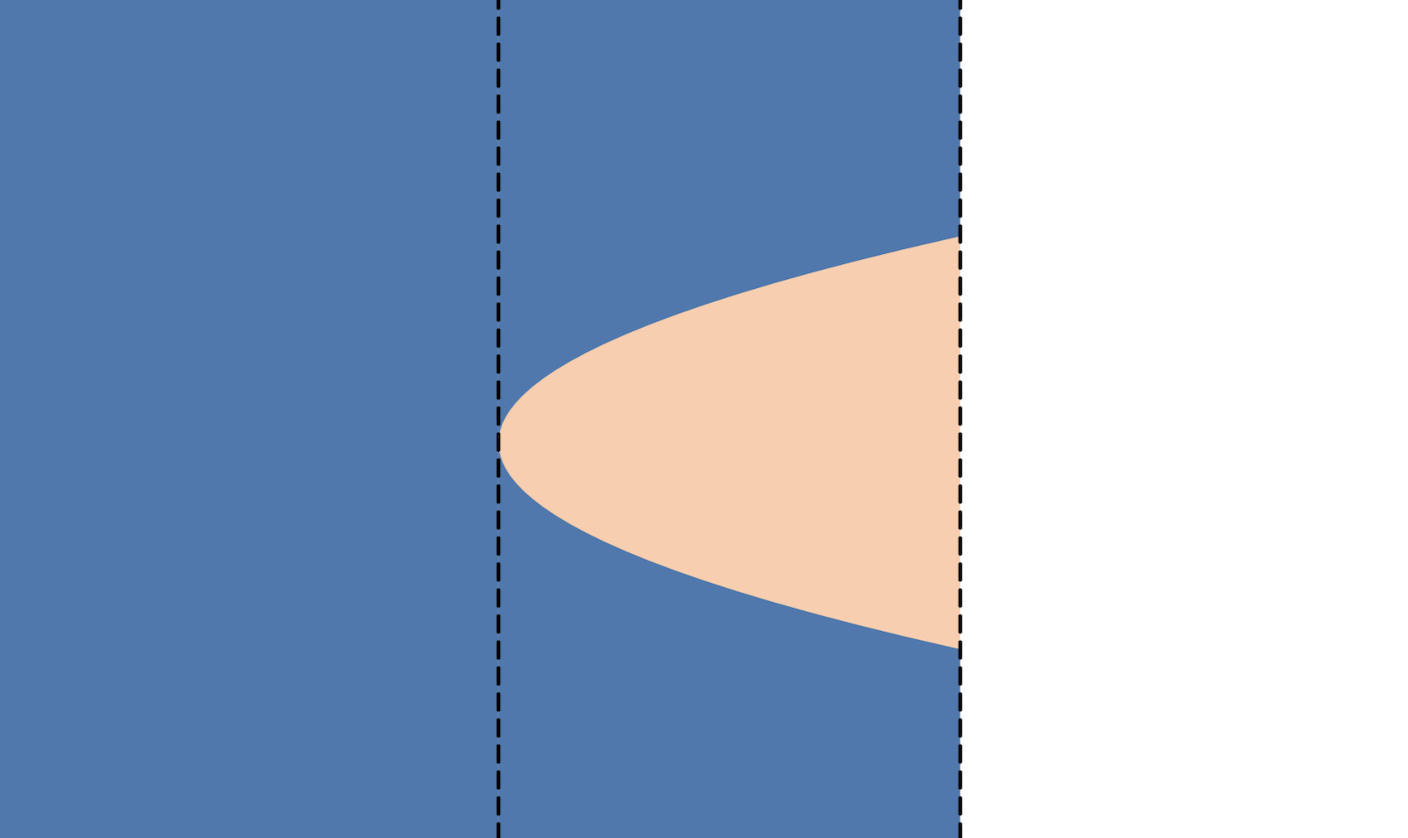}
			\qquad
			\includegraphics[width=1.3in]{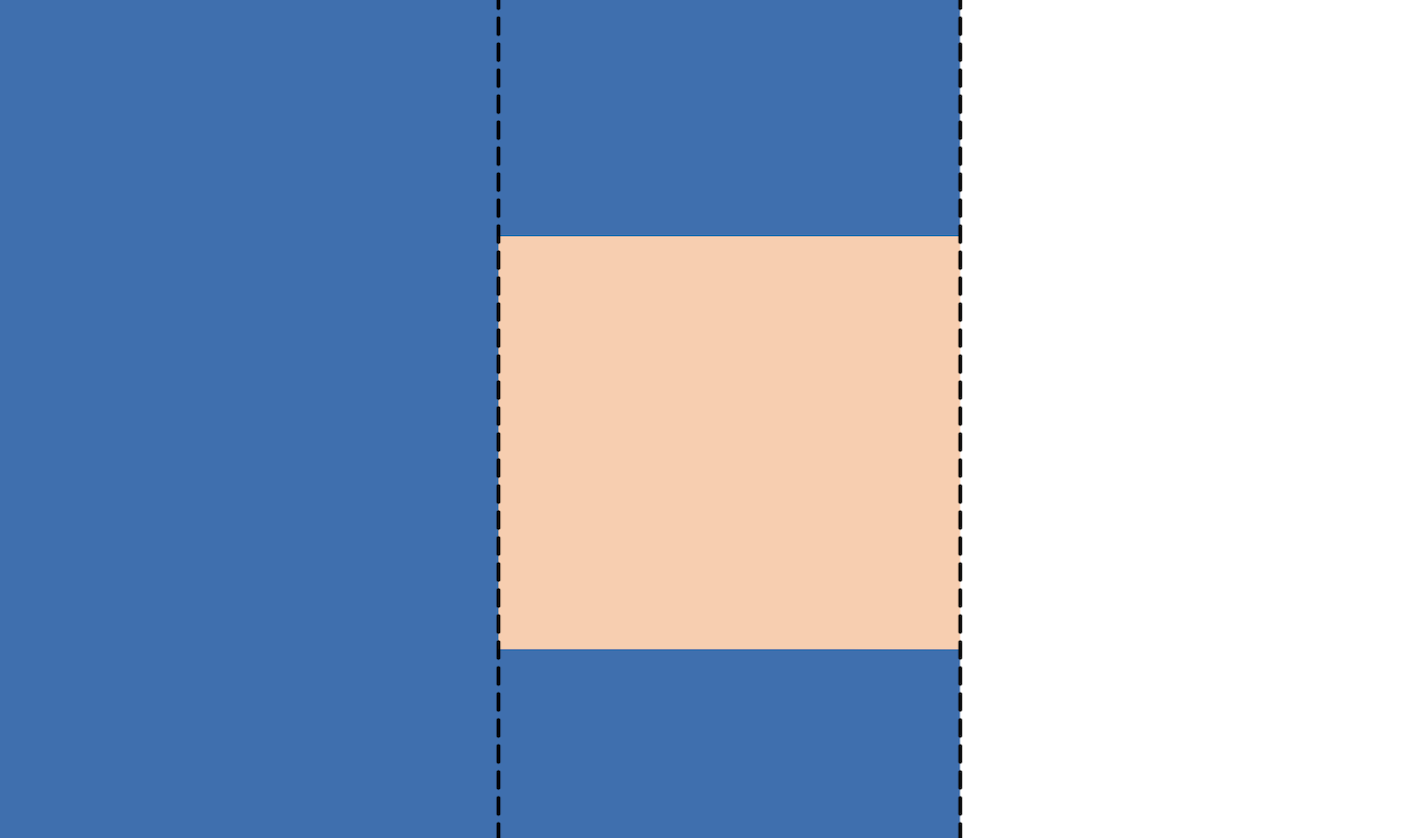}
			}
			\endxy
		\]
	\caption{
	As before, all images here take place in the $(g,\eta)$ plane. 
	The leftmost image is an image of the pair $(g^{\leq b}_{-},g^{\leq a}_{-})$, with the darker shaded region representing $g^{\leq a}_{-}$. The middle image is the result of retracting to the locus where $g \leq b$. The rightmost image is the domain of~\eqref{eqn. negative stabilization map}, obtained by retracting the curved dark region of the middle image to the rectilinear dark region in the rightmost image.
	}
	\label{fig:retract-minus}
\end{figure}

\begin{rem}
\label{remark. suspension compatibility}
We saw that the collection of $\delta^{\leq \infin}/\delta^{\leq \epsilon}$ is constant up to homotopy equivalence in Remark~\ref{rem. sublevel pair functoriality}. We now explore the dependency of the maps~\eqref{eqn. positive stabilization map}  and~\eqref{eqn. negative stabilization map}  on $a,b$ values.
Fix $a<a'<b<b'$. 
To save space, let us write
	\eqnn
	C_{a,b,b'} := g^{\leq a} \times \RR \bigcup g^{\leq b} \times \{|\eta| \geq \sqrt{b'-a}\}
	\eqnnd
so we have natural inclusions fitting into a commutative diagram as follows:
	\eqnn
	\xymatrix{
	C_{a,b,b} \ar[d]
		& C_{a,b,b'} \ar[r]  \ar[l] \ar[d]
		& C_{a,b',b'}  \ar[d] \\
	C_{a',b,b}   
		& C_{a',b,b'}  \ar[r]  \ar[l] 
		& C_{a',b',b'} 
	}
	\eqnnd
and in particular a commuting diagram of pairs
	\eqnn
	\xymatrix{
	(g^{\leq b} \times \RR , C_{a,b,b}) \ar[d]
		&(g^{\leq b} \times \RR ,  C_{a,b,b'} )\ar[r]  \ar[l] \ar[d]
		& (g^{\leq b'} \times \RR , C_{a,b',b'} ) \ar[d] \\
	(g^{\leq b} \times \RR , C_{a',b,b}   )
		& (g^{\leq b} \times \RR , C_{a',b,b'} ) \ar[r]  \ar[l] 
		& (g^{\leq b'} \times \RR , C_{a',b',b'} ),
	}
	\eqnnd
which in turn forms the back face of the following commutative diagram of pairs:
	\eqn
	\label{eqn. naturality of suspension equivalence}
	\xymatrix{
	(g^{\leq b} \times \RR , C_{a,b,b}) \ar[dd] \ar[dr]
		&(g^{\leq b} \times \RR ,  C_{a,b,b'} )\ar[r]  \ar[l] \ar[dd]
		& (g^{\leq b'} \times \RR , C_{a,b',b'} ) \ar[dd] \ar[dr] \\
	\,
		&	(g_-^{\leq b}, g_-^{\leq a}) \ar[rr] \ar@/^/[dd]
		&&	(g_-^{\leq b'}, g_-^{\leq a})  \ar@/^/[dd] \\
	(g^{\leq b} \times \RR , C_{a',b,b}   ) \ar[dr]
		& (g^{\leq b} \times \RR , C_{a',b,b'} ) \ar[r]  \ar[l] 
		& (g^{\leq b'} \times \RR , C_{a',b',b'} ) \ar[dr] \\
	\,
		&	(g_-^{\leq b}, g_-^{\leq a'}) \ar[rr]
		&&	(g_-^{\leq b'}, g_-^{\leq a'}).
	}
	\eqnd
As long as $a'$ and $b'$ are chosen from the neighborhoods of $a$ and $b$ guaranteed in Lemma~\ref{lemma. invariance of sublevel spaces}~\eqref{item. negative stab is suspension}, every map in~\eqref{eqn. naturality of suspension equivalence} is a homotopy equivalence of pairs. (The diagonal maps -- from the back corners to front corners of the diagram -- are equivalences by Lemma~\ref{lemma. invariance of sublevel spaces}.) 
\end{rem}

\begin{rem}[Stabilization induces suspension]
\label{remark. stabilization causes suspension}
Consider $g = \delta_f$ and $a = \epsilon, b = \infin$ for $(\epsilon,\infin)$ satisfying~\eqref{eqn. epsilon omega constraints}.
The hypotheses of Lemma~\ref{lemma. invariance of sublevel spaces} are then satisfied thanks to Proposition~\ref{prop. derivative of difference function is well-behaved}. Moreover, by 
Remark~\ref{rem:stab-choice-difference}, we have that
	\eqnn
	\delta_{f_+} \qquad\text{and}\qquad \delta_{f_-}
	\eqnnd
are both rank 2 stabilizations of 
$\delta_f$ by a quadratic of index $1$.  
Thus, Lemma~\ref{lemma. invariance of sublevel spaces} implies that if we stabilize a generating family $f$ (positively or negatively), then for any choice of $\epsilon,\infin$ from~\eqref{eqn. epsilon omega constraints}, the map~\eqref{eqn. negative stabilization map}  is a homotopy equivalence of pairs. 
Interpreting the domain pair using Remark~\ref{remark. sublevel set pair suspension}, we conclude that stabilization of $f$ causes the sublevel set pair of the difference function $\delta_f$ to undergo a suspension of pairs. 
\end{rem}

  Lemma~\ref{lem:stab2-suspension} together with Remarks~\ref{remark. sublevel set pair suspension} and \ref{remark. stabilization causes suspension} leads to:

\begin{prop}\label{prop:leg-stab} If $f'$ differs from $f$ by a rank $1$ stabilization, then for any compact interval of positive length, ~\eqref{eqn. negative stabilization map} induces a map
	\eqn\label{eqn:leg-stab}
	i^\leg: \left(\delta_{f}^{\leq \infin} \times I, \  \delta_{f}^{\leq \epsilon} \times I \cup \delta_{f}^{\leq \infin} \times \partial I \right)
	\to
	\left(\delta_{f'}^{\leq \infin}, \  \delta_{f'}^{\leq  \epsilon}   \right) 
	\eqnd
inducing a homotopy equivalence
$$\sigma: \Sigma\left(\delta_f^{\leq \infin}/ \delta_f^{\leq \epsilon} \right) \stackrel{\simeq}{\longrightarrow} 
\delta_{f'}^{\leq \infin}/ \delta_{f'}^{\leq \epsilon}.
$$
\end{prop}

\begin{rem}[Naturality of the stabilization-suspension pathway]
\label{remark. naturality of stabilization}
Moreover, we observed in Remark~\ref{rem. sublevel pair functoriality} that the sublevel set pair $\delta_f^{\leq \infin}/\delta_f^{\leq \epsilon}$ is independent of choice of $\epsilon$ and $\infin$ (up to homotopy equivalence of pairs). This constant in the choice of $(\epsilon,\infin)$ is compatible with the suspension maps thanks to Remark~\ref{remark. suspension compatibility}. Indeed, note that the front rectangle of~\eqref{eqn. naturality of suspension equivalence} consists of the homotopy equivalences mentioned in Remark~\ref{rem. sublevel pair functoriality}.
\end{rem}

\section{The generating family spectrum of a Legendrian}
In this section, we define the spectrum of a Legendrian submanifold equipped with a generating family, prove Theorem~\ref{theorem. bound on family dimension}, which gives a lower bound
on the needed fiber dimension for a Legendrian  and generating family within their equivalence class,  and show that homology groups of a spectrum recover the
previously established generating family homology groups (Theorem~\ref{thm:spec-lift}).  Background on homotopy theory of spectra  is included in
  Appendix~\ref{ssec:spectra}
and are referenced throughout this section.

\subsection{Definition}

\begin{defn}[$C(-;\Sphere)$] 
\label{defn:leg-spectrum} Given a Legendrian $\leg \subset J^1B$ with a linear-at-infinity generating family 
	\eqn\label{eqn. generating family domain and codomain}
	f\co B \times \rr^N \to \rr,
	\eqnd
define the sequence
 of functions 
 	\eqnn
	\{f_i\co B \times \rr^i \to \rr \}_{i \geq N}
	\eqnnd
where $f_N = f$, and $f_{i}$ is the rank $1$ stabilization of $f_{i-1}$ by  either $Q_+(\e) = \e^2$ or $Q_-(\e) = -\e^2$.  (See Remark~\ref{rem:plus-or-negative-stab}.) Then for all $i \geq N$,  we have spaces and homotopy equivalences as follows:  
\begin{enumerate}
\item For all $i \geq N$, let $X_{i} = \delta_{f_{i}}^{\leq \infin}/ \delta_{f_{{i}}}^{\leq \epsilon}$, 
 \item  $\Sigma X_{i} \xrightarrow{\simeq} X_{i+1}$ provided by Proposition~\ref{prop:leg-stab}.
 \end{enumerate}
These data define the generating family prespectrum of $(\leg,f)$. The {\bf generating family spectrum of $(\leg, f)$} is the associated spectrum (Construction~\ref{construction. spectrification}), and we denote this spectrum by
 	\eqnn
	\gfc(\leg, f; \mathbb S).
	\eqnnd
 \end{defn}

\begin{rem}\label{rem:plus-or-negative-stab}   
 To define $f_i$, stabilizing by either $Q_+$ or $Q_-$  gives the same end result due the symmetry of $\delta_f$ -- see Remark~\ref{rem:stab-choice-difference}. 
\end{rem}

\begin{rem}
By Remark~\ref{remark. naturality of stabilization}, the generating family spectrum associated to $f$ is naturally independent of the choices of $\epsilon$ and $\infin$ -- as long as $\epsilon$ and $\infin$ satisfy the inequalities in Choice~\ref{choice:epsilon-omega} -- up to equivalence of spectra.
\end{rem}

\begin{rem}
 Recall that a spectrum $X$ is called {\em finite} if, after finitely many suspensions, $X$ is equivalent to the suspension spectrum of a finite CW complex.  Since $\leg$ is compact,
Proposition~\ref{prop:leg-crit-point} and standard Morse theory arguments imply that the space $\delta^{\leq \omega}/\delta^{\leq \epsilon}$ is homotopy equivalent to a CW complex with finitely many cells (in bijection with the positive-length Reeb chords). It follows that $\gfc(\leg,f;\Sphere)$ is a  finite  spectrum.
\end{rem}

\begin{ex}
\label{example. cubic over a point}
Take $B$ to be a point and let $f: B \times \RR \to \RR$ be any cubic function with two distinct critical values. Choosing a diffeomorphism $\RR \to \RR$ which equals $\eta \mapsto \eta^{1/3}$ outside a compact subset, we see that $f$ is linear at infinity (Definition~\ref{defn. linear-at-infinity}). Further, $f$ generates a Legendrian $\leg \subset J^1B \cong \RR$, where $\leg$ is a zero-dimensional manifold consisting of two points. One can compute that for $\epsilon,\omega$ as in Choice~\ref{choice:epsilon-omega}, the quotient space $\delta_{f}^{\leq \infin}/ \delta_{f}^{\leq \epsilon}$ is homotopy equivalent to a two-dimensional sphere (with basepoint given by the quotient locus). 
Because $N = 1$ for this choice of $f$, we see that $C(\leg,f;\Sphere)$ defines a prespectrum beginning at index $N=1$:
	$$
	X_1 \simeq S^2,\qquad X_2 \simeq \Sigma S^2 \simeq S^3,
	\qquad\ldots,\qquad X_i \simeq S^{i+1},\qquad\ldots.
	$$
In particular, $C(\leg,f;\Sphere)$ is a 1-fold suspension of the sphere spectrum, otherwise known as the suspension spectrum of the circle:
	\eqnn
	C(\leg,f;\Sphere) \simeq \Sigma^\infty S^1 \simeq \Sigma \Sigma^\infty S^0 \simeq \Sigma \Sphere =: \Sphere^1.
	\eqnnd
\end{ex}

\subsection{Invariance}

\begin{prop}\label{prop:equiv-class-specta}
   If $f, f'$ are both linear-at-infinity generating families for $\leg$, and $f, f'$ differ by a sequence of
fiber-preserving diffeomorphisms and stabilizations, then the associated spectra are equivalent:
$$\gfc(\leg, f; \mathbb S) \simeq 
 \gfc(\leg, f'; \mathbb S).$$
 \end{prop}

 \begin{proof}   If $f'$ differs from $f$ by fiber-preserving diffeomorphism, then, there is an immediate diffeomorphism of pairs  
 	$$\left(\delta_{f}^{\leq \infin}, \delta_{f}^{\leq \epsilon}\right) \cong
 \left(\delta_{f'}^{\leq \infin}, \delta_{f'}^{\leq \epsilon}\right)$$
compatible with the stabilization maps, so the associated spectra are equivalent. 
Further, if $f'$ is a stabilization of $f$, then (up to homotopy equivalence of pointed spaces) the sequence of spaces defining the generating family prespectrum for $f'$ is a subsequence of those spaces defining the prespectrum of $f$, so the spectra are equivalent by Proposition~\ref{prop:prespectra-spectra-maps}.
\end{proof}

Further, Legendrian isotopies induce equivalences of generating family spectra. 

\begin{thm} \label{thm:iso-spectra} 
Fix a compact embedded Legendrian
$\leg \subset J^1B$
and a linear-at-infinity generating family $f$ for $\leg$.
Fix a path $\leg_t$ of Legendrians in $J^1B$ with $\leg_0 = \leg$. Then for any path of generating families $f_t$ for $\leg_t$ as in Proposition~\ref{prop:leg-persist},
there exists an equivalence of spectra
$$\gfc(\leg, f; \mathbb S) \stackrel{\simeq}{\longrightarrow} \gfc(\leg_t, f_t; \mathbb S), \quad \forall t \in [0,1].$$
\end{thm}

\begin{proof}  
By Proposition~\ref{prop:leg-persist}, we know
that the path $\leg_t$ lifts to a path of linear-at-infinity generating families $f_t\co B \times \rr^N \to \rr$, where $f_0$ is a stabilization
of $f$ (and $N$ is some large integer). One thus obtains a path of difference functions $\delta_t \co B \times \rr^{2N} \to \rr$. 
By Proposition~\ref{prop:equiv-class-specta}, the spectra $\gfc(\leg, f; \mathbb S)$ and $\gfc(\leg, f_0; \mathbb S)$ are equivalent. 
By Proposition~\ref{prop:leg-crit-point}(2) and Lemma~\ref{lem:uniform-non0}, the family $\delta_t$ of difference functions satisfies the hypotheses of Lemma~\ref{lem:crit-non-crossing}.
Thus  we get a homotopy equivalence between the spaces in the prespectra associated to $f_0$ and $f_t$. Thus, 
the spectra $\gfc(\leg, f_0; \mathbb S)$ and 
$\gfc(\leg, f_t; \mathbb S)$ are equivalent by Proposition~\ref{prop:prespectra-spectra-maps}.
\end{proof}

\subsection{Proof of Theorem~\ref{theorem. bound on family dimension} } \label{ssec:dim-bound}
\begin{proof}[Proof]
Suppose $f' : B \times \RR^M \to \RR$ is a generating family equivalent to $f$ (up to stabilizations and fiber-preserving diffeomorphism and Legendrian isotopy). 
Then by Definition~\ref{defn:leg-spectrum}, 
the $+M$-fold suspension $\Sigma^{M} \gfc(\leg,f;\Sphere)$ -- see Notation~\ref{notation:shift-spec} -- is the suspension spectrum (Definition~\ref{defn:suspension-spec}) of the pointed space
 $\delta_{f'}^{\leq \omega}/\delta_{f'}^{\leq \epsilon}$.   
\end{proof}

\subsection{Recovering generating family homology}
\label{section. GFH from GS spectra}

In this section, we omit the coefficient abelian group $A$ from our homologies. The results are true regardless of choice of $A$.

\begin{defn} \label{defn:GF-homology} Given a Legendrian $\leg \subset J^1B$ with linear-at-infinity generating family $f\co B \times \rr^N \to \rr$,
the {\bf generating family homology groups} 
are defined as
 $$GFH_{k}(\leg, f) := H_{k+N}\left(\delta_f^{\leq \infin}, \delta_f^{\leq \epsilon}\right).$$
As before, $\epsilon$ and $\infin$ are from Choice~\ref{choice:epsilon-omega}.
\end{defn} 

\begin{notation}
\label{notation. c =0 grading}
Implicit in the notation $GFH_k$ is that we are using the $c=1$ grading convention -- see~\eqref{eqn. gfh is lch}. 
For the $c=0$ convention, we will explicitly include a superscript and set
 \eqn\label{eqn. c = 0 grading GFH}
 GFH_{k}^{c=0}(\leg, f) := H_{k+N+1}\left(\delta_f^{\leq \infin}, \delta_f^{\leq \epsilon}\right).
 \eqnd
\end{notation}
 
\begin{rem} \label{rem:GFH-indexing} Generating family homology for Legendrians have their roots in the generating family homology groups of links defined in \cite{traynor:gf-polys, lisa-jill}; 
these papers restrict to the setting of Legendrian links where each component has a unique {\em quadratic}-at-infinity generating family, up to fiber-preserving diffeomorphism and stabilization, and show
 that generating family homology is an effective invariant.
The version of generating family homology
for a single component Legendrian was defined in \cite{f-r}. 

Given a generating family $f: B \times \RR^N \to \RR$, one can index the $k$th generating family homology group to be either 
$$H_{k+N+1}\left(\delta_f^{\leq \infin}, \delta_f^{\leq \epsilon}\right) \quad{\text{or}} \quad
H_{k+N}\left(\delta_f^{\leq \infin}, \delta_f^{\leq \epsilon}\right).$$
To remove confusion, we have placed the superscript $c=0$ to indicate the first of these conventions~\eqref{eqn. c = 0 grading GFH} -- a convention we only use when this superscript is explicitly shown.
As demonstrated in \cite{f-r}, by choosing the $k+N+1$ option, indices match with 
linearized contact homology $LCH_k(\leg,\epsilon)$, in the sense that for $\leg \subset \rr^3$, for every linear-at-infinity generating family $f$ of $\leg$, there
exists an augmentation $\epsilon_f$ such that 
$$GFH_k^{c=0}(\leg, f; \ZZ/2\ZZ) \cong LCH_k(\leg, \epsilon_f; \ZZ/2\ZZ).$$ 
On the other hand, the $k+N$ convention -- which is the $c=1$ grading convention in~\eqref{eqn. gfh is lch}, and for which we never display a superscript ``$c=1$''  -- has its benefits (Remark~\ref{remark. grading shift in LCH}).
  As pointed out to us by the referee, the $c = 1$ convention is consistent with the usual grading induced by the bar construction; see Remark~\ref{rem:spectral lift conjectures}.
 \end{rem}

\begin{proof}[Proof of Theorem~\ref{thm:spec-lift} ]
We have that
	\begin{align}
	H_k(C(\leg,f;\Sphere))
		& := \colim_{i \to \infty} \widetilde{H}_{N+k+i} \Sigma^i ( \delta_{f}^{\leq \omega}/ \delta_{f}^{\leq \epsilon} )
		\label{homology-computation-1}\\
		& \cong \widetilde{H}_{N+k} \left( \delta_{f}^{\leq \omega}/ \delta_{f}^{\leq \epsilon} \right)
		\label{homology-computation-2}\\
		& \cong H_{N+k} \left(  \delta_{f}^{\leq \omega} , \delta_{f}^{\leq \epsilon}  \right)
		\label{homology-computation-3}\\
		& =: GFH_{k}(\leg,f) 
		\label{homology-computation-4}
	\end{align}
where $\widetilde{H}$ denotes reduced homology.
Here, \eqref{homology-computation-1} is the definition of homology of a (pre)spectrum -- see Definitions~\ref{defn. homology of prespectrum} and~\ref{defn. homology of spectrum}. 
The isomorphism~\eqref{homology-computation-2} is a consequence of the fact that (for $i$ large enough) the maps $\Sigma X_i \to X_{i+1}$ in Definition~\ref{defn:leg-spectrum} are homotopy equivalences by Proposition~\ref{prop:leg-stab}; this renders the sequential colimit constant up to isomorphism, meaning the colimit is computed at any stage (which we take to be $i=0$). 
The isomorphism~\eqref{homology-computation-3} is a standard result from algebraic topology. See, for example, \cite[Proposition 2.22]{Hatcher}. Namely, the quotient map 
$$q: \left(\delta_f^{\leq \infin}, \delta_f^{\leq \epsilon}\right) \to
\left( \delta_f^{\leq \infin}/ \delta_f^{\leq \epsilon}, \delta_f^{\leq \epsilon}/ \delta_f^{\leq \epsilon} 
\right)$$
induces isomorphisms
$$H_n\left(\delta_f^{\leq \infin}, \delta_f^{\leq \epsilon}\right) \xrightarrow{\cong}
H_n \left( \delta_f^{\leq \infin}/ \delta_f^{\leq \epsilon}, \delta_f^{\leq \epsilon}/ \delta_f^{\leq \epsilon} \right)
\cong 
\widetilde{H}_n \left( \delta_f^{\leq \infin}/ \delta_f^{\leq \epsilon} \right), \quad \forall n.$$
Finally,~\eqref{homology-computation-4} is the definition of $GFH$ (Definition~\ref{defn:GF-homology}).
  \end{proof}
 
 \section{Lagrangian fillings and sheared difference functions}
Fix a Legendrian $\leg$ with a generating family $f$. In this section, we assume it is possible to extend $\leg$ to a Lagrangian filling $L$,  and it is  also possible to extend 
 $f$ by an appropriately compatible generating family $F$ for $L$. 
In this special situation, we show that the spectrum $C(\leg, f; \mathbb S)$ reflects the stable topology of the filling (Theorem~\ref{thm:suspension computation}). 
 Proving this involves defining, from $F$, 
  a ``sheared difference function'' and showing that  restricting this sheared difference function to particular domains 
recovers topological information of the filling. 
 This section heavily builds off the constructions in \cite[Section 4]{S-T:obstruct}. 
  
\subsection{Fillings}
 A Lagrangian filling of a Legendrian  can be viewed as an extension of a Legendrian $\leg \subset J^1B$  to a Lagrangian submanifold $L$ inside the symplectization of $J^1B$ -- the symplectic manifold $\rr \times J^1B$ with symplectic form $d(e^s(\alpha))$, where $\alpha = dz - ydx$ defines the contact structure on $J^1B$.

\begin{defn} \label{defn:fill} 
Fix a Legendrian $\leg \subset J^1 B$.
A \dfn{Lagrangian filling} of $\leg$ is a properly embedded Lagrangian submanifold $\sclag  \subset \rr \times J^1B$ such that, for some $s_-, s_+ \in \rr$,
$$\begin{aligned}
\sclag  \cap \left((-\infty, s_-] \times J^1B\right)  &= \emptyset, \quad \text{and }\\
\sclag \cap \left([s_+, \infty) \times J^1B\right) &=  [s_+, \infty) \times \leg.
\end{aligned}$$
 \end{defn}

\begin{rem}
 A Legendrian submanifold $\leg \subset J^1B$ gives rise to a {\bf Lagrangian
cylinder} $Z_\leg = \rr \times \leg$.
A Lagrangian filling, by definition, has a cylindrical end coinciding with $Z_\leg$.
\end{rem}

\begin{rem}\label{rem. can set s-plus to zero}
By applying a
translation in the $\rr$-coordinate of $\rr \times J^1B$, which is a conformal symplectic transformation and thus preserves Lagrangians, we can always assume $s_+ = 0$.
\end{rem}

\subsection{Moving to cotangent bundles}

We will apply the technique of generating families to study Lagrangian fillings.  To use this technique, we need to do a change of coordinates so that we are working in a cotangent bundle.
\begin{notation}
We let $(t,T)$ denote coordinates on $T^*\RR_{>0}$ (so $t >0$ and $T \in T^*_t \RR_{>0}$) and  $(q,p)$ denote local coordinates on $T^*B$ (so $p \in T^*_q B)$. Accordingly, we let 
	\eqn\label{eqn. tqup}
	(t,q,T,p)	
	\eqnd
be (local) coordinates on $T^*(\rr_{>0} \times B)$. We will utilize the following primitive 1-form:
	\eqnn
	\lambda_{0} = -Tdt -pdq.
	\eqnnd 
The derivative of $\lambda_0$ is (one convention for) the canonical symplectic form on $T^*(\rr_{>0} \times M)$.
\end{notation}

\begin{notation}
To study Lagrangian fillings using generating families, we identify $\rr \times J^1B$ with  $T^*(\rr_{>0} \times B)$ by the symplectomorphism
\begin{equation} \label{eqn:id}
  \begin{split}
    \theta\co \rr \times J^1B &\to T^*(\rr_{>0} \times B) \\
    (s,x,y,z) &\mapsto (e^s, x, z, e^sy).
  \end{split}
\end{equation}
(See~\eqref{eqn. tqup} for the coordinates on the codomain.) A direct calculation shows that $\theta^*(\lambda_{0}) = e^s\alpha + df$, where $f:\rr \times J^1B \to \rr$ is given by $f(s,x,y,z) = -e^sz$, and thus $\theta$ preserves exact Lagrangian submanifolds.  We let 
	\eqn\label{eqn. clat and sclag}
	\clag := \theta(\sclag).
	\eqnd
 We relabel 
 	\eqnn
	e^{s_-} = t_-
	\qquad
	\text{and}
	\qquad
	e^{s_+} = t_+.
	\eqnnd
\end{notation}

\begin{rem} \label{rem:cylinderical-set t-plus to zero}
Observe that the cylindrical end of $\sclag$ becomes a 
conical end for $\clag$: the  non-varying  $\{s = \text{constant}\}$ Legendrian slices of $\sclag$ are mapped to 
$\{t = \text{positive constant}\}$ slices of $\clag$ with projections to $T^*B$ whose $p$-coordinates  expand with $t$. 
By Remark~\ref{rem. can set s-plus to zero}, we can always assume that $\clag$ is conical on $\{ t >  t_+ = 1 \}$. 
\end{rem}

For a Lagrangian filling $\sclag$ of $\leg$,  
we will be interested in the situation where $\ \clag = \theta(\sclag)  \subset T^*(\rp \times B)$  has a generating family $F$ that is an ``extension" of
a generating family $f$ for $\leg$ in the following sense.

\begin{defn} \label{defn:compatible}   
Suppose $\sclag$ is a Lagrangian filling of $\leg$ that is cylindrical over $\leg$ for $s \in [0,\infty)$ (see Remark~\ref{rem. can set s-plus to zero}),
 $f\co B \times \rr^N \to \rr$ is a  linear-at-infinity generating family for $\leg$, and
$F\co  (\rp \times B) \times \rr^N \to \rr$ is a generating family for $\clag$.
We then say that $(\sclag,F)$ {\bf is a  filling of} $(\leg,f)$ if there exists $0< t_- < 1$ such that
$$F(t, x, \eta)  = 
\begin{cases}
    t  f (x, \eta), &t \geq 1 \\
  t A(\e), &t \leq t_-, 
  \end{cases}$$ 
  where   $A(\e)$ is a non-zero linear function.
 Furthermore, we will say that $(\sclag,F)$ is a {\bf linearly-controlled} filling if 
  there exists a compact set $K \subset B \times \rr^N$ with complement $K^c$ such that
   $$F|_{(0, \infty) \times K^c} = t A(\e).$$
  \end{defn}

\begin{rem}\label{rem:sli}  
 In the terminology of \cite[Definition 4.3]{S-T:obstruct}, if $(\sclag,F)$ is a linearly-controlled filling of $(\leg, f)$, then 
 $F$ is ``slicewise linear-at-infinity."  This analytic condition will guarantee that 
Morse-theoretic arguments from Section~\ref{ssec:morse-theory} will apply. 
\end{rem}
 
\begin{rem}
Just as we generalized Definition~\ref{defn. classical linear-at-infinity} to Definition~\ref{defn. linear-at-infinity}, there is a natural reformulation of linear control that is preserved under fiberwise diffeomorphisms. We do not pursue this here.
\end{rem}

\subsection{Sheared difference functions} We saw in Section~\ref{ssec:diff} that the difference function $\delta_f$
associated to a generating family $f$ of a Legendrian $\leg$ captures the dynamically important Reeb chords of $\leg$. 
For our Lagrangian  $\clag \subset T^*(\rp \times B)$ with a conical end over the Legendrian $\leg$, 
we will be able to capture the topology of the filling 
\begin{equation} \label{eqn:cpct-filling}
L_0 = \sclag \cap \{s \leq 0\}
\end{equation}
 {\it and} the Reeb chords in the Legendrian boundary $\leg$
through a 	``sheared" difference function   denoted as $\Delta_F$. 
This will be the sum of the standard difference function associated to a
generating family $F(t,x,\e)$ for $\clag$ and a Hamiltonian $H(t)$.   
 The following definition is \cite[Definition 4.4]{S-T:obstruct} simplified since we are assuming $\sclag$ is a filling that is cylindrical for  $s \in [0,\infty)$.

\begin{choice}[$u$]\label{choice. u for shearing}
Choose $u$ such that
$$ 1 < u <  \min 
\left\{   \sqrt{ 1 + \lmin\,}, 2 
\right\}.$$
(For $\lmin$, see Notation~\ref{notation. lmin lmax}.)
\end{choice}

\begin{rem} \label{rem:sigma-possible}
  The upper bounds on $u$ will be used in the proofs of Lemmas~\ref{lem:Infin-mu} and \ref{lem:lambda-sigma}, where some analysis is done
for functions that will allow us to see $(\Delta_F^{\leq \Infin}, \Delta_F^{\leq -\mu})$ as a relative mapping cone. In particular, we will use that $u <  \sqrt{1 + \lmin\,}$ implies that
$0 < u^2 -1  <  \lmin$. 
\end{rem}

\begin{defn}[Shearing Functions] \label{defn:shear-H} 
Fix a Lagrangian filling  $\sclag$ of $\leg$  and let $\clag \subset T^*(\rp \times B)$ denote the corresponding Lagrangian~\eqref{eqn. clat and sclag}. We assume $t_+ = 1$ as in 
Remark~\ref{rem:cylinderical-set t-plus to zero} and that we have fixed $u$ as in Choice~\ref{choice. u for shearing}. We then let  
	\eqnn
	\mathcal{H}\left(\clag \right)
	\eqnnd
be the set of 
  decreasing, smooth functions $H\co \rr_{>0} \to \rr$ 
  satisfying  
  	\eqnn
  	H(t) = 
  	\begin{cases}
        0, & t \leq 1\\
    	- \frac{1}{2}( t - 1)^2, & t \geq u;
  	\end{cases}
  	\qquad\text{and}
	\qquad
  	\text{$H''(t) < 0$ on $(1,u)$}.
	\eqnnd
See Figure~\ref{fig:H-dH}.  We will call any $H \in \mathcal{H}(\clag)$ a \dfn{shearing function}.
\end{defn}

\begin{figure}
 \centerline{\includegraphics[height=1in]{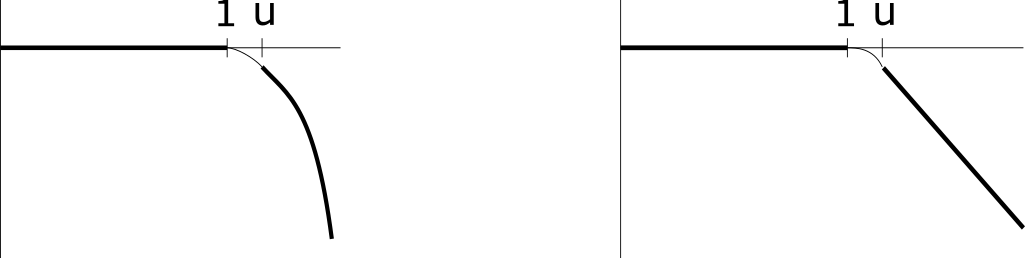}}
  \caption{A schematic picture of a shearing function $H(t)\in \mathcal H(\clag)$ and its derivative $H'(t)$.
  }
  \label{fig:H-dH}
\end{figure}

For $H \in \mathcal{H}\left( \clag\right)$, let $X_H$ denote the associated Hamiltonian vector field, using the convention $\iota_{X_H} \infin = -dH$. If $\phi_H^1$ denotes the time-1 flow of this vector field and $F$ generates $\clag$, then   $F(t,x,\e) + H(t)$ generates $\phi_H^1(\clag)$.   
In parallel to the definition of the difference function $\delta_f$ in Definition~\ref{defn:leg-diff}, a shearing function $H \in \mathcal{H}\left( \clag \right)$ may be used to define a ``sheared" difference function:

\begin{defn} \label{defn:pert-diff} Suppose $(\sclag,F)$ is a filling of $(\leg, f)$, where $f$ is linear-at-infinity.  Then 
given  $H \in \mathcal{H}\left( \clag \right)$, the \dfn{sheared difference function} $\Delta_F\co \rp \times B \times \rr^{N} \times \rr^{N} \to \rr$ is defined as: \begin{equation}  
  \Delta_F(t,x,\e,\te) = F(t,x,\te) + H(t) - F(t,x,\e).
\end{equation}
\end{defn}

\begin{rem}
  We may apply a  fiber-preserving diffeomorphism so that $f(x, \e)$ agrees with the linear function $A(\e)$ outside a compact set.  Having done this, 
observe that for any filling $(\sclag,F)$ of $(\leg,f)$,  
\begin{equation} \label{eqn:compat} 
\Delta_F(t,x,\e,\te) = \begin{cases}
 t \delta_f (x, \e, \te) + H(t), & t \geq 1\\
 t A(\e,\te), & t \leq t_-,
\end{cases}
\end{equation}
where $\delta_f$ is the difference function  for $(\leg, f)$, and $A(\e,\te) = A(\te) - A(\e)$ is a non-zero linear function. 
\end{rem}

In parallel to Proposition~\ref{prop:leg-crit-point},  the critical points of $\Delta_F$ detect information about the intersection points of
$\clag$ and $\phi^1_H({\clag})$:

\begin{prop} \label{prop:crit-pt-wgh} \cite[Proposition 4.5]{S-T:obstruct}  
Suppose that $(\sclag, F)$ is a linearly-controlled filling of $(\leg, f)$ (Definition~\ref{defn:compatible}) and  $H \in \mathcal{H}(\clag)$ (Definition~\ref{defn:shear-H}). 
Then 
\begin{enumerate}
	\item There is a one-to-one correspondence between intersection points in ${\clag} \cap \phi^1_H\left({\clag}\right)$ and critical points of $\Delta_F$.  

	\item Moreover, 
	\eqn\label{eqn. reeb chords via sheared function}
	\left( {\clag} \cap \phi^1_H\left({\clag}\right) \cap \{ t \in (u, \infty)\} \right)
	=
	\left( {\clag} \cap \phi^1_H\left({\clag}\right) \cap \{ t \in [u, \lmax + 1]\}\right),
	\eqnd
and there is a one-to-one correspondence between Reeb chords $\gamma$ of $\leg$ and the set~\eqref{eqn. reeb chords via sheared function}. In fact, the critical value of the point corresponding to the Reeb chord $\gamma$ has $t$-coordinate given by $\ell(\gamma)+1$, where $\ell(\gamma)$ is the length of the Reeb chord (Notation~\ref{notation:length}), and critical value
$$ \ell(\gamma) + \frac{\ell(\gamma)^2}{2} > 0.$$
	\item All other critical points 
lie in the critical submanifold with boundary $$C=\left\{(t,x,\e,\e) \;:\; (t,x,\e) \in \Sigma_F \text{ with }
    t \in [t_-, 1] \right\};$$   $C$ is diffeomorphic 
      to $L_0 = \sclag \cap \{ s \in [s_-, 0] \}$,  has critical value $0$, and, for generic $F$, is non-degenerate of index $N$.  
\end{enumerate}
\end{prop}	

\begin{rem} \label{rem:Reeb-crit-pts}
 Calculations, as shown in the proof of \cite[Proposition 4.5]{S-T:obstruct},  show that one gets a critical point corresponding to the Reeb chord with length $\ell(\gamma)$ when
$-H'(t) = \ell(\gamma)$.
 \end{rem}

\subsection{Sublevel spaces over the conical end}
From Proposition~\ref{prop:crit-pt-wgh}, we understand the critical values of $\Delta_F$. 
The overall strategy of this section is to carefully choose positive constants $\mu, \Omega$ such that it is possible to  realize the pair $(\Delta_F^{\leq \Infin}, \Delta_F^{\leq -\mu})$ as a relative mapping cone.  To do this,
we will argue that over $[u, \infty)$, the pair of spaces can be identified with the relative cone on $(\delta_f^{\leq \infin}, \delta_f^{\leq \epsilon})$.

\begin{choice}[$\Infin$, $\mu$] \label{choice:mu-Infin} 
Given a linearly-controlled filling $(\sclag,F)$ of $(\leg, f)$, 
for $H \in \mathcal H(\clag)$, choose $\mu, \Infin > 0$    such that
\begin{equation}   \label{eqn:mu-Omega-bounds}
\begin{aligned}
0<   &\mu <    \min \left\{\lmin, \frac{(u-1)^2}{2} \right\},\\
    \lmax + \frac{\lmax^2}{2} < & \Infin.
    \end{aligned}
\end{equation}
(For $\lmin, \lmax$, see Notation~\ref{notation. lmin lmax}.)
\end{choice}

\begin{notation}
It will be convenient to work over subsets corresponding to intervals in the $\rp$-coordinate.
For $J  \subset \rp$, we use the shorthand
  \begin{gather*} \label{eqn:Delta-restrict}
    {\Delta_F}|_J = \Delta_F|_{\{(t,x,\e, \te)\;:\; t \in J\}}, \\
\Delta_F^{\leq \alpha}|_J = \Delta_F^{\leq \alpha} \cap \{t \in J\} = \left\{(t,x,\e,\te)\;:\; t \in J,\, \Delta_F(t,x,\e,\te) \leq \alpha \right\}.
\end{gather*}
\end{notation}

\begin{notation}[$\lambda_\alpha(t)$]
To identify the fibers of $\Delta_F^{\leq \alpha}|_J$ over $t \in J \subset [1, \infty)$, consider the function
\begin{equation}\label{eqn:lambda}
\lambda_\alpha(t) := \frac{1}{t}(\alpha-H(t)).
\end{equation}
We call $\lambda_\alpha(t)$ the {\bf  $\alpha$-level $\Delta$-$\delta$ translation function}.
\end{notation}

The name for this function is explained by the following remark.

\begin{rem}
Since $\clag$ is conical over $\{t \in [1,\infty)\}$, and  $F$  is a conical extension of $f$, if $J \subset [1, \infty)$,    Equation~(\ref{eqn:compat})  shows that we have:
\begin{equation} \label{eqn:end-level}
\begin{aligned}
  \Delta_F^{\leq \alpha}|_J &= \left\{ (t,x,\e,\te) \;:\; t \in J,\  \delta_f(x,\e,\te) \leq \frac{1}{t}\left(\alpha - H(t)  \right) = \lambda_\alpha(t) \right\}. \\  
\end{aligned}
\end{equation}
That is, the fiber of $\Delta_F^{\leq \alpha}|_J$ above $t \in J$ is 
	\eqnn
	\Delta_F^{\leq \alpha}|_{t \in J} = \{t\} \times \delta_f^{\leq \lambda_{\alpha}(t)}
	\cong
	\delta_f^{\leq \lambda_{\alpha}(t)}.
	\eqnnd
\end{rem}

We will do some basic analysis to understand the $\Delta$-$\delta$ translation functions
$\lambda_\alpha(t)$ for $t \geq 1$, when $\alpha = \Infin, -\mu$ (Choice~\ref{choice:mu-Infin} ).

\begin{lem} \label{lem:Infin-mu} 
  The $\Delta$-$\delta$ translation functions 
have limiting behavior  $$\lim_{t \to \infty} \lambda_{\alpha}(t) = \infty,$$
for any $\alpha$.  For any $\alpha < 0$, $\lambda_{\alpha}(t)$ is increasing on $(0, \infty)$. 
Furthermore, for $u$ as in Choice~\ref{choice. u for shearing} and $\Omega, \mu$ as in Choice~\ref{choice:mu-Infin}, we have:
\begin{enumerate}
	\item\label{item. Omega regular value} $\lambda_\Infin(t)  > \lmax$, for all $t \in (0,\infty).$
	\item\label{item. mu regular value} $-\lmin < \lambda_{-\mu}(1) <  0 <  \lambda_{-\mu}(u) < \lmin$.
 \end{enumerate}
 \end{lem}

\begin{proof}  Direct calculations show that, for any $\alpha$, by construction of $H(t)$ in Definition~\ref{defn:shear-H}, 
$$\lim_{t \to \infty} \lambda_{\alpha}(t) = \lim_{t \to \infty}\frac{\alpha - H(t)}{t}  = \infty,$$
 and the derivative of $\lambda_\alpha(t)$ is
\begin{equation}\label{eqn:lambda-prime}
\lambda_\alpha'(t) = \frac{H(t) - t H'(t) - \alpha}{t^2}.
\end{equation}
By construction of $H$, 
\begin{equation}\label{eqn:H-Hprime}
H(t) - tH'(t) \geq 0, \quad t \in (0, \infty),
\end{equation}
since $H(t) - t H'(t)= 0$ on $t \leq 1$, and $H(t) - t H'(t)$ is strictly increasing on $t > 1$.
 To see this last statement, observe that 
$\frac{d}{dt}\left(H(t) - tH'(t)\right) = -tH''(t)$,  
and since $H''(t) < 0$ for $t > 1$,  $H(t) - tH'(t)$  is strictly increasing on $(1,\infty)$. 

When $\alpha = -\mu < 0$,  since $H(t) - tH'(t) \geq 0$, we see 
that $\lambda_{-\mu}'(t) > 0$, and thus 
$\lambda_{-\mu}$ is  strictly increasing on $(0, \infty)$. Furthermore,  Choice~\ref{choice:mu-Infin} guarantees 
 $$-\lmin < \lambda_{-\mu}(1) = -\mu < 0, \quad \text{ since } 0 <  \mu < \lmin.$$
 We also know that by Choice~\ref{choice. u for shearing} and Choice~\ref{choice:mu-Infin} that
 $$\begin{aligned}
 \lambda_{-\mu}(u) = \frac{1}{u} (-\mu - H(u)) 
 &=  \frac{1}{u} \left(-\mu + \frac{(u-1)^2}{2} \right) \\
  &<   \left(-\mu + \frac{u-1}{2} \right), \quad \text{ since } 1< u < 2,   \mu < \frac{(u-1)^2}{2}  \\
   &< \left(-\mu + \frac{\lmin}{2} \right), \quad \text{ since } u <    \sqrt{\lmin + 1\,} < \lmin + 1  \\
  &<  \frac{\lmin}{2}, \quad \text{ since } \mu  > 0\\
  &< \lmin.
 \end{aligned} $$

When $\alpha = \Infin$, we find that $\lambda_\Infin(t)$ is decreasing when $H(t) - tH'(t)< \Infin$,  and increasing when $H(t) - tH'(t) > \Infin$. Since $H(t) - tH'(t)$ is strictly
increasing when $t > 1$, there will be a unique $t_c > 1$ such that $H(t_c) - tH'(t_c) = \Infin$.  We want to show that $\lambda_{\Infin}(t_c) > \lmax$.  Since $H(t_c) - tH'(t_c) = \Infin$,
$$\lambda_{\Infin}(t_c) = \frac{\Infin- H(t_c)}{t_c} = -H'(t_c).$$
As mentioned in Remark~\ref{rem:Reeb-crit-pts},  at $\overline t := \lmax + 1$, $-H'(\overline{t}) = \lmax$.  Since  $-H'(t)$ is an increasing function, it suffices to show that
$\overline{t} < t_c$. We next use the fact that $H(t) - tH'(t)$ is an increasing function to argue that $t_c > \overline t := \lmax + 1$. 
By construction of $H$, on $(u, \infty)$,
 $H(t) - tH'(t) = \frac12(t+1)(t-1)$, and thus, since, by Choice~\ref{choice:mu-Infin}, $\lmax + \frac{\lmax^2}{2} < \Infin$,
 $$H(\overline t) - tH'(\overline t) =  \frac12(\lmax + 2)\lmax = \lmax + \frac{\lmax^2}{2} < \Infin.$$
 Thus $u < \overline t = \lmax + 1 <  t_c$, and, since $-H'(t)$ agrees with  $t-1$ on $t > u$,  
 $$\lambda_{\Infin}(t_c) =  -H'(t_c) > -H'(\overline t) = \overline t -1 = \lmax,$$
 as desired.
      \end{proof}

 On the path to showing that $(\Delta_F^{\leq \Infin}, \Delta_F^{\leq -\mu})$ is a mapping cone, the following lemma will be used to 
understand sublevel sets over $\{t \in [1,u]\}$.    Recall Remark~\ref{rem:sigma-possible} tells us that $0 < u^2 -1 <  {\lmin}$.

 \begin{lem}\label{lem:lambda-sigma} If $\sigma$ is chosen such that
 $u^2 -1  < \sigma < {\lmin}$, then for  $t \in [1,u]$, $\lambda_{\sigma}$ is decreasing, and $0 <\lambda_{\sigma}(t) < {\lmin}$.
 \end{lem}

 \begin{proof}  By Equation~(\ref{eqn:lambda-prime}), to show that $\lambda_\sigma$ is decreasing, it suffices to show that 
 $$H(t) - tH'(t) < \sigma.$$
 As in the argument for Equation~(\ref{eqn:H-Hprime}), $H(t) - tH'(t)$ is increasing on $[1, u]$. Furthermore, by the construction of $H$ on $[u, \infty)$, and
 the  hypothesis that
 $u^2 -1< \sigma$, we have that 
 $$H(u) - uH'(u) = \frac12 (u+1)(u-1)< u^2 - 1 < \sigma.$$
 Thus $\lambda_\sigma(t)$ is strictly decreasing on $[1,u]$, and $\lambda_\sigma(t)$  obtains a maximum at $1$ with value
 $$\lambda_\sigma(1) = \sigma < {\lmin}, $$
 and a minimum at $u$, which satisfies 
 $$\lambda_\sigma(u) = \frac{1}{u}\left( \sigma + \frac{1}{2}(u-1)^2 \right) > 0.$$
  \end{proof}

\subsection{Important pairs associated to a filling}
We will be interested in studying the pair  $\left(\Delta_F^{\leq \Infin}, \Delta_F^{\leq -\mu}\right)$, where $\Infin, \mu$ satisfy the inequalities in Choice~\ref{choice:mu-Infin}.
We can apply the analysis of the $\Delta$-$\delta$ translation functions $\lambda_\Infin$, $\lambda_{-\mu}$ to understand this pair on  $\{t = u\}$, $\{t \geq u\}$, $\{t \leq u\}$, as well as the entire domain $\{t \in (0, \infty)\}$.

\begin{lem} \cite[Lemma 6.2]{S-T:obstruct} \label{lem:pos-end} 
\begin{enumerate}
\item There is a diffeomorphism of pairs
$$
    \left(\Delta_F^{\leq \Infin}|_{\{u\}}, \Delta_F^{\leq -\mu}|_{\{u\}} \right) \cong \left(\delta_f^{\leq \infin}, \delta_f^{\leq \epsilon}\right).    $$
\item Moreover, for all $v$ sufficiently large, there is a homotopy equivalence 
 $$
    {\rho}\co \left(\Delta_F^{\leq \Infin}|_{[u, \infty)}, \Delta_F^{\leq -\mu}|_{[u, \infty)}\right) \to 
    \left(\delta_{f}^{\leq \infin} \times [u,v], \delta_f^{\leq \epsilon} \times [u,v] \cup \delta_f^{\leq \infin} \times\{v\} \right),   
$$
  with ${\rho}|_{\Delta_F^{\leq \Infin}|_{\{u\}}}  \simeq \id$.
  \end{enumerate}
  \end{lem}

\begin{proof}   As mentioned in Equation~(\ref{eqn:end-level}),
 $$\Delta_F^{\leq \alpha}|_{[u, \infty)} = \left\{ \left(t, \delta_f^{\leq \lambda_{\alpha}(t)} \right): t \in [u,\infty) \right\},$$    
 for $\lambda_{\alpha} (t)$ as defined in Equation~(\ref{eqn:lambda}).
Now we use our analysis of the functions $\lambda_{-\mu}(t)$ and $\lambda_{\Infin}(t)$
  for $t \geq u$.  By Lemma~\ref{lem:Infin-mu},  $$0 < \lambda_{-\mu}(u) < \lmin, \qquad \lmax  < \lambda_{\Infin}(u),$$
  which gives rise to the diffeomorphism
 $$ \left(\Delta_F^{\leq \Infin}|_{\{u\}}, \Delta_F^{\leq -\mu}|_{\{u\}} \right) \cong \left(\delta_f^{\leq \infin}, \delta_f^{\leq \epsilon}\right).$$
 Furthermore, by Lemma~\ref{lem:Infin-mu} we  know that $\lambda_{-\mu}(t)$ is strictly increasing on $t \geq u$, and,  for $v$ sufficiently large,
  $\lmax < \lambda_{-\mu}(v)$.  After applying some fiberwise
  homotopy equivalences, as in \cite[Lemma 5.8]{S-T:obstruct}, we can apply \cite[Lemma 5.6]{S-T:obstruct} to construct a deformation retraction 
  $$ {\rho}\co \left(\Delta_F^{\leq \Infin}|_{[u, v]}, \Delta_F^{\leq -\mu}|_{[u, v]}\right)  \to 
  \left(\delta_{f}^{\leq \infin} \times [u,v], \delta_f^{\leq \epsilon} \times [u,v] \cup \delta_f^{\leq \infin} \times\{v\} \right).$$
  \end{proof}

 Our next lemma studies the pair $\left(\Delta_F^{\leq \Infin}, \Delta_F^{\leq -\mu}\right)$ on $\{t \leq u\}$.  Here we see that the pair can be
 identified with a pair that can be identified with a quotient of a trivial disk bundle over the Lagrangian filling $L$. For the hypothesis of this lemma, recall
 that by Remark~\ref{rem:sigma-possible} our restrictions on $u$ guarantees that $ 0 < u^2 -1 < {\lmin}$.
   Portions of the proof of the following lemma employ standard Conley index theory, \cite{Conley}.  

\begin{lem}   \cite[Lemma 6.3, Lemma 6.5]{S-T:obstruct} \label{lem:neg-end}
For any $\sigma > 0$ such that
\begin{equation}
  u^2 -1  < \sigma < {\lmin},
\end{equation}
\begin{enumerate}
\item there exists a deformation retraction
$${\rho}\co \left(\Delta_F^{\leq \Infin}|_{(0, u]}, \Delta_F^{\leq -\mu}|_{(0, u]}\right)
    \to 
    \left(\Delta_F^{\leq \sigma}|_{(0, u]},  \Delta_F^{\leq -\mu}|_{ (0, u]} \right).
    $$

\item Let $D^N$ denote a trivial $N$-dimensional disk bundle over $L_0$ (Equation~\ref{eqn:cpct-filling}) and $S^{N-1}$ the associated sphere bundle.
Then there exists a homotopy equivalence 
    $$ \left(\Delta_F^{\leq \sigma}|_{(0, u]},  \Delta_F^{\leq -\mu}|_{ (0, u]} \right) 
    \simeq \left( D^N(L_0),  S^{N-1}(L_0) \cup D^N(\partial L_0) \right).$$
\end{enumerate}
\end{lem}

\begin{proof}  
Fix a    product Riemannian structure  on $(0, u] \times B \times \rr^{2N}$. 
 The idea for the map $\rho$ is to follow the negative gradient vector field of $\Delta_F$ until
we reach level $\sigma$.
We need to be sure that this vector field is integrable on our domain, which amounts to
checking that the vector field is parallel to or inward-pointing along the sets $\{t = w \} \cup \{t = u\}$,
for all sufficiently small $w>0$ .

Fix $w$ such that  $0 < w < t_-$.  
As noted in Equation~(\ref{eqn:compat}), for $t < t_-$,  $\Delta_F(t,x,\e,\te) = tA(\e, \te)$, where $A(\e, \te)$ is a non-zero linear function,
and for $t \geq 1$, $\Delta_F(t,x,\e,\te) = t\delta_f(x,\e,\te) + H(t)$. 
Thus we have
$$ \grad \Delta_F(t, x, \e, \te)   = \begin{cases}
 A(\e,\te)   \pd{}{t} + t \grad A(\e,\te), & t \leq t_- \\
 \left( \delta_{f}(x,\e,\te) + H'(t) \right) \pd{}{t} + t \grad \delta_f(x,\e,\te), & t \geq 1.
\end{cases}$$
 On $\{t \in (0, t_-]\}$, we will modify the gradient of $\Delta_F$ to one that is integrable   
by ``removing" the $\pd{}{t}$ portions as we approach the set  $\{t \in (0, w]\}$.
   Choose $\tau \co (0,1] \to [0,1]$ to be a smooth function with $\tau|_{[t_-, 1]} = 1$
 and $\tau^{-1}\{ 0 \} = (0, w]$.
 Then let $X$ be the vector field
 $$ X (t, x, \e, \te)  = \begin{cases}
 \left( \tau(t) A(\e,\te)  \right) \pd{}{t} + t \grad A(\e,\te), & t \leq t_- \\
\grad \Delta_F(t, x, \e, \te), &t \geq t_-.\\
\end{cases}$$
By construction, $X$ is a gradient-like vector field for $\Delta_F|_{(0, 1]}$ when
  $t \in [t_-, 1]$.  
   When $t \leq t_-$, 
 $$
 \begin{aligned}
 \langle X, \grad \Delta_F \rangle 
 &= \langle X,      A(\e,\te)   \pd{}{t} + t \grad A(\e,\te)      \rangle \\
& = 
 \tau(t) \left( A(\e,\te)  \right)^2 + t^2 \| \grad A \|^2 \geq 0.
 \end{aligned}$$
Since the non-zero linear function $A$ will not have any critical points, we see that 
 $X$ is a gradient-like vector field for $\Delta|_{(0, t_-]}$. 

We cannot apply this argument to modify the gradient of $\Delta_F$ near $\{ t = u \}$ since
 $\delta_f$ will have critical points.  Here we can do a direct argument using the 
 assumption that $  u^2 - 1 < \sigma$. When $t \in [1, u]$,
  $\Delta_F(t,x,\e,\te) = t\delta_f(x,\e,\te) + H(t)$, and due to the convexity condition on $H(t)$
$$\frac{\partial}{\partial t} \Delta_F(t,x,\e,\te) = \delta_f(x,\e,\te) + H'(t) > \delta_f(x,\e,\te) - (u-1).
$$
When $\Delta_F(t,x,\e,\te) > \sigma$, and $t \in [1,u]$,  since $H(t) \leq 0$, we have that
$$t \delta_f(x,\e,\te) + H(t) > \sigma \implies t\delta_f(x,\e,\te) > \sigma.$$
Thus 
$$\delta_f(x,\e,\te) > \frac{\sigma}{t} \geq \frac{\sigma}{u}, \quad \text{ since }  1 \leq t \leq u.$$
Thus we find that on $\{t \in [1,u]\}$,
$$
\frac{\partial}{\partial t} \Delta_F(t,x,\e,\te) > \delta_f(x,\e,\te) - (u-1) \geq  \frac{\sigma}{u} - (u-1),$$
which will be positive since $\sigma > (u+1)(u-1) > u(u-1)$.  
The positivity of $\frac{\partial}{\partial t} \Delta_F$ on $\{ t \in [1,u]\}$ guarantees that $-\grad \Delta_F$ is
inward pointing on $\{t = u\}$.
\vskip .1in
  Lastly we will sketch how to show that if $\sigma < \lmin$, there is a homotopy equivalence
    \begin{equation*} \left(\Delta_F^{\leq \sigma}|_{(0, u]},  \Delta_F^{\leq -\mu}|_{(0, u]} \right) 
    \simeq \left(D^N(L_0),  \left(S^{N-1}(L_0) \cup D^N(\partial L_0)\right) \right).
    \end{equation*}
    As an overview of the strategy, we first show that there is a homotopy equivalence 
    \begin{equation}\label{eqn:step1} \left(\Delta_F^{\leq \sigma}|_{(0, u]},  \Delta_F^{\leq -\mu}|_{ (0, u]} \right)
    \simeq 
    \left(\Delta_F^{\leq \sigma}|_{[t_-, 1]},  \Delta_F^{\leq -\mu}|_{ [t_-, 1]} \cup \Delta_F^{\leq \sigma}|_{\{1\}} \right)
   \end{equation}
and then apply a Morse-Bott argument to construct a homotopy equivalence   
\begin{equation} \left(\Delta_F^{\leq \sigma}|_{[t_-, 1]},  \Delta_F^{\leq -\mu}|_{ [t_-, 1]} \cup \Delta_F^{\leq \sigma}|_{\{1\}} \right)
\simeq
\left(D^N(L_0),  \left(S^{N-1}(L_0) \cup D^N(\partial L_0)\right) \right).
\end{equation}
To verify (\ref{eqn:step1}), first observe that since $L$ is a Lagrangian filling, on $\{t \in (0,t_-]\}$, $\Delta_F(t,x,\e,\te) = tA(\e,\te)$.  
  We now apply the analysis of the $\lambda_{\sigma}, \lambda_{-\mu}$ functions from Lemmas~\ref{lem:lambda-sigma} and \ref{lem:Infin-mu}.
After doing fiberwise flows,  the arguments \cite[Lemma 5.4, Corollary 5.5]{S-T:obstruct} show that there is
 a deformation retraction
$$\left(\Delta_F^{\leq \sigma}|_{(0, t_-]},  \Delta_F^{\leq -\mu}|_{ (0, t-]} \right) \to \left(\Delta_F^{\leq \sigma}|_{\{t_-\}},  \Delta_F^{\leq -\mu}|_{\{t_-\}} \right).$$
For $t \in [1,u]$, after applying a fiberwise homotopy equivalence, we can apply 
\cite[Lemma 5.6]{S-T:obstruct} to construct a deformation retraction
 $$\left(\Delta_F^{\leq \sigma}|_{[1, u]},  \Delta_F^{\leq -\mu}|_{ [1,u]} \right) \to
\left(\Delta_F^{\leq \sigma}|_{\{1\}} \times [1,u],  \Delta_F^{\leq -\mu}|_{\{1\}} \times [1,u] \cup  
\Delta_F^{\leq \sigma}|_{\{u\}}\right).$$
Again by the analysis of $\lambda_{\sigma}$, we see that we have a homotopy equivalence 
$$\left(\Delta_F^{\leq \sigma}|_{\{1\}} \times [1,u],  \Delta_F^{\leq -\mu}|_{\{1\}} \times [1,u] \cup  
\Delta_F^{\leq \sigma}|_{\{u\}}\right)
\simeq \left(\Delta_F^{\leq \sigma}|_{\{1\}},  \Delta_F^{\leq \sigma}|_{\{1\}} \right).
$$
Combining the above analysis for $t < t_-$ and $t > 1$ gives the desired homotopy equivalence stated in (\ref{eqn:step1}). 

Now we apply a Morse-Bott argument to analyze the topology as we pass through the critical level $0$ on the
way up from $\Delta_F^{\leq -\mu}|_{[t_-, 1]}$ to $\Delta_F^{\leq \sigma}|_{[t_-, 1]}$. Recall that there is a 
non-degenerate critical
submanifold with boundary  
$(C,\partial C) \subset \left( \{ t \in [t_-, 1] \}, \{ t = 1\} \right)$, which is diffeomorphic to $(L_0,\partial L_0)$, of index $N$ and critical value $0$. 
We employ a simple modification of the standard constructions of Morse-Bott theory to allow for critical submanifolds with boundary.
The argument in \cite[Lemma 6.5]{S-T:obstruct} explains that  the effect of passing through the critical level is to attach an $N$-disk  bundle over $C$ to $\Delta^{-\mu}_{[t_-, 1]}$ along its unit sphere bundle, and this $N$-disk bundle is isomorphic to the negative-eigenvalue bundle associated to the Hessian of $\Delta_F$, which by Corollary~\ref{cor:neg-index-trivial}, is trivial.
We thus obtain a homotopy equivalence between the pairs
  $$\left(\Delta^{\sigma}_{[t_-, 1]}, \Delta^{-\mu}_{[t_-, 1]} \cup \Delta_{\{1\}}^\sigma \right) \simeq 
   \left(D^N(L_0), S^{N-1}(L_0) \cup D^N(\partial L_0)\right).$$
\end{proof}

Our last lemma tells us that on the full domain $\{t \in (0, \infty)\}$,  our pair
$\left(\Delta_F^{\leq \Infin}|_{(0,\infty)}, \Delta_F^{\leq -\mu}|_{(0, \infty} \right)$ is a ``trivial" pair.

\begin{lem}   \label{lem:contractible pair for (-oo,oo)}
 There is a deformation retraction of  $\Delta_F^{\leq \Infin}|_{(0,\infty)}$ to 
 $\Delta_F^{\leq -\mu}|_{(0,\infty)}$.
\end{lem}

\begin{proof}   
We can write $f(x, \e) = f^c(x, \e) + A(\e)$, where $f^c(x, \e)$ is compactly supported. Then 
$$\Delta_F(t, x, \e, \te) = D^c_t(x, \e, \te) + tA(\e, \te) + H(t),$$
where $D^c_t: B \times \rr^N \times \rr^N \to \rr$ is compactly supported, for all $t$,  vanishes when $t \leq t_-$ and agrees with $t(f^c(x, \te) - f^c(x,\e))$ for $t \geq 1$,
and $A(\e,\te) = A(\te) - A(\e)$. 
Then consider 
$$\Delta_s(t, x, \e, \te) = (1-s) D^c_t(x, \e, \te) + tA(\e, \te) + H_{s}(t),$$
for $H_{s}(t) \in \mathcal{H}\left(\clag_s \right)$ chosen with respect to the  (singular) Lagrangian $\clag_s$ generated by $(1-s)(D^c_t(x, \e)) + A(\e)$.
Choose paths $\Infin_s$ and $\mu_{s}$ such that $\Infin_0 = \Infin$, $\mu_0 = \mu$, and all critical values of $\Delta_s$ lie in $[-\mu_s, \Infin_s]$.  Notice that $\Delta_1(t, x, \e, \te) = tA(\e, \te) + H_{1}(t)$, and hence has no critical values.  As in the proofs of Lemma~\ref{lem:pos-end} and Lemma~\ref{lem:neg-end}, we choose $w, v \in \rr_{>0}$ such that
$0< w < t_-$ and $u < v$ satisfies $\lambda_{-\mu}(v) > \lmax$.
If we can show that there exists 
 an integrable, gradient-like vector field $X_s$ for $\Delta_s$ on $[w, v] \times B \times \rr^{2N}$, then the Critical Non-Crossing Lemma~\ref{lem:crit-non-crossing} implies that 
$$
 \left((\Delta_0)_{[w, v]}^{\leq \Infin_0}, (\Delta_0)_{[w, v]}^{\leq -\mu_0} \right) \\
\simeq \left((\Delta_1)_{[w, v]}^{\Infin_1}, (\Delta_1)_{[w, v]}^{-\mu_1} \right). $$
 Then the fact that $\Delta_1$ has no critical values implies
 $$
 \left((\Delta_1)_{[w, v]}^{\Infin_1}, (\Delta_1)|_{[w, v]}^{-\mu_1} \right) \simeq
\left((\Delta_1)_{[w, v]}^{-\mu_1}, (\Delta_1)_{[w, v]}^{-\mu_1} \right),
 $$
 as desired.  The construction of the integrable, gradient-like vector field $X_s$ for
$\Delta_s$ on $[w, v] \times B \times \rr^{2N}$ is as in the argument in the proof of Lemma~\ref{lem:neg-end}.

\end{proof}

\section{Lifting the Seidel isomorphism (Theorem~\ref{thm:suspension computation})}\label{sec:proof} 
We first outline the strategy of the proof (executed in Section~\ref{section. proof of main theorem}) to orient the reader.
Background results on homotopy theory are included in the appendices, and are referenced throughout.
  
Fix a linearly-controlled filling $(\sclag,F)$ of $(\leg,f)$, a shearing function $H \in \mathcal{H}\left(\claghiro \right)$ for $u>1$, and constants
$\Infin, \mu$. We let $N$ be the same integer $N$ appearing in the domain of $f$~\eqref{eqn. generating family domain and codomain}.
We will define four pointed spaces $W_N, A_N, B_N, C_N$ for which one has a {pushout square}	\eqnn
	\xymatrix{
	W_N \ar[r] \ar[d] & B_N \ar[d]
	\\
	A_N \ar[r] & C_N
	}
	\eqnnd
(i.e., $C_N$ is the union of $A_N$ and $B_N$ along $W_N$). In fact, this pushout square lives over the elementary pushout square:  
 	\eqnn
	\xymatrix{
	\{u\} \ar[r] \ar[d] & [u,\infty) \ar[d]
	\\
	(0,u] \ar[r] & (0, \infty)
	}
	\eqnnd
Further, each arrow in our pushout square of pointed spaces will be a cofibration.  This implies that our pushout square is a homotopy pushout square of pointed spaces. We will see that Lemma~\ref{lem:stab2-suspension} implies that
 stabilizing $F$ (i.e., letting $N \to \infty$) induces a pushout square of  prespectra:
 	\eqnn
	\xymatrix{
	W\ar[r] \ar[d] & B \ar[d]
	\\
	A \ar[r] & C.
	}
	\eqnnd
	Any homotopy pushout square of prespectra gives rise to
	 a
 long exact sequence of homotopy groups:
  \eqnn
	\ldots \to \pi_k(W) \to \pi_k(A) \oplus \pi_k(B) \to \pi_k(C) \to \ldots .
	\eqnnd
 In our situation, we will see that
 \begin{itemize}
 \item $B$ and $C$ are trivial (Corollaries~\ref{corollary. B is contractible} and~\ref{corollary. C is contractible}),
 \item the spectrum  associated to 
$W$ is equivalent to $\gfc(\leg,f;\Sphere)$ (Corollary~\ref{corollary. W is gf spectrum}), and
 \item the spectrum associated to 
$A$ is equivalent to $\Sigma^\infty(L_0/\leg)$ (Corollary~\ref{corollary. L mod Lambda}).
 \end{itemize}
 So an application of Whitehead's theorem 
 then implies
the equivalence of the spectra associated to $W$  and $A$. 
This concludes the outline of the proof of Theorem~\ref{thm:suspension computation} (Section~\ref{section. proof of main theorem}). 

 \subsection{Constants, Families, and Stabilizations}
Throughout the next subsections, we will make the following assumptions and choices.

\begin{assume} \label{assume:spec}$\text{ }$
\begin{enumerate}
\item $(\sclag, F)$ is a linearly-controlled filling of $(\leg, f)$ such that
\begin{enumerate} 
\item $\sclag$ is cylindrical over $[0,\infty)$, and thus $\clag = \theta(\sclag) \subset T^*(\rr_{>0} \times M)$ is conical over $[1,\infty)$;
\item  $F\co \rr_{>0} \times M \times \rr^N \to \rr$ is a generating family for $\clag$.
  \end{enumerate}
\item From $F$, we construct the family 
$$\{F_i\}_{i \geq N}$$
defined as 
 $F_{N} = F$, and for $i \geq N+1$,  $F_{i}$ is the rank $1$ stabilization of $F_{i-1}$ by either $Q_+(\e) = \e^2$ or $Q_-(\e)=-\e^2$. 
 \end{enumerate}
 \end{assume}
  \noindent
  Recall that, by Definition~\ref{defn:compatible}, $f = F|_{t = 1}\co M \times \rr^N \to \rr$ is a generating family for $\leg$, and $F|_{t = t_0} = t_0 f$, when $t_0 \geq 1$.

 \begin{choice} \label{choice:spec}$\text{ }$
 \begin{enumerate}
 \item The constant $u > 1$ is chosen sufficiently close to $1$ in order to satisfy the inequalities specified in  Choice~\ref{choice. u for shearing}.
 \item From a shearing function $H \in \mathcal{H}\left(\clag \right)$  as specified in Definition~\ref{defn:shear-H}, for all $i$, we construct the sheared difference function $\Delta_{F_i}: \rr_{>0} \times B \times \rr^{2N} \to \rr$
   as in Definition~\ref{defn:pert-diff}.
 \item The constants $\Infin$ and  $\mu$ are chosen to satisfy the inequalities in Choice~\ref{choice:mu-Infin}. 
\end{enumerate}
\end{choice}
\noindent
By Remark~\ref{rem:stab-choice-difference} with either choice of stabilization for $F_i$, $\Delta_{F_i}$ is well defined up to fiber-preserving diffeomorphism.

In our construction of the spectra, we will be using the following stabilization argument, which parallels that for Proposition~\ref{prop:leg-stab}.

\begin{prop}\label{prop:fill-stab} 
For Assumptions~\ref{assume:spec} and for all Choice~\ref{choice:spec},
if 
$F_i'$ differs from $F_i$ by a rank $1$ stabilization, then for all $J \subset \rr_{>0}$, there is a homotopy equivalence
$$\sigma^J: \Sigma\left(\Delta_{F_i}^{\leq \Infin}|_J/ \Delta_{F_i}^{\leq -\mu} |_J \right)\stackrel{\simeq}{\longrightarrow} 
\Delta_{F_i'}^{\leq \Infin}|_J/ \Delta_{F_i'}^{\leq-\mu}|_J.
$$
\end{prop}

\begin{proof} Suppose  
$F'\co \rr_{>0} \times B \times \rr^{N+1} \to \rr$ is a 
rank $1$ stabilization of $F$. Then for $J \subset \rr$, we have  
$$\begin{aligned}
\Delta_F&\co J \times B  \times \rr^{2N} \to \rr,\\
\Delta_{F'}&\co J \times B \times \rr^{2N} \times \rr^2 \to \rr.
\end{aligned}
$$
The argument in Remark~\ref{rem:stab-choice-difference} shows  we can assume that up to a fiber-preserving diffeomorphism 
$\Delta_{F'} = \Delta_F + Q(\e_1, \e_2)$, where $Q(\e_1, \e_2) = \e_1^2 - \e_2^2$.  Then  
Lemma~\ref{lem:stab2-suspension}   tells us that
for  $I = [-1,1]$, the inclusion $i: J \times B \times \rr^{2N} \times I  \to J \times B \times \rr^{2N} \times \rr^2$ {defined using
 the maps in Equations~\eqref{eqn. positive stabilization map}, \eqref{eqn. negative stabilization map}} given by 
$$i(t, x, \e, \tau) = (t, x, \e, 0, c\tau), \qquad   {c > {\sqrt{\Omega + \mu}}}$$
induces a homotopy equivalence
\begin{equation}
\iota^J: \left(\Delta_F^{\leq \Infin}|_J \times I, \Delta_F^{\leq -\mu}|_J \times I \bigcup \Delta_F^{\leq \Infin}|_J \times \partial I\right) \to 
\left(\Delta_{F'}^{\leq \Infin}|_J, \Delta_{F'}^{\leq -\mu}|_J \right).
\label{eqn:shear-stab}
\end{equation}
Thus, via Remark~\ref{remark. sublevel set pair suspension}    we get an induced map
$$\sigma^J: \Sigma\left(\Delta_F^{\leq \Infin}|_J/ \Delta_F^{\leq -\mu}|_J \right) \stackrel{\simeq}{\longrightarrow} 
\Delta_{F'}^{\leq \Infin}|_J/ \Delta_{F'}^{\leq -\mu}|_J.
$$
\end{proof}

\begin{prop} \label{lem:inclusions-stabs} Given Assumptions~\ref{assume:spec},  for all Choice~\ref{choice:spec}, if $J_1 \subset J_2 \subset (0, \infty)$, then
for all $i \geq N$, there is a map between quotients
$$\Delta_{F_i}^\Infin|_{J_1}/ \Delta_{F_i}^{-\mu}|_{J_1} \to \Delta_{F_i}^\Infin|_{J_2}/ \Delta_{F_i}^{-\mu}|_{J_2} $$
 that fits into the following diagram that commutes up to homotopy:
$$\label{eqn:changing-J}
	\xymatrix{
	\Sigma\left(\Delta_{F_i}^{\leq \Infin}|_{J_1} / \  \Delta_{F_i}^{\leq -\mu}|_{J_1}   \right) 
	\ar[r] 
	 \ar[d] 
	 & 
	\Sigma  \left(\Delta_{F_i}^{\leq \Infin}|_{J_2}/ \Delta_{F_i}^{\leq -\mu}|_{J_2}   \right) 
	  \ar[d] 	\\
	\left(\Delta_{F_{i+1}}^{\leq \Infin}|_{J_1} / \  \Delta_{F_{i+1}}^{\leq -\mu}|_{J_1}   \right)  
	\ar[r] 	& 
	\left(\Delta_{F_{i+1}}^{\leq \Infin}|_{J_2} / \  \Delta_{F_{i+1}}^{\leq -\mu}|_{J_2}   \right)   .
	}
$$
\end{prop} 

\begin{proof} The map of quotients
$$\Delta_{F_i}^\Infin|_{J_1}/ \Delta_{F_i}^{-\mu}|_{J_1} \to \Delta_{F_i}^\Infin|_{J_2}/ \Delta_{F_i}^{-\mu}|_{J_2} $$
is induced by the inclusion 
$$(\Delta_{F_i}^\Infin|_{J_1}, \Delta_{F_i}^{-\mu}|_{J_1}) \to \left(\Delta_{F_i}^\Infin|_{J_2}, \Delta_{F_i}^{-\mu}|_{J_2}\right). $$

Suppose  
$F_{i+1}\co \rr_{>0} \times B \times \rr^{i+1} \to \rr$ is a rank $1$ stabilization of $F_i$.   
The desired result follows from a straightforward check that shows that the maps $\iota^{J_i}$ defined in 
Equation~(\ref{eqn:shear-stab}) fit into
the following diagram of pairs, which commutes up to homotopy: 
{\scriptsize{
\eqnn
	\xymatrix{
	\left(\Delta_{F_i}^{\leq \Infin}|_{J_1} \times I, \  \Delta_{F_i}^{\leq -\mu}|_{J_1}  \times I \cup \Delta_{F_i}^{\leq \Infin}|_{J_1}  \times \partial I \right) 
	\ar@{^{(}->}[r]^{i}
	\ar[d]_{\iota^{J_1}}^\simeq 
	& 
	\left(\Delta_{F_i}^{\leq \Infin}|_{J_2} \times I, \  \Delta_{F_i}^{\leq -\mu}|_{J_2}  \times I \cup \Delta_{F_i}^{\leq \Infin}|_{J_2}  \times \partial I \right)  
	  \ar[d]^{\iota^{J_2}}_\simeq
	\\
	\left(\Delta_{F_{i+1}}^{\leq \Infin}|_{J_2} , \  \Delta_{F_{i+1}}^{\leq -\mu}|_{J_2}   \right)  
	\ar@{^{(}->}[r]_{i}
		& 
	\left(\Delta_{F_{i+1}}^{\leq \Infin}|_{J_2} , \  \Delta_{F_{i+1}}^{\leq -\mu}|_{J_2}   \right).  
	}
\eqnnd
}}
 \end{proof}

\subsection{The prespectrum $W$} Our first spectrum will be associated to the point $u \in \rr_{>0}$ and is defined in parallel to Definition~\ref{defn:leg-spectrum}.
\begin{defn}[$W_N$ and $W$]
\label{defn:W}  
Given Assumption~\ref{assume:spec}(1) and Choice~\ref{choice:spec}, 
we define the pointed space $W_N$  to be the quotient
$$
	W_N := \left(\Delta_F^{\leq \Infin}|_{\{u\}}\right) / \left(\Delta_F^{\leq -\mu}|_{\{u\}} \right).
$$
Given the family $\{F_i\}_{i \geq N}$ from Assumption~\ref{assume:spec}(2), we define a prespectrum 
$$W = \{(W_i, \sigma_i^{u})\}_{i \geq N}$$ 
as follows. 
\begin{enumerate}
\item $W_{i} = \left(\Delta_{F_i}^{\leq \Infin}|_{\{u\}}\right) / \left(\Delta_{F_i}^{\leq -\mu}|_{\{u\}} \right)$, 
  \item  $\sigma_{i}^{u}: \Sigma W_{i} \to W_{i+1}$  provided by Proposition~\ref{prop:fill-stab}.
 \end{enumerate}
\end{defn}

The following proposition is key in establishing that the spectrum associated to $W$ is equivalent to the generating family spectrum $\gfc(\leg, f; \Sphere)$.

\begin{prop} Given Assumptions~\ref{assume:spec} and Choice~\ref{choice:spec}, 
there is a  homotopy equivalence
	\begin{equation}\label{eqn:W}
\rho_{N} \co 
W_N = \left( \Delta_F^{\leq \Infin}|_{\{u\}}\right) / \left(\Delta_F^{\leq -\mu}|_{\{u\}} \right)
   \stackrel{\simeq}{\longrightarrow} \delta_f^{\leq \infin} / \delta_f^{\leq \epsilon}. 
	\end{equation}
	Moreover, for all $i \geq N$, $F_i$ induces a generating family $f_i$ of $\leg$ such that the corresponding maps
$\rho_{i}$ commute with stabilization: 
 if $F_{i+1}$ is a rank $1$ stabilization of $F_i$, then $F_{i+1}$ induces a rank $1$ stabilization $f_{i+1}$ of $f_i$ such that
the following diagram commutes up to homotopy:
\begin{equation}\label{eqn:W-GFH} 
	\xymatrix{
	\Sigma W_i
	\ar[r]^{\Sigma \rho_{i}\quad}_{\simeq\quad} \ar[d]_{\sigma_{i}^{u}}^{\simeq} & 
	 \Sigma \left(  \delta_{f_i}^{\leq \infin} / \delta_{f_i}^{\leq \epsilon}  \right)
	  \ar[d]^{\sigma_{i}^\leg}_{\simeq}
	\\
	W_{i+1} \ar[r]_{\rho_{{i+1}\quad}}^{\simeq\quad}
	 & 
	  \delta_{f_{i+1}}^{\leq \infin} / \delta_{f_{i+1}}^{\leq \epsilon},
	 }
\end{equation}
for $\sigma_i^\leg$ as in Definition~\ref{defn:leg-spectrum}.
\end{prop} 

\begin{proof} 
Recall that 
$$\left( \Delta_F^{\leq \Infin}|_{\{u\}},  \Delta_F^{\leq -\mu}|_{\{u\}}\right) =
\left( 
\left\{ \left(u, \delta_f^{\leq \lambda_{\Infin}}(u)\right) \right\}
\right), \ 
\left( 
\left\{ \left(u, \delta_f^{\leq \lambda_{-\mu}}(u)\right) \right\}
\right).
$$
Our analysis in Lemma~\ref{lem:Infin-mu} of the $\Delta$-$\delta$ translation function $\lambda_\alpha$, for $\alpha = \Infin, -\mu$ shows that the fiberwise gradient flow of $\Delta_F|_{\{u\}}$ gives
rise to the homotopy equivalence between pairs
$$\rho_N'\co\left(\Delta_F^{\leq \Infin}|_{\{u\}}, \Delta_F^{\leq -\mu}|_{\{u\}} \right) \stackrel{\simeq}{\longrightarrow} \left(\delta_f^{\leq \infin}, \delta_f^{\leq \epsilon}\right),$$
which gives rise to the homotopy equivalence between the quotients in Equation~\ref{eqn:W}.

Suppose   
$F_{i+1}\co \rr_{>0} \times B \times \rr^{i+1} \to \rr$ is a rank $1$ stabilization of $F_i$; this induces a rank $1$ stabilization, $f_i$, of $f$; observe that the
 functions
$\lambda_{\alpha}(t)$, $\alpha = \Infin, -\mu$ are unchanged under stabilization of $F_i$.  
We have 
$$\begin{aligned}
\Delta_{F_i}\co& \rr_{>0} \times B \times \rr^{2i} \to \rr, \\
\Delta_{F_{i+1}}\co& \rr_{>0} \times B \times \rr^{2i} \times \rr^2 \to \rr.
\end{aligned}
$$ 
To verify the commutativity of Diagram (\ref{eqn:W-GFH}), it suffices to verify that the maps $i^\leg$ and $\iota^J$ defined in Equations~(\ref{eqn:leg-stab}) and
(\ref{eqn:shear-stab}) fit into the following 
the following diagram of pairs that commutes up to homotopy: 
\eqnn
	\xymatrix{
	\left(\Delta_{F_i}^{\leq \Infin}|_{\{u\}} \times I, \  \Delta_{F_i}^{\leq -\mu}|_{\{u\}}  \times I \cup \Delta_{F_i}^{\leq \Infin}|_{\{u\}}  \times \partial I \right) 
	\ar[r]^{\qquad\quad\Sigma \rho_{i}'}_{\qquad\quad\simeq} \ar[d]_{\iota^J}^\simeq & 
	\left(\delta_{f_i}^{\leq \infin} \times I, \  \delta_{f_i}^{\leq -\mu} \times I \cup \delta_{f_i}^{\leq \infin} \times \partial I \right) 
	  \ar[d]^{i^\leg}_\simeq
	\\
	\left(\Delta_{F_{i+1}}^{\leq \Infin}|_{\{u\}} , \  \Delta_{F_{i+1}}^{\leq -\mu}|_{\{u\}}   \right)  
	\ar[r]_{\qquad \rho_{i+1}'}^{\qquad\simeq} & 
	\left(\delta_{f_{i+1}}^{\leq \infin}, \  \delta_{f_{i+1}}^{\leq -\mu}   \right).
	}
\eqnnd
 As described in the proofs of Propositions~\ref{prop:leg-stab} and \ref{prop:fill-stab}, for  $c > \sqrt{\Infin + \mu}$, $\iota^J$ is induced by the inclusion map,  
$$\begin{aligned}
\Delta_{F_i}^{\leq \alpha}|_{\{u\}} \times I  &\to  \Delta_{F_{i+1}}^{\leq \alpha}|_{\{u\}} \\
(u,x,\e,\te, \tau) &\mapsto (u,x,\e,\te, 0,  c\tau), ,
\end{aligned}
$$
and, similarly,  $i^\leg$ is induced by the inclusion map
$$\begin{aligned} 
\delta_{f_i}^{\leq \alpha} \times I &\to  \delta_{f_{i+1}}^{\leq \alpha} \\
(x,\e,\te, \tau) &\mapsto (x,\e,\te, 0, c\tau).
\end{aligned}
$$
A straightforward check shows that, up to homotopy, $i^J_*$, $i^\leg$ commute with the maps $\Sigma \rho_{i}'$, $\rho_{i+1}'$ given by fiberwise deformations.
\end{proof} 
 
The maps on the left-hand side of Diagram~(\ref{eqn:W-GFH}) define the prespectrum $W$, while the maps on the 
of the right-hand side 
define the generating family prespectrum of $f$ from Definition~\ref{defn:leg-spectrum}. 
 Thus, Proposition~\ref{prop:prespectra-spectra-maps} gives:

\begin{cor}\label{corollary. W is gf spectrum}
The spectrum associated to the prespectrum $W$ is equivalent to the generating family spectrum of $f$, $\gfc(\leg, f; \Sphere)$.
\end{cor}

\subsection{The prespectrum $A$}

In parallel to Definition~\ref{defn:W}, we now construct a spectrum associated to $J = (0, u] \subset \rr_{>0}$.
\begin{defn}[$A_N$ and $A$]
\label{defn:A} 
Given Assumptions~\ref{assume:spec}(1), 
we define the pointed space $A_N$ to be the quotient
	\eqnn
	A_N:=
	\left(\Delta_F^{\leq \Infin}|_{(0, u]}\right) / \left(  \Delta_F^{\leq -\mu}|_{ (0, u]} \right).
	\eqnnd
Given Assumptions~\ref{assume:spec}(2), we construct a prespectrum 
\begin{enumerate}
\item $A_{i} = \left(\Delta_{F_i}^{\leq \Infin}|_{(0,u]}\right) / \left(\Delta_{F_i}^{\leq -\mu}|_{(0,u]} \right)$, 
  \item  $\sigma_{i}^{(0,u]}: \Sigma A_{i} \to A_{i+1}$ provided by Proposition~\ref{prop:fill-stab}.
 \end{enumerate}
\end{defn}

\begin{ex} \label{ex:quotient-suspensions}  Given a submanifold with boundary $L$, we have the following homotopy equivalences ($\simeq$) and homeomorphisms ($\cong$):
$$\begin{aligned}
\Sigma(L/\partial L) & \simeq L \times I / (\partial L \times I \cup L \times \partial I); \\
\Sigma^2(L/\partial L) &\simeq (L \times I) \times I / \left((\partial L \times I \cup L \times \partial I) \times I\right) \cup (L \times I) \times \partial I\\
&\cong L \times I^2 / (\partial L \times I^2 \cup L \times \partial I^2)\\
&\cong D^2(L) / (D^2(\partial L) \cup \partial D^2(L)) \\
&=D^2(L) / (D^2(\partial L) \cup S^{1}(L)),
\end{aligned}
$$
where $D^2(L), S^1(L)$ are the trivial $2$-dimensional disk and $1$-dimensional sphere bundles over $L$.
More generally, for all $i \geq 1$, we find a homotopy equivalence
$$\zeta_i: \Sigma^i(L/\partial L) \simeq D^i(L) / (D^i(\partial L) \cup S^{i-1}(L)),$$
where $D^i(L), S^{i-1}(L)$ are the trivial $i$-dimensional disk and $(i-1)$-dimensional sphere bundles over $L$.
 \end{ex}

\begin{prop}  Given Assumptions~\ref{assume:spec} and Choice~\ref{choice:spec},
there is a homotopy equivalence
$$
\beta_N \co A_N
   \stackrel{\simeq}\longrightarrow D^N(L_0) / \left(D^N(\partial L_0) \cup S^{N-1}(L_0) \right),
    $$
    where $D^N(L_0)$ and $S^{N-1}(L_0)$ denote the trivial $N$-dimensional disk and
    $(N-1)$-dimensional sphere bundles over the compact end of the  Lagrangian filling $L_0$.
    Moreover, for all $i \geq N$, there is a  homotopy equivalence $\beta_{i}$
          that  commutes with stabilization:
    if $F_{i+1}$ is a rank $1$ stabilization of $F_i$, then 
   the following diagram commutes  up to homotopy
\begin{equation}\label{eqn:A-filling} 
	\xymatrix{
	\Sigma  A_i
	\ar[r]^{\Sigma \beta_{i}\qquad\qquad\qquad}_{\simeq\qquad\qquad\qquad} \ar[d]_{\sigma_{i}^{(0,u]}}^{\simeq} & 
	\Sigma\left( D^i(L) / \left( D^i(\partial L) \cup S^{i-1}(L) \right) \right)
	   \ar[d]^{\zeta_i}_{\simeq}
	\\
	  A_{i+1} \ar[r]_{\beta_{{i+1}}\qquad\qquad\qquad}^{\simeq\qquad\qquad\qquad}
	 & 
	 \left( D^{i+1}(L) \right)/ \left( D^{i+1}(\partial L)\cup S^{i}(L) \right),
	 }
	\end{equation}
where $\zeta_i$ is the homeomorphism defined in   Example~\ref{ex:quotient-suspensions}.
     \end{prop}
    
    \begin{proof} 
    The claimed map $\beta_N$ follows directly from the homotopy equivalence of pairs given in Lemma~\ref{lem:neg-end}: as shown in the proof of Lemma~\ref{lem:neg-end}, the idea for $\beta_N$  is to first follow the flow of a gradient-like vector field for $\Delta_F$ until
we reach level $\sigma$.  Then, applying our analysis of the $\Delta$-$\delta$ translation function $\lambda_{\sigma}(t)$ for $t \in (0,t_-] \cup [t_+, u]$, we show that there is a homotopy equivalence 
  $$\left(\Delta_F^{\leq \sigma}|_{(0, u]},  \Delta_F^{\leq -\mu}|_{ (0, u]} \right)
    \simeq 
    \left(\Delta_F^{\leq \sigma}|_{[t_-, 1]},  \Delta_F^{\leq -\mu}|_{ [t_-, 1]} \cup \Delta_F^{\leq \sigma}|_{\{1\}} \right), 
  $$
and lastly we apply a Morse-Bott argument to construct a homotopy equivalence   
$$ \left(\Delta_F^{\leq \sigma}|_{[t_-, 1]},  \Delta_F^{\leq -\mu}|_{ [t_-, 1]} \cup \Delta_F^{\leq \sigma}|_{\{1\}} \right)
\simeq
\left(D^N(L),   D^N(\partial L) \cup S^{N-1}(L) \right).
$$

   Suppose  
$F_{i+1}\co \rr_{>0} \times M \times \rr^{N+1} \to \rr$ is a rank $1$ stabilization of $F_i$. 
    stabilizing $F_i$ will not affect the functions $\lambda_a(t)$, for $a = \Infin, -\mu, \sigma$.
    To verify the commutativity of Diagram (\ref{eqn:A-filling}), 
it suffices  to verify that the map $\iota^{(0,u]}$ defined in Equation~(\ref{eqn:shear-stab}) 
 makes the following commutative diagram of pairs:
 {\scriptsize
\eqnn
	\xymatrix{
	\left(\Delta_{F_i}^{\leq \Infin}|_{(0,u]} \times I, \  \Delta_{F_i}^{\leq -\mu}|_{(0,u]}  \times I \cup \Delta_{F_i}^{\leq \Infin}|_{\{(0,u]\}}  \times \partial I \right) 
	\ar[r]^{\qquad\qquad\quad\Sigma \widetilde\beta_i}_{\qquad\qquad\quad\simeq} \ar[d]_{\iota^{(0,u]}}^\simeq & 
	\Sigma \left(D^{i}(L),  D^{i}(\partial L) \cup S^{i-1}(L)  \right)
	  \ar[d]^{\zeta_i}_\simeq
	\\
	\left(\Delta_{F_{i+1}}^{\leq \Infin}|_{(0,u]} , \  \Delta_{F_{i+1}}^{\leq -\mu}|_{(0,u]}   \right)  
	\ar[r]_{\qquad \widetilde\beta_{i+1}}^{\qquad\simeq} & 
	\left(D^{i+1}(L),  D^{i+1}(\partial L) \cup S^{i}(L) \right)}
\eqnnd
}
As mentioned in the Proof of Proposition~\ref{prop:fill-stab},
the map $\iota^{(0,u]}$ is induced by the inclusion map
$$\begin{aligned}
\Delta_F^{\leq \alpha}|_{\{u\}} \times I  &\to  \Delta_{F'}^{\leq \alpha}|_{\{u\}} \\
(u,x,\e,\te, \tau) &\mapsto (u,x,\e,\te, 0,  c\tau),
\end{aligned}
$$
 for an appropriate constant $c$.
A straightforward check shows that the maps $\iota^{(0,u]}$, $\zeta_{i}$ commute up to homotopy with the maps $\Sigma \widetilde \beta_{F_i}$, $\widetilde \beta_{F_{i+1}}$ whose
constructions are outlined in the previous paragraph. 
    \end{proof} 
	
The maps on the left-hand side of Diagram~(\ref{eqn:A-filling}) define the prespectrum $A$, while the maps on the 
of the right-hand side 
define the $N$-tail of the prespectrum that defines the suspension spectrum $\Sigma^\infty(L,\partial L)$; see   Example~\ref{ex:quotient-suspensions} and Definition~\ref{defn:suspension-spec}. 
 Thus, Proposition~\ref{prop:prespectra-spectra-maps} gives:

\begin{cor}\label{corollary. L mod Lambda}
The spectrum associated to the prespectrum $A$ is equivalent to the suspension spectrum of the quotient space $L_0/\partial L_0 $, $\Sigma^\infty(L_0/\partial L_0)$.
\end{cor}

\subsection{The prespectrum $B$}

In parallel to Definitions~\ref{defn:W} and \ref{defn:A}, we now construct a spectrum associated to $J = [u, \infty) \subset \rr_{>0}$.
\begin{defn}[$B_N$ and $B$]
\label{defn:B} Given Assumptions~\ref{assume:spec}(1), 
we define the pointed space $B_N$ to be the quotient
	\eqnn
	 B_N:= \left(\Delta_{F}^{\leq \Omega}|_{[u,\infty)} \right)/ \left(\Delta_{F}^{\leq -\mu}|_{[u,\infty)}\right).
	\eqnnd
Given Assumptions~\ref{assume:spec}(2), we construct a prespectrum 
$$B = \left\{\left(B_i, \sigma_i^{[u,\infty)}\right)\right\}_{i \geq N}$$ as follows:
\begin{enumerate}
\item $B_{i} = \left(\Delta_{F_i}^{\leq \Infin}|_{[u,\infty)}\right) / \left(\Delta_{F_i}^{\leq -\mu}|_{[u,\infty)} \right)$,
 \item  $\sigma_{i}^{[u,\infty)}: \Sigma B_{i} \to B_{i+1}$ provided by Proposition~\ref{prop:fill-stab}.
 \end{enumerate}
 \end{defn}

\begin{prop} For all $N$, $B_N$  
is contractible. 
\end{prop}

\begin{proof} 
As shown in Lemma~\ref{lem:pos-end},
for sufficiently large $v > u$, there is a homotopy equivalence 
$$
    \left(\Delta^{\leq \Infin}_{[u, \infty)}, \Delta^{\leq -\mu}_{[u, \infty)}\right) \simeq
    \left(\delta_{f}^{\leq \infin} \times [u,v], \delta_f^{\leq \epsilon} \times [u,v] \cup \delta_f^{\leq \infin} \times\{v\} \right).  
$$
It follows that $B_N = \Co\left(\delta_f^{\leq \infin}/\delta_f^{\leq \epsilon}\right)$
and is thus contractible. 
 \end{proof} 

By example~\ref{example. trivial spectrum}, we have:

\begin{cor}\label{corollary. B is contractible}
The spectrum associated to $B$ is a trivial spectrum.
\end{cor}

\subsection{The prespectrum $C$}

In parallel to Definitions~\ref{defn:W}, \ref{defn:A}, and \ref{defn:B}, we now construct a spectrum associated to $J = (0, \infty) = \rr_{>0}$.

\begin{defn}[$C_N$ and $C$]
\label{defn:C}
Given Assumptions~\ref{assume:spec}(1), 
we define the pointed space $C_N$ to be  the quotient
	\eqnn
	C_N:= \left(\Delta_{F}^{\leq \Omega}|_{(0,\infty)} \right)/ \left(\Delta_{F}^{\leq -\mu}|_{(0,\infty)}\right).
	\eqnnd
Given Assumptions~\ref{assume:spec}(2), we construct the prespectrum
	$$C= \left\{\left(C_i, \sigma_i^{(0,\infty)}\right)\right\}_{i \geq N}$$ as follows:
\begin{enumerate}
\item $C_{i} = \left(\Delta_{F_i}^{\leq \Infin}|_{(0,\infty)}\right) / \left(\Delta_{F_i}^{\leq -\mu}|_{(0,\infty)} \right)$,
 \item  $\sigma_{i}^{(0,\infty)}: \Sigma C_{i} \to C_{i+1}$ provided by Proposition~\ref{prop:fill-stab}.
 \end{enumerate}
\end{defn}

\begin{prop}
For all $N$, $C_N$ is contractible.
\end{prop}

\begin{proof}
This follows immediately from Lemma~\ref{lem:contractible pair for (-oo,oo)}.
\end{proof}

By example~\ref{example. trivial spectrum}, we have:

\begin{cor}\label{corollary. C is contractible}
The spectrum associated to $C$  is a trivial spectrum.
\end{cor}

\subsection{Proof of Theorem~\ref{thm:suspension computation}} \label{section. proof of main theorem}
Given our Assumptions~\ref{assume:spec} and Choice~\ref{choice:spec}, for each $i \geq N$, we have the following 
commuting square of pairs:
	\eqn\label{eqn. pushout of pairs of spaces ABCW}
	\xymatrix{
	\left(\Delta_{F_i}^{\leq \Infin}|_{\{u\}} , \Delta_{F_i}^{\leq -\mu}|_{\{u\}} \right) \ar[r] \ar[d] & \left(\Delta_{F_i}^{\leq \Omega}|_{[u,\infty)} , \Delta_{F_i}^{\leq -\mu}|_{[u,\infty)}\right) \ar[d]\\
	\left(\Delta_{F_i}^{\leq \Infin}|_{(0, u]} ,   \Delta_{F_i}^{\leq -\mu}|_{ (0, u]} \right) \ar[r] & \left(\Delta_{F_i}^{\leq \Omega}|_{(0,\infty)}  , \Delta_{F_i}^{\leq -\mu}|_{(0,\infty)}\right).
	}
	\eqnd
We obtain a pushout square of pointed spaces   by passing to the associated quotients. We denoted these quotient spaces by $W_i, A_i, B_i,$ and $C_i$ in Definitions
\ref{defn:W},
\ref{defn:A},
\ref{defn:B}, and
\ref{defn:C}, respectively, so we obtain a pushout square of spaces
	\begin{equation} \label{eqn:square-i}
	\xymatrix{
	W_i \ar[r] \ar[d] & B_i  \ar[d]\\
	A_i \ar[r] & C_i.
	}
	\end{equation}

 We prove in Lemma~\ref{lemma. why a cofibration} that this pushout square is a homotopy pushout.
 By Lemma~\ref{lem:inclusions-stabs}, each map in the pushout square of pairs~\eqref{eqn. pushout of pairs of spaces ABCW} is compatible with stabilizations, and thus 
we obtain a commuting square of prespectra:
	\eqnn
	\xymatrix{
	W \ar[r] \ar[d] & B  \ar[d]\\
	A \ar[r] & C.
	}
	\eqnnd
	Any homotopy pushout square of prespectra gives rise to a homotopy pushout square of spectra; see Remark~\ref{remark. spectrification preserves colimits}.
	Because every  pushout square of spectra is also a pullback square of spectra (Theorem~\ref{thm:pushout=pullback}) one obtains an associated long exact sequence of homotopy groups:
	\eqnn
	\ldots \to \pi_k(W) \to \pi_k(A) \oplus \pi_k(B) \to \pi_k(C) \to \ldots,
	\eqnnd
see Proposition~\ref{prop:pull-back-spectra-homotopy}.
By Corollaries~\ref{corollary. C is contractible}, \ref{corollary. B is contractible}, the homotopy groups of $B$ and $C$ are trivial. We conclude that 
 	\eqnn
	W \to A
	\eqnnd
is an equivalence of spectra (Definition~\ref{defn. equivalence of spectra}). 
Corollary~\ref{corollary. W is gf spectrum} identifies the spectrum associated to $W$ with 
$\gfc(\leg,f;\Sphere)$, while Corollary~\ref{corollary. L mod Lambda} identifies the spectrum associated to $A$ with the suspension spectrum $\Sigma^\infty(L/\Lambda)$.   \qed

It remains to prove:

\begin{lem}\label{lemma. why a cofibration}  The pushout square 
\eqref{eqn:square-i} is a homotopy pushout square.
\end{lem}

For this, we will choose a neighborhood $\tilde U$ of 
$\Delta^{\leq \Omega}_{F_i|\{u\}}$
inside
$\Delta^{\leq \Omega}_{F_i|[u,\infty)}$
and construct a deformation retraction of $\tilde U$ onto 
$\Delta^{\leq \Omega}_{F_i|\{u\}}
\bigcup
\Delta^{\leq -\mu}_{F_i|[u,\infty)}$.
The argument is a slight modification of the arguments in Lemmas~ in~\cite[Lemmas~5.4 and~5.6]{S-T:obstruct} -- the idea is to lift a deformation retraction evident in $\RR^2$ to a deformation retraction taking place in $\Delta^{\leq \Omega}_{F_i|[u,\infty)}$.

We set some notation.  First recall that  by Lemma~\ref{lem:Infin-mu},  $0 < \lambda_{-\mu}(u) < \lmin$.   Also straightforward calculations from Definition~\ref{defn:shear-H} and \eqref{eqn:lambda} show that 
 $\lambda_{\Omega}(t)$ is concave up along $[u,\infty)$ with a global minimum occurring at some $t > u$.  This makes it possible to make the following choice.

\begin{choice}[$\epsilon$]
Choose a real number $\epsilon$ satisfying:	
\newenvironment{eps-choice}{
		\renewcommand*{\theenumi}{(E\arabic{enumi})}
		\renewcommand*{\labelenumi}{(E\arabic{enumi})}
		\enumerate
	}{
		\endenumerate
	}
	\begin{eps-choice}
	\item $\epsilon>0$.
	\item $\epsilon$ is small enough so that $\lambda_{-\mu}(t)$  is a regular value of $\delta_{f_i}$ for all $t \in [u,u+ \epsilon)$. 

	\item\label{item. global min of lambda Omega} $\epsilon$ is small enough so that the global minimum of $\lambda_{\Omega}(t)$ along the interval $[u,\infty)$ is attained where $ t> u + \epsilon$. 

	\end{eps-choice}	
\end{choice}

\begin{choice}[$\tilde \lambda$]
\label{choice. tilde lambda}
We choose a smooth function
	\eqnn
	\tilde \lambda : [u,\infty) \to \RR
	\eqnnd
such that (see Figure~\ref{figure. cofibration retraction}):
\newenvironment{tildelambda-choice}{
		\renewcommand*{\theenumi}{(L\arabic{enumi})}
		\renewcommand*{\labelenumi}{(L\arabic{enumi})}
		\enumerate
	}{
		\endenumerate
	}
	\begin{tildelambda-choice}
	\item\label{item. tilde lambda starts high} $\lambda_\Omega(u) < \tilde \lambda(u)$. 
	\item\label{item. tilde lambda convex} $\tilde \lambda$ is concave up.
	\item\label{item. tilde lambda min} $\tilde \lambda$ attains a global minimum at $t = u+\epsilon/2$. 

	\item\label{item. tilde lambda values sandwiched} For all $t > u + \epsilon/2$, $\lambda_{-\mu}(t) < \tilde\lambda(t) < \lambda_{\Omega}(t)$.
	\item\label{item. no critical values} For all $t \in [u+\epsilon/4,u+\epsilon]$, $\lambda_{-\mu}(u) < \tilde\lambda(t) < \lmin$. 

	\item\label{item.unique crossing}  There is a unique $t_0 \in [u,\epsilon/2)$ such that $\tilde\lambda(t_0)=\lambda_\Omega(t_0)$. 
	\end{tildelambda-choice}	
\end{choice}

\noindent
Observe that \ref{item. tilde lambda convex}and \ref{item. tilde lambda min} guarantee that $\tilde\lambda(t)$ is decreasing along the interval $[u,u+\epsilon/2)$, and increasing along $(u+\epsilon/2,\infty)$.  Condition~\ref{item. no critical values} guarantees that for all $t \in [u+\epsilon/4,u+\epsilon]$, there are no critical values of  $\delta_{f_i}$  in   $[\lambda_{-\mu}(t),\tilde\lambda(t)]$.)

\begin{notation}[$U$]\label{notation. U not U tilde}
Consider the (neither open nor closed in $\RR^2$) region 
	\eqn\label{eqn. U locus}
	U := \{(t,y) \, | \, t \geq u \text{ and } y < \tilde\lambda(t) \text{ and } y \leq \lambda_\Omega(t)\} \}
	\subset
	\RR^2.
	\eqnd
	
 By \ref{item.unique crossing}, we
can accordingly define $U$ piecewise:
	\eqnn
	U = 
	\{ t < t_0 \text{ and } y \leq \lambda_\Omega(t) \}
	\bigcup
	\{ t \geq t_0 \text{ and } y < \tilde\lambda(t) \}.
	\eqnnd
\end{notation}

\begin{notation}[$K$ and $v$]
\label{notation. K and v}
Now consider the subset
	\eqnn
	K := 
	\{t = u \text{ and } y \leq \lambda_\Omega(u)\}
	\bigcup
	\{t \geq u \text{ and } y \leq \lambda_{-\mu}(t)\} \subset U.
	\eqnnd
As $K$ and $U$ are both contractible, there exists a deformation retraction of $U$ onto $K$. Even better, there exists a smooth vector field $v$ on $U$, which we write as
	\eqn\label{eqn. v phi psi}
	v(t,y) = \phi(t,y)\del_t + \psi(t,y)\del_y,
	\eqnd
satisfying the following. 
(See Figure~\ref{figure. cofibration retraction}.)
\newenvironment{v-choice}{
		\renewcommand*{\theenumi}{(V\arabic{enumi})}
		\renewcommand*{\labelenumi}{(V\arabic{enumi})}
		\enumerate
	}{
		\endenumerate
	}
	\begin{v-choice}
	\item\label{item. v at top} If $y \geq {\frac 1 2}(\underline{\ell} + \tilde\lambda(u+\epsilon/2))$ or $y \leq {\frac 1 2} \lambda_{-\mu}(t)$,  then $v(t,y)$ has no vertical components. (That is, $\psi(t,y) = 0$.)
	\item\label{item. v at bottom} Along the region ${\frac 1 2} \lambda_{-\mu}(t) \leq y \leq {\frac 1 2}(\underline{\ell} + \tilde\lambda(u+\epsilon/2))$,  $v(t,y)$ has non-positive vertical components. (That is, $\psi(t,y) \leq 0$ on this region.) 
	\item\label{item. v horizontally} If $t > \epsilon/2$, $v(t,y)$ has positive $\del_t$ component (so $\phi(t,y) > 0$). If $t < \epsilon/2$, $v(t,y)$ has negative $\del_t$ component (so $\phi(t,y) < 0$). 
	\item\label{item. v retracts} There is a continuous non-negative function $\tau : U \times [0,1] \to \RR$ so that the time-$\tau$ flow by $v$
		\eqnn
		( (t,y),s) \mapsto \flow_v^{\tau( (t,y),s)}(t,y)
		\eqnnd
	exhibits a strong deformation retraction of $U$ to $K$. (In particular, $\tau((t,y),s) = 0$ if $(t,y) \in K$.)
	\end{v-choice}
\end{notation}

\begin{rem}[$v$ leaves $U$ invariant]
Let us check that the conditions~\ref{item. v at top}, \ref{item. v at bottom}, and \ref{item. v horizontally} guarantee that for all non-negative $\tau$ vanishing along $K \times [0,1]$,   the time-$\tau$ flow of $v$ leaves $U$ invariant:
	\eqn\label{eqn. v leaves U invariant}
	((t,y),s) \in U  \times [0,1]
	\implies \flow_v^{\tau((t,y),s)}(t,y) \in U.
	\eqnd
Note we know $\psi$ is a non-positive function by \ref{item. v at top} and \ref{item. v at bottom}.
We address \eqref{eqn. v leaves U invariant} in three cases:

\begin{enumerate}
\item If $t < \epsilon/2$, then $\phi<0$ function by \ref{item. v horizontally}. 
Because both $\tilde \lambda$ and $\lambda_\Omega$ are concave up, and have negative derivative when $t<\epsilon/2$, decreasing the $t$ coordinate preserves the inequality 
$y <  \tilde \lambda(t),
y \leq \lambda_\Omega(t)$ in~\eqref{eqn. U locus}. Decreasing (or keeping fixed) the $y$ coordinate also preserves this inequality. Because $\tau$ vanishes along $K \times [0,1]$ -- and in particular along $\{y=u\} \times [0,1]$ -- the condition $t \geq u$ is also preserved. This proves \eqref{eqn. v leaves U invariant} in this case.

\item If $t > \epsilon/2$, then $\phi>0$ by \ref{item. v horizontally}. 
Because $\tilde \lambda$ is concave up, and has positive derivative when $t>\epsilon/2$, increasing  the $t$ coordinate preserves the inequality $y <  \tilde \lambda(t)$. (Note that when $t > \epsilon/2$, $ \tilde \lambda(t) < \lambda_\Omega(t)\} $ by \ref{item. tilde lambda values sandwiched}.) Decreasing (or keeping fixed) the $y$ coordinate also preserves this inequality. This proves \eqref{eqn. v leaves U invariant} in this case. 

\item If $t = \epsilon/2$, \ref{item. v horizontally} guarantees that $\phi(t,y)=0$. Thus, $v(t,y)$ is some non-positive multiple of $\del_y$, and decreasing (or keeping fixed) the $y$ coordinate preserves  the inequality $y <  \tilde \lambda(t)$.
\end{enumerate}
\end{rem}

\begin{figure}
		\[
			\xy
			\xyimport(8,8)(0,0)
			{
			\includegraphics[width=5in]{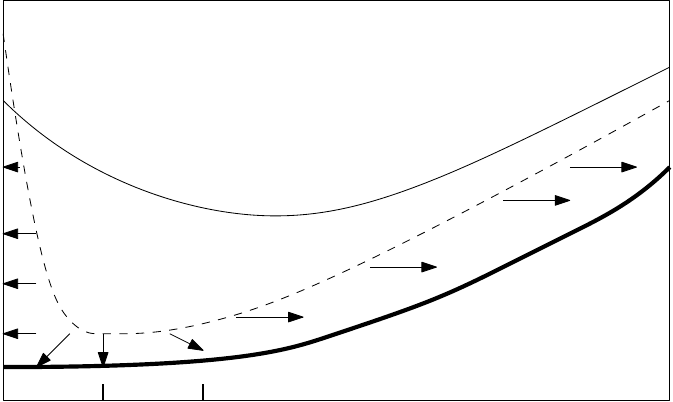}
			}
			,(7.7,-0.3)*{t}
			,(0.2,-0.3)*{u}
			,(1.2,-0.3)*{u+\epsilon/2}
			,(2.45,-0.3)*{u+\epsilon}
			,(-0.2,7.7)*{\delta_{f_i}}
			\endxy
		\]
	\caption{The dashed curve is the graph of $\tilde \lambda$. 
	The thick solid curve is the graph of $\lambda_{-\mu}$; the region along and below this curve is the image of $\Delta_{F_i}^{\leq -\mu}|_{[u,\infty)}$.  The thinner solid curve is the graph of $\lambda_{\Omega}$.  We invite the reader to verify conditions \ref{item. tilde lambda starts high} through \ref{item. tilde lambda values sandwiched}, along with condition~\ref{item. global min of lambda Omega}.	 Drawn using arrows is the vector field $v$. 
	}
	\label{figure. cofibration retraction}
\end{figure}

\begin{proof}[Proof of Lemma~\ref{lemma. why a cofibration}.]
By Remark~\ref{rem:old-lem-B9}, it suffices to show that either of the initial arrows in (\ref{eqn:square-i}) is a cofibration. We give the argument for the top rightward-pointing arrow. (In fact, all four arrows are cofibrations, but we do not need this fact.) 

To see that the inclusion in question is a cofibration, it suffices to show that the subset
	\eqn\label{eqn. B_N cofibration}
	{\frac
	{\Delta_{F_i}^{\leq \Omega}|_{\{u\}}}
	{\Delta_{F_i}^{\leq -\mu}|_{\{u\}}}
	}
	\subset
	{\frac
	{\Delta_{F_i}^{\leq \Omega}|_{[u,\infty)}}
	{\Delta_{F_i}^{\leq -\mu}|_{[u,\infty)}}
	}
	\eqnd
admits a neighborhood deformation retract.

Now we define $\tilde U$ to be
	\eqnn
	\tilde U :=  \left\{
	(t,x,\eta,\tilde\eta) \in \Delta_{F_i}^{\leq \Omega}|_{[u,\infty)}  \, \text{ s.t. } \, \delta_{f_i}(x,\eta,\tilde\eta) < \tilde\lambda(t)
	\right\}.
	\eqnnd
We will exhibit a deformation retraction of $\tilde U$ to
	\eqnn
	\tilde K := 
	\Delta^{\leq \Omega}_{F_i|\{u\}}
	\bigcup
	\Delta^{\leq -\mu}_{F_i|[u,\infty)}. 
	\eqnnd
For this, consider the projection
	\eqn\label{eqn. pi projection U tilde}
	\pi: \RR \times B \times \RR^i \times \RR^i\to U,
	\qquad
	(t,x,\eta,\tilde\eta) \mapsto (t,\delta_{f_i}(x,\eta,\tilde\eta)).
	\eqnd
By design, the preimages of $U$ and $K$ (Notation~\ref{notation. U not U tilde} and Notation~\ref{notation. K and v}) are
	\eqnn
	\pi^{-1}(U) = \tilde U 
	\qquad\text{and}\qquad
	\pi^{-1}(K) = \tilde K.
	\eqnnd
Choose a Riemannian metric on $B \times \RR^i \times \RR^i$ and define the following vector field on $\tilde U \subset \RR \times B \times \RR^i \times \RR^i$:
	\eqnn
	\tilde v
	:= (\phi \circ \pi) \del_t \bigoplus (\psi \circ \pi) {\frac {\nabla \delta_{f_i}} {||\nabla \delta_{f_i}||^2} }.
	\eqnnd
Here, $\phi$ and $\psi$ are the functions on $U$ from~\eqref{eqn. v phi psi}. 
The direct sum notation is utilizing the splitting $T_{t,x,\eta,\tilde\eta}(\tilde U ) \cong T_t \RR \oplus T_{x,\eta,\tilde\eta}(B \times \RR^i \times \RR^i)$. 
(Let us assuage the reader concerned by the division by $|| \nabla \delta_{f_i}||^2$. By~\ref{item. v at top}, the function $\psi \circ \pi$ is non-zero only at points $(t,x,\eta,\tilde\eta)$ satisfying the inequality 
	\eqnn
	0 
	< {\frac 1 2} \lambda_{-\mu}(t)
	< \delta_{f_i}(x,\eta,\tilde\eta) 
	< {\frac 1 2}(\underline{\ell} + \tilde\lambda(u+\epsilon/2))
	<\underline{\ell}.
	\eqnnd
In particular, $\delta_{f_i}$ has no critical points in the support of $\psi \circ \pi$.)
By definition of $\pi$, we see that the derivative  satisfies
	\eqnn
	D\pi(\del_t) = \del_t
	\qquad\text{and}\qquad
	D\pi\left({\frac {\nabla \delta_{f_i}} {||\nabla \delta_{f_i}||^2} }\right) = \del_y \qquad\text{ (wherever $\nabla \delta_{f_i} \neq 0$) }.
	\eqnnd
Thus, 
	\eqnn
	D\pi\left(\tilde v(t,x,\eta,\tilde\eta)\right) = v\left(t,\delta_{f_i}(x,\eta,\tilde\eta)\right)
	\eqnnd
where $v$ is the vector field from~\eqref{eqn. v phi psi}. In particular, $\tilde v$ is a lift of $v$ to $\tilde U$. 
Defining $\tilde \tau$ by the composition
	\eqnn
	\tilde\tau: \tilde U \times [0,1] \xrightarrow{\pi \times \id_{[0,1]}}
	U \times [0,1] \xrightarrow{\tau} \RR
	\eqnnd
(see \ref{item. v retracts} for $\tau$) we find that $\pi$ intertwines the time-$\tilde\tau$ flow of $\tilde v$ with the time-$\tau$ flow of $v$. By \ref{item. v retracts}, we conclude that the time-$\tilde \tau$ flow of $\tilde v$ exhibits a deformation retraction of $\tilde U$ to $\tilde K$, as desired.
\end{proof}	

\appendix

\section{Some homotopy equivalences from Morse theory}

  In this section, we review some basic homotopy equivalences that arise through Morse-theoretic arguments.

The following lemma is an extension of the key deformation lemma in Morse theory to non-closed domains.
The proof is a standard variation on the usual proof for closed domains, for example, \cite[Section 3]{milnor:morse}.

\begin{lem} \label{lem:retract}
Let $B$ be a smooth manifold. Given a smooth function $g\co B \times \rr^N \to \rr$, if there
is a complete, gradient-like vector field $X$ for $g$ such that $X(g)$ is bounded away from $0$ on the set $g^{-1}[a, b]$, then the sublevel set $g^a$ is a deformation retract of the sublevel set $g^b$.  
\end{lem}

We generalize the above lemma to discuss 1-parameter families of smooth functions and intervals such that throughout the time interval the paths of the critical values of the functions do not cross the endpoints of the intervals.  
This result is sometimes referred to as ``critical non-crossings" and at the level of homology was used in~\cite{traynor:gf-polys} and in~\cite[Lemma 2.4]{S-T:obstruct}.

\begin{lem}[Critical Non-Crossings]
\label{lem:crit-non-crossing}  Given 
\begin{enumerate}
\item a continuous $1$-parameter family of functions
  $g_s\co B \times \rr^{N} \to \rr$, $s \in [0, 1]$, that agree outside a compact set, and
  \item continuous paths $\alpha, \beta\co [0,1] \to \rr$, with $\alpha(s) \leq \beta(s)$, and $\epsilon > 0$ such that, for all $s$, 
  there exists a  complete,   gradient-like vector field
$X_s$ for $g_s$ such that $X_s(g_s)$ is bounded away from $0$ on 
 $$g_s^{-1}\left( [\alpha(s) - \epsilon, \alpha(s) + \epsilon] \cup [\beta(s) - \epsilon, \beta(s) + \epsilon] \right),$$
 \end{enumerate}
there is then a homotopy equivalence
$$\left(g_0^{\leq \beta(0)}, g_0^{\leq \alpha(0)}\right) \simeq \left(g_1^{\leq \beta(1)}, g_1^{\leq \alpha(1)}\right).$$
\end{lem}

\begin{proof} The argument parallels that for \cite[Proposition 6.6]{traynor:shomology}.
Hypothesis (2) implies that  
there are no critical values of $g_s$ in $(\alpha(s) - \epsilon, \alpha(s) + \epsilon) \cup
(\beta(s) - \epsilon, \beta(s) + \epsilon)$.
Hypothesis (1) guarantees that, by breaking the path into smaller segments, we can 
assume that for all $s\in [0,1]$, 
\begin{equation} \label{eqn:dist-gap}
\min\left\{ \| g_{1} - g_{0} \|, |\beta(1) - \beta(0)|, |\alpha(1) - \alpha(0)|\right\}  < \frac{\epsilon}{3},
\end{equation}
where $\| \cdot \|$ denotes the sup norm.
Observe that for all $c$,
$$g_{0}^{\leq c - \epsilon/3} \subset  g_{1}^{\leq c} \subset g_{0}^{\leq c + \epsilon/3} \subset g_{1}^{\leq c+2\epsilon/3}.$$
These inclusions give rise to inclusions of pairs on the top line in the following diagram:
{\tiny
\eqnn
	\xymatrix{
	\left(g_0^{\leq \beta(0) - \epsilon/3}, g_0^{\leq \alpha(0) - \epsilon/3}\right) \ar@{^{(}->}[r]\ar[d]^{\cong}
	&\left(g_1^{\leq \beta(0) }, g_1^{\leq \alpha(0) }\right) \ar@{^{(}->}[r]\ar[d]^{\cong}
	& \left(g_0^{\leq \beta(0) + \epsilon/3}, g_0^{\leq \alpha(0) + \epsilon/3}\right)\ar@{^{(}->}[r]\ar[d]^{\cong} 
	&\left(g_1^{\leq \beta(0) + 2\epsilon/3 }, g_1^{\leq \alpha(0)+ 2\epsilon/3}\right)\ar[d]^{\cong}\\
	\left(g_0^{\leq \beta(0)}, g_0^{\leq \alpha(0) }\right)  \ar[r]^{\phi_1}
	& \left(g_1^{\leq \beta(1)}, g_1^{\leq \alpha(1) }\right)   \ar[r] ^{\phi_2}
	&\left(g_0^{\leq \beta(0)}, g_0^{\leq \alpha(0) }\right) \ar[r] ^{\phi_3}
	& \left(g_1^{\leq \beta(1)}, g_1^{\leq \alpha(1) }\right).	}
\eqnnd}
The vertical diffeomorphisms are provided by deformations guaranteed by hypothesis (2) and Lemma~\ref{lem:retract}.  The resulting maps $\phi_1, \phi_2, \phi_3$ satisfy the homotopy
equivalences 
$\phi_2 \circ \phi_1 \simeq \operatorname{id}, \phi_3 \circ \phi_2 \simeq \operatorname{id}$, $\phi_3 \simeq \phi_1$. 
 \end{proof}

\section{Some homotopy theory of spectra} \label{ssec:spectra}
  We assume the reader is familiar with basic ideas from the homotopy theory of spaces, and in particular with pushout and pullback squares, cofibrations, homotopy pushout and pullback squares, and the reduced suspension of a pointed space.  Sample texts for these background notions are \cite{May, dugger, Lurie-HTT, Hatcher} as well as Appendix B of  the  first arXiv  version of this paper \cite{tanaka-traynor-v1}. In this appendix, we gather background material on the homotopy theory of spectra.  

Somewhat confusingly, there are various definitions (i.e., models) of spectra that give rise to the same theory. For example, the notion found in the popular textbook of Hatcher \cite[Section 4.F]{Hatcher} agrees with what many would now call a prespectrum (see for example \cite[Chapter 22]{May}).  

\subsection{Spectra}
Given a pointed topological space $A$, we recall that $\Omega A$ denotes the based loop space of $A$ -- the space of continuous paths that begin and end at the basepoint. While mapping spaces normally do not admit CW structures for reasons of point-set topology, they are always homotopy equivalent to CW complexes as long as 
the domain is compact and the codomain is homotopy equivalent to a CW complex~\cite{milnor-spaces-having-homotopy-type-of-CW-complex}. A similar statements holds for relative mapping spaces, and in particular, $\Omega A$ is homotopy equivalent to a CW complex.  

\begin{defn}
A {\bf spectrum}, sometimes also called an $\Omega$-spectrum, is the data of the following:
\begin{enumerate}
	\item For every $i \geq 0$, a pointed topological space $Y_i$ homotopy equivalent to a CW complex.
	\item For every $i \geq 0$, a homotopy equivalence $Y_i \to \Omega Y_{i+1}$.
\end{enumerate}
Given two spectra, a {\bf map} $f: Y \to Y'$ is the data of the following:
\begin{enumerate}
	\item For every $i \geq 0$, a continuous map $f_i: Y_i \to Y_i'$ respecting basepoints.
	\item For every $i \geq 0$, a homotopy making the following diagram commute up to homotopy:
		\eqnn
		\xymatrix{
		Y_i \ar[r] \ar[d]_{f_i} & \Omega Y_{i+1} \ar[d]^{\Omega f_{i+1}} \\
		Y'_i \ar[r] & \Omega Y'_{i+1}.
		}
		\eqnnd
\end{enumerate}
\end{defn}

\begin{rem}  For spectra,  starting the indexing of the spaces at $0$ is convenient and the standard model.  However,
the information of a spectrum is completely recoverable from a collection of indices that limit to $\infty$: if  $Y_j$ is known, then $Y_{j-1}$ is determined by
our condition that $Y_{j-1} \simeq \Omega Y_j$. 
\end{rem} 

\begin{defn} \label{defn:spec-homotopy}
Given a spectrum $Y$ and an integer $j \in \ZZ$, one can define the {\bf $j^{th}$ homotopy group of $Y$} as follows:
	\eqnn
	\pi_{j}(Y) := \pi_{j+i}(Y_{i}).
	\eqnnd
\end{defn}

So for example, $\pi_k(Y) \cong \pi_k(Y_0)$, and $\pi_{-3}(Y) \cong \pi_0(Y_3)$. Note that $\pi_j(Y)$ is well-defined up to the group isomorphisms dictated by  the homotopy equivalences $Y_i \to \Omega Y_i$.    We will later see how one can define the homology groups of a spectrum; see Section~\ref{ssec:homology-pre-spectra}.

\begin{defn}\label{defn. equivalence of spectra}
An {\bf equivalence} of spectra is a map $Y \to Y'$ that induces an isomorphism on all homotopy groups.
\end{defn}

It is a theorem that every equivalence admits an inverse up to homotopy. See for example \cite[Remark~1.4.3.8]{higher-algebra}.

\begin{defn} 
A spectrum $Y$ is called {\bf trivial} if all its homotopy groups vanish. 
\end{defn}

\begin{ex}
\label{example. trivial spectrum}
Define a spectrum $Y$ by choosing, for each $i$, $Y_i$ to be some contractible CW complex. Then $Y$ is trivial. Any map from $Y$ to any other trivial spectrum is an equivalence (Definition~\ref{defn. equivalence of spectra}). 
Similarly, suppose $X$ is a prespectrum (see Definition~\ref{defn. prespectrum}) for which every space $X_i$ is contractible. Then the spectrum associated to $X$ (Construction~\ref{construction. spectrification}) is trivial -- indeed, the homotopy colimit of~\eqref{eqn. spectrification diagram} must have trivial homotopy groups.   
\end{ex}

\subsection{Prespectra}
Prespectra are one convenient tool for producing spectra; for example, the invariants in our work arise as prespectra. 

\begin{defn}\label{defn. prespectrum}
 A {\bf prespectrum} $X$ consists of the following data: 
 \begin{enumerate}
 \item For all $i \geq 0$, a pointed space $X_{i}$ homotopy equivalent to a CW complex, and
 \item For all $i \geq 0$, a continuous map
 	$\Sigma X_{i} \to X_{{i+1}}$
	respecting basepoints. 
 \end{enumerate}
\end{defn}

 Although we have started by indexing for a prespectrum to start at $i = 0$,  the next construction shows
that one obtains a well-defined spectrum even if the indexing of the prespectrum starts at any finite $i = N>0$.

\begin{construction}[The spectrum associated to a prespectrum]
\label{construction. spectrification}
Given a prespectrum $X$, for all $i$, one has natural maps of pointed topological spaces
	\eqn\label{eqn. spectrification diagram}
	X_i \to \Omega X_{i+1} \to \Omega^2 X_{i+2} \to \ldots
	\eqnd
by utilizing the adjoint map to $\Sigma X_j \to X_{j+1}$. (A map $\Sigma X_j \to X_{j+1}$ produces 
an $X_j$'s-worth of loops in $X_{j+1}$ -- i.e., a map $X_j \to \Omega X_{j+1}$.)  Let $Y_i$ denote the homotopy colimit of the diagram~\eqref{eqn. spectrification diagram}. (One can model this as a mapping telescope if one likes; see~\cite{dugger}.)  Then the natural map $Y_i \to \Omega Y_i$ is a homotopy equivalence. 
 If the indexing of the  prespectrum $X$ begins at $N > 0$, then the above procedure produces $Y_i$, $i \geq N$; by setting $Y_{N-k} := \Omega^{k}Y_N$ for $0 \leq k \leq N$, one thus obtains a spectrum $Y$ with indexing starting at $i = 0$. We call $Y$ {\bf the spectrum associated to the prespectrum} $X$, or the  {\bf spectrification} of $X$.
\end{construction}

\begin{rem}
In fact, every spectrum $Y$ arises from some prespectrum. To see this, let $X_i := Y_i$ and define the maps $\Sigma X_i \to X_{i+1}$ to be adjoint to the maps $Y_i \to \Omega Y_{i+1}$. Indeed, if $Y$ is a spectrum, it is easy to see that the spectrification of this adjoint prespectrum is (equivalent to) $Y$. In other words, spectrification is an idempotent operation that takes prespectra to spectra.
\end{rem}

\begin{rem}
Given a prespectrum $X$, we have maps between homotopy groups
\begin{equation} \nonumber
	\pi_k(X_N) \to \pi_{k+1}(\Sigma X_N) 
	\to 
	\pi_{k+1}(X_{N+1}) \to 
	\pi_{k+2}(\Sigma X_{N+1})
	\to \dots .
\end{equation}
It is easy to verify that the colimit of the above sequence of groups is isomorphic to $\pi_k$ of the spectrum associated to $X$; see Definition~\ref{defn:spec-homotopy}.
\end{rem}

\begin{defn}[Suspension spectra]
 \label{defn:suspension-spec} 
 Given a pointed space $A$ homotopy equivalent to a CW complex, we can consider
the prespectrum 
 given by 
 $$
 X_{i} = \Sigma^i A
 $$
with the natural equivalences $\Sigma(X_i) \xrightarrow{\cong} X_{i+1}$. 
The associated spectrum is called the {\bf suspension spectrum of $A$}, and we denote it by
 $\Sigma^\infty A$.
 \end{defn}
 
\begin{ex}[The sphere spectrum $\Sphere$] \label{ex:sphere-spectrum}
The suspension spectrum of $S^0$ is often written $\Sphere$. 
 This spectrum, which plays a fundamental
role in stable homotopy theory,   
 is called the {\bf sphere spectrum}.
\end{ex}

\begin{defn} \label{not:prespectra-maps}
Fix prespectra $X$ and $X'$.  A {\bf map} $h$ from $X$ to $X'$ consists of:
\begin{enumerate}
\item For  all sufficiently  large   $i$, a base-point preserving continuous map $h_i: X_{i}\to X'_{i}$, and 
\item For all sufficiently  large $i$,
a homotopy $H_i: \Sigma X_{i} \times [0,1] \to X'_{i+1}$ 
that makes the following diagram homotopy-commutative:
	\eqn\label{eqn. suspension homotopy}
	\xymatrix{
\Sigma X_{i} \ar[d]_{\Sigma h_i} \ar[r] &X_{i+1}    \ar[d]^{h_{i+1}}  \\
\Sigma X'_{i}  \ar[r] & X'_{i+1}.
   }
   	\eqnd
\end{enumerate}
\end{defn}

\begin{rem}
By the $\Omega$-$\Sigma$ adjunction, the data of the homotopy in~\eqref{eqn. suspension homotopy} is equivalent to the data of a homotopy making the following commute:
	\eqnn
	\xymatrix{
	X_i \ar[r] \ar[d]^{h_i} & \Omega X_{i+1} \ar[d]^{\Omega h_{i+1}} \\
	X_i' \ar[r] & \Omega X'_{i+1}.
	}
	\eqnnd
In particular, by applying the spectrification construction of Remark~\ref{construction. spectrification}, one sees that a map of prespectra induces a map of the associated spectra.
We note that, in some works, one demands that a map of prespectra actually be defined on a cofinal subdiagram of CW complexes. (See for example \cite[Chapter 2]{adams-blue-book}  and \cite{beaudry-campbell}.) Because prespectra are (as the name suggests) objects we use to arrive at actual spectra, this distinction will not matter for us. Indeed, in any reasonably defined $\infty$-category of prespectra, one can recover the actual $\infty$-category of spectra via an $\infty$-categorical localization along the weak equivalences (i.e., those maps inducing isomorphisms on homotopy groups). So while our definition of prespectra does not have the same collection (even up to $\pi_0$) of all morphisms of spectra, indeed every map of prespectra does give rise to a map of spectra -- which is the only fact we will use.
For further details, we refer the reader to~\cite[Definition~3.1.13]{gepner-chapter}. 
\end{rem}

The following is a sufficient condition for a map $h$ between prespectra to induce an equivalence of the associated spectra:

\begin{prop}\label{prop:prespectra-spectra-maps} 
Let $h: X \to X'$ be a map of prespectra such that, for all $i$ large, $h_i: X_i \to X_i'$ is a homotopy equivalence. Then $h$ induces an equivalence of the  spectra  associated to $X$, $X'$.
\end{prop} 

\begin{proof} By hypothesis, 
for $i$ large, the vertical arrows below are equivalences:
	\eqnn
	\xymatrix{
	X_i \ar[d]^{\simeq}
		\ar[r] &\Omega X_{i+1} \ar[d]^{\simeq}
		\ar[r] & \ldots \ar[d]^{\simeq} \\
	X_i' 
		\ar[r] &\Omega X_{i+1}'
		\ar[r] & \ldots  
	}.
	\eqnnd
Let $Y$ denote the spectrum associated to $X$. For every $j \geq 0$, the colimit of the rows of the above diagram enjoy induced homotopy equivalences
	\eqnn
	Y_j \xrightarrow{\simeq} Y_j'
	\eqnnd
hence isomorphisms of homotopy groups $\pi_\ast(Y) \to \pi_\ast(Y')$.
\end{proof}

\begin{cor} \label{cor:quotient-stab-spec} Given a manifold $L$ with boundary $\del L$, consider a prespectrum $Q$ given by 
 $$\begin{aligned}
  Q_{N} &= D^N(L)/\left(D^N(\partial L) \cup S^{N-1}(L) \right), \quad Q_i = \Sigma Q_{i-1}, \quad \forall i > N
  \end{aligned}
  $$
  where $D^N(L), S^{N-1}(L)$ are the trivial $N$-dimensional disk and $(N-1)$-dimensional sphere bundles over $L$, and the maps $\Sigma Q_i \xrightarrow{\cong} Q_{i+1}$ are as in Example~\ref{ex:quotient-suspensions}.
  Then the spectrum associated to $Q$ is equivalent to the suspension spectrum of $L/\partial L$.
 \end{cor} 
\noindent
The proof follows immediately from Proposition~\ref{prop:prespectra-spectra-maps} and  Example~\ref{ex:quotient-suspensions}.

\begin{ex} \label{ex:homotopy-suspension-stabilized-homotopy} 
For any pointed topological space $X$, we have the sequence of homotopy groups
\begin{equation} 
\pi_k(X) \to \pi_{k+1}(\Sigma X) \to \pi_{k+2} (\Sigma^2 X) \to \dots
\label{eqn:Freudenthal} 
\end{equation}
The {\bf stable homotopy groups} of a space $X$ are defined as a colimit:
$$\pi_k^{s}(X) := \col_{j} \pi_{k+j} \Sigma^j(X).$$
We see that the 
homotopy groups of the suspension spectrum of a space $X$, $\Sigma^\infty(X)$, agree with the stabilized
homotopy groups of the space $X$:
$$\pi_k(\Sigma^\infty(X) ) = \pi_k^{s}(X).$$
A foundational result in stable homotopy theory is the Freudenthal Suspension Theorem \cite[Corollary 4.24]{Hatcher}, which implies that the $j$th maps in~\eqref{eqn:Freudenthal}  are all isomorphisms for large enough $j$; thus the colimit is computed at a finite stage of the sequence.
\end{ex}

\begin{notation}[Suspension] \label{notation:shift-spec}
 Given a prespectrum $X$, we define the 
{\bf suspension of $X$}, denoted $\Sigma X$, by shifting indices:
	\eqnn
	(\Sigma X)_i := (X)_{i+1}
	\eqnnd
Likewise, the suspension of a spectrum $Y$, denoted $\Sigma Y$,  is defined by $(\Sigma Y)_i = Y_{i+1}$.
There is a natural notion of $\Sigma^{-1} X$ as well:
	\eqnn
	(\Sigma^{-1}X)_i := (X)_{i-1}.
	\eqnnd
One can now define $\Sigma^j X$ for any $j \in \ZZ$ (and $\Sigma^j Y$ for any spectrum $Y$).
Observe that for $j < 0$, $\Sigma^j X$ may be thought of as having indexing starting at $-j > 0$.
\end{notation}

\begin{rem}
One often denotes $\Sigma^{-1} X$ also as $\Omega X$. Note that the operations $\Omega, \Sigma$ are not inverse operations for (pointed) spaces, but they are for spectra and prespectra. 
\end{rem}

\begin{ex}[Shifts of the sphere spectrum] \label{ex:suspension-Sn}
$\Sigma^\infty$ commutes with suspensions (of spaces and of spectra). In particular, we have that $\Sigma^\infty(S^n) \simeq \Sigma^n \Sphere$. 
\end{ex}

\begin{ex}   \label{ex:homotopy-calculations} 
By construction, the  homotopy groups of  the spectrum $\Sphere = \Sigma^\infty(S^0)$ agree
with the stable homotopy groups of $S^0$:
$$\pi_k (\Sphere) = \pi_k\left(\Sigma^\infty\left(S^0\right)\right) := \col_{j} \pi_{k+j}\left(S^j\right) = \pi_k^{s}\left(S^0\right).$$
A result of Serre tells us that
$\pi_k^s(S^0)$ is a finite group for all $k > 0$. 
The stable homotopy groups of $S^0$, $\pi_k^s(S^0)$, have been calculated for $k \leq 84$~\cite{IWX}. 
 For example:
\begin{center}
    \begin{tabular}{ | m{3em} ||c |c |c |c |c |c |c |c| c| c|} 
    \hline
    $k$ &  $\leq -1$ &0 & 1 & 2 & 3 & 4 & 5 & 6 & 7 & 8 \\ 
     \hline 
     $\pi_k^{s}(S^0)$ 
      & 0 & $\ZZ$ & ${\ZZ}{/2}$& ${\ZZ}{/2}$ & ${\ZZ}{/24}$ & 0 & 0 & $\ZZ/2$ & $\ZZ/240$ & $\ZZ/2 \times \ZZ/2$
      \\ 
     \hline  
    \end{tabular}
\end{center}

It is easy to check that shifting a spectrum shifts the homotopy groups. So for example, $\Sigma^\infty S^3$ has the following homotopy groups:
\begin{center}
\begin{tabular}{ | m{3em} ||c |c |c |c |c |c |c |c| c| c| } 
\hline
$k$ &  $\leq 2$ & 3 & 4 & 5 & 6 & 7 & 8 & 9 & 10 & 11 \\ 
 \hline 
 $\pi_{k}^s(S^3)$  & 0 & $\ZZ$ & $\ZZ/2$& $\ZZ/2$ & $\ZZ/{24}$ & 0 & 0 & $\ZZ/2$ & $\ZZ/240$ & $\ZZ/2 \times \ZZ/2$
  \\ 
 \hline 
\end{tabular}
\end{center}
\end{ex}

\begin{ex}[$\Sigma^\infty(T^2)$]\label{example. suspension of torus}
One can construct $\Sigma T^2$ by suspending the usual CW structure of the torus. In particular, $\Sigma T^2$ is obtained by gluing $D^3$ to the wedge sum $S^2 \vee S^2$ along a map $\del D^3 = S^2 \to S^2 \vee S^2$ homotopic to $aba^{-1}b^{-1} = 0 \in \pi_2(S^2 \vee S^2) \cong H_2(S^2 \vee S^2)$ (Here, $a$ and $b$ represent generators for the homology of each copy of $S^2$ in the wedge sum). Because $D^3$ is attached to $S^2 \vee S^2$ along a null-homotopic map, we conclude there is a homotopy equivalence
	\eqnn
	\Sigma T^2 \simeq S^2 \vee S^2 \vee S^3.
	\eqnnd
Tracing through the definitions, and noting that $\Sigma$ commutes with $\vee$ (thereby defining the wedge sum of spectra -- which we often denote by $\oplus$) we thus find 
$$\Sigma^\infty(T^2) \simeq \Sigma^{\infty}(S^1) \oplus \Sigma^{\infty}(S^1) \oplus \Sigma^\infty(S^2).$$
\end{ex} 

\subsection{Definitions of homology for prespectra and spectra} \label{ssec:homology-pre-spectra}

\begin{defn}[Homology of a prespectrum]\label{defn. homology of prespectrum}
Fix a coefficient abelian group $A$ and a prespectrum $X$. Then for all $i \in \ZZ$, the $i^{th}$ homology group of $X$ with coefficients in $A$ is defined to be the sequential colimit (i.e., direct limit) of abelian groups:
	\begin{equation}
	\label{eqn. homology of prespectrum}
	H_i(X;A)
	:=
	\colim_{j \to \infty} \widetilde{H}_{i+j}(X_j;A)
	\end{equation}
where $\widetilde{H}_{i+j}(X_j;A)$ is the ordinary reduced homology of a pointed topological space. 
\end{defn}

\begin{rem}
Each step in the sequential colimit diagram of Definition~\ref{defn. homology of prespectrum} is obtained as the composition
\begin{equation} 
\label{eqn:hom-lim}
	 \widetilde{H}_{i+j}(X_j;A)
	 \cong \widetilde{H}_{i+j+1}(\Sigma X_j;A)
	 \to \widetilde{H}_{i+j+1}(X_{j+1};A)
\end{equation} 
where the first isomorphism is the suspension isomorphism for homology, and the second map is the pushforward along the defining maps of $X$.
\end{rem}

\begin{rem}
Note that when $i$ is a negative number, the colimit is taken over all $j$ such that $j \geq |i|$. For example, when $i = -3$, 
	$$
	H_{-3}(X;A) := \colim \left( \widetilde{H}_0(X_3;A)
	\to
	\widetilde{H}_1(X_4;A)
	\to
	\widetilde{H}_2(X_5;A)
	\to \ldots \right).
	$$
\end{rem}

When $Y$ is a prespectrum, the homotopy equivalences $Y_i \to \Omega Y_{i+1}$ have adjoints $\Sigma Y_i \to Y_{i+1}$. We can then again use the sequence of
maps as in Equation~(\ref{eqn:hom-lim})   to define:

\begin{defn}[Homology of a spectrum]\label{defn. homology of spectrum}
Fix a coefficient abelian group $A$ and a prespectrum $Y$. Then the $i^{th}$ homology group of $Y$ with coefficients in $A$ is defined to be the  sequential colimit:	
	\begin{equation}
	\label{eqn. homology of spectrum}
	H_i(Y;A)
	:=
	\colim_{j \to \infty} \widetilde{H}_{i+j}(Y_j;A)
	\end{equation}
where $\widetilde{H}_{i+j}(Y_j;A)$ is the ordinary reduced homology of a pointed topological space. 
\end{defn}

Let us sketch a proof of the following:

\begin{prop}
Let $Y$ be the spectrum associated to a prespectrum $X$. Then there is a natural isomorphism between the homology groups~\eqref{eqn. homology of prespectrum} and~\eqref{eqn. homology of spectrum}.
\end{prop}

\begin{proof}
For brevity, we fix, and omit from our notation, the coefficient $A$.
Fixing $i$, let us recall that (by definition of direct limit) $H_i(X)$ is isomorphic to the following quotient abelian group:
	$$
	{\frac 
	{\text{$a$ such that for some $j$, $a \in H_{i+j}(X_j)$}}
	{
		\text{
		$a \sim a'$  
		if $a'$ is the image of $a$ under the map 
		$H_{i+j}(X_j) \to H_{i+j'}(X_{j'})$
		}
	}.
	}
	$$
We can parse the maps $\Sigma Y_k \to Y_{k+1}$ of the associated prespectrum as arising from the diagram
	$$
	\xymatrix{
	\Sigma X_k \ar[r]  & \Sigma \Omega X_{k+1} \ar[r] \ar[d] & \Sigma \Omega^2 X_{k+2} \ar[r] \ar[d] & \ldots\ar[d]  \\
	&	    X_{k+1} \ar[r] & \Omega X_{k+2} \ar[r] & \ldots
	}
	$$
where the colimit of the top row is the space $\Sigma Y_k$, and the colimit of the bottom row gives $Y_{k+1}$. The composition $\Sigma X_k \to \Sigma \Omega X_{k+1}\to X_{k+1}$ in the above diagram is the map arising in the definition of the prespectrum $X$. Because homology respects sequential colimits, we see that the map on homology $\widetilde{H}_\ast (\Sigma Y_k) \to \widetilde{H}_\ast (Y_{k+1})$ is induced by taking the homology $\widetilde{H}_\ast$ of every space appearing in the above diagram, then examining the map between the direct limit of  (the homology of) each row.

So we define a group homomorphism $H_i(X) \to H_i(Y)$ as follows. Given an element $[a]$ in $H_i(X)$, so that $[a]$ is represented by some element $a \in \widetilde{H}_{i+j}(X_j)$ for some $j$, one has a well-defined element $f([a])$ of $\widetilde{H}_{i+j}(Y_j)$ arising from the map $X_j \to Y_j$. (This map arises because, by definition, $Y_j$ is the colimit of $X_j \to  \Omega X_{j+1} \to \ldots$. The well-definedness of $f([a])$ follows from the previous paragraph.) 

To prove $f$ is a surjection, choose an element $[b]$ of $H_i(Y)$, which by definition is represented by some element of $\widetilde{H}_{i+j}(Y_j)$, and hence (for some $n$) by an element $b' \in\widetilde{H}_{i+j}\Omega^n X_{j+n}$. On the other hand, we have a commuting diagram (generalizing the above two-row diagram)
	$$
	\xymatrix{
	\Sigma^n X_j \ar[r] & \Sigma^n \Omega X_{j+1} \ar[d]  \ar[r] & \ldots \ar[r] \ar[d] & \ldots \ar[r] \ar[d]  & \Sigma^n \Omega^n X_{j+n} \ar[d] \ar[r] & \ldots \ar[d]\\
		& \Sigma^{n-1} X_{j+1} \ar[r] &\ldots \ar[r] \ar[d] & \ldots \ar[r] \ar[d] & \ldots \ar[r] \ar[d] & \ldots \ar[d]  \\
		&& \ldots   \ar[r]& \ldots  \ar[d] \ar[r]& \ldots \ar[r]  \ar[d]& \ldots  \ar[d] \\
		&&  & \Sigma X_{j+n-1} \ar[r]  & \Sigma \Omega X_{j+n} \ar[r] \ar[d] & \ldots  \ar[d] \\
			&&&& X_{j+n} \ar[r] & \ldots .
	}
	$$
$\Sigma^n b'$ maps to some element $b \in \widetilde{H}_{i+j+n}(X_{j+n})$ by the composite map $\Sigma^n \Omega^n X_{j+n} \to X_{j+n}$ gleaned from the relevant column above.
Thus, under the composition $\Sigma^n Y_j \to \Sigma^{n-1} Y_{j+1} \to \ldots \to Y_{j+n}$, we see that $b'$ is identified with the element in $\widetilde{H}_{i+j+n} Y_{j+n}$ arising from $b$. This shows that $f$ is a surjection.

Suppose $f([a]) = 0$. Choosing a representative $a \in \widetilde{H}_{i+j}(X_j)$, $f([a])$ is represented by the corresponding element $a'$ in $\widetilde{H}_{i+j}(Y_j)$. That $f([a])=0$ means that for some finite $n$, $\Sigma^n a'$ is in the kernel of the map on homology of the composition $\Sigma^n Y_j \to \Sigma^{n-1} Y_{j+1} \to \ldots \to Y_{j+n}$, which means that (using the staircase-shaped diagram above) $\Sigma^n a'$ is sent to zero under the map $\Sigma^n X_j \to \Omega^k X_{j+n+k}$ for some finite $k$. In particular, $\Sigma^{n+k}a'$ is sent to zero under the composite map $\Sigma^{n+k} X_j \to \Sigma^k \Omega^k X_{j+n+k} \to X_{j+n+k}$, which is equal to the obvious composition of structure maps of $X$. Thus, $[a] = 0$ in $H_i(X)$.
\end{proof}

\subsection{Homotopy pushouts of spectra} 
A homotopy pushout square of spectra consists of a diagram of spectra commuting up to homotopy
 $$\xymatrix{
 W \ar[d] \ar[r] & B \ar[d] \\
 A \ar[r] &  C
 }$$
 satisfying the universal mapping-out property with respect to any other spectrum $Z$:
 $$
	\xymatrix{
	W \ar[rr] \ar[ddrrr] \ar[d] & & B \ar[d] \ar[ddr]&
	\\
	A \ar[rr]\ar[drrr] & & C\ar@{-->}[rd]  & \\
	& & & Z.
	}
$$

At the level of spaces, pushout squares need not be pullback squares.    However, an important
property of spectra is that pushout and pullback squares coincide.

 \begin{thm}\label{thm:pushout=pullback} For spectra, homotopy pushout squares are homotopy pullback squares.
  \end{thm} 
This fact can be deduced, for example, from \cite[Definition 1.1.1.9,
Definition 1.4.2.8,
Corollary 1.4.2.17,
Definition 1.4.3.1] {higher-algebra} or found in the discussion following \cite[Definition 2.2.1]{SS-uniqueness}.

\begin{rem}
\label{remark. spectrification preserves colimits}
Further, the spectrification operation is a left adjoint  in an adjunction between $\infty$-categories, so it sends homotopy pushout squares (of prespectra) to homotopy pushout squares (of spectra) -- see for example Remark~3.1.15 of~\cite{gepner-chapter}.
\end{rem}

\subsection{Mapping spaces and homotopy groups}
\label{section. homotopy groups via mapping spaces}
It turns out that the functor sending a pointed space $A$ to its suspension spectrum $\Sigma^\infty A$ is a left adjoint in an adjunction between $\infty$-categories, with right adjoint given by sending a spectrum $Y$ to $Y_0$. In particular, we have a homotopy equivalence of spaces
	\eqnn
	\hom(S^0, Y_0) \simeq \hom(\Sphere,Y)
	\eqnnd
where the lefthand side is the space of pointed continuous maps, and the righthand side is the mapping space of spectra. The above being a homotopy equivalence, we have isomorphisms of abelian groups; and also noting that $\hom(S^0,Y_0) \simeq Y_0$,
	\eqnn
	\pi_k Y \cong \pi_k Y_0 \simeq \pi_k \hom(\Sphere,Y),
	\qquad
	k \geq 0.
	\eqnnd
In other words, the homotopy groups of spectra are computed via homotopy groups of mapping spaces out of the sphere spectrum. By shifting, we also obtain homotopy groups in non-positive degrees as well:
	\eqnn
	\pi_k Y \cong \pi_0 Y_{-k} \simeq \pi_0 \hom(\Sigma^k \Sphere, Y),
	\qquad
	k \leq 0
	\eqnnd

\subsection{Long exact sequences}
 \begin{prop} \label{prop:pull-back-spectra-homotopy}
Suppose one has a homotopy pullback square (or, by Theorem~\ref{thm:pushout=pullback}, a homotopy pushout square) of spectra
	\eqnn
	\xymatrix{
	W \ar[r] \ar[d] & B \ar[d] \\
	A \ar[r] & C.
	}
	\eqnnd
Then the induced maps on homotopy groups
	\eqnn
	\pi_i W \to \pi_iA \oplus \pi_i B \to \pi_i C
	\eqnnd
admit connecting maps $\pi_i C \to \pi_{i-1}W$ fitting into a long exact sequence.
\end{prop}

\begin{proof}
Fix an integer $j$, and let $\Sphere[j] = \Sigma^j \Sphere$ denote the $j^{th}$ suspension of the sphere spectrum;  see Notation~\ref{notation:shift-spec}.
 By definition of homotopy pullback (of spectra), the homotopy-commuting square of mapping spaces  
	\eqnn
	\xymatrix{
	\hom(\Sphere[j],W)  \ar[r] \ar[d] & \hom(\Sphere[j],B)  \ar[d] \\
	\hom(\Sphere[j],A)  \ar[r] & \hom(\Sphere[j],C)
	}
	\eqnnd
is a homotopy pullback square (of spaces). It is thus classical that one obtains a long exact sequence involving the homotopy groups of these mapping spaces. By Section~\ref{section. homotopy groups via mapping spaces}, we know that for any spectrum $X$, we have natural isomorphisms
	\eqnn
	\pi_{i+j} X \cong \pi_{0}\hom(\Sphere[i+j],X)
	\cong \pi_{i}\hom(\Sphere[j],X),
	\eqnnd
so the long exact sequence of homotopy groups of the above spaces ends as follows:
	\eqnn
	\ldots \to \pi_{1+j} W \to \pi_{1+j} A \oplus \pi_{1+j} B \to \pi_{1+j} C \to
	\pi_j W \to \pi_j A \oplus \pi_j B \to \pi_j C.
	\eqnnd
By varying the value of $j$, we can extend this exact sequence to more and more negative values of $j$, thus obtaining the desired long exact sequence.
\end{proof}

\begin{ex}
We saw in Example~\ref{example. suspension of torus} that the suspension of the torus splits as a wedge sum $\Sigma(S^1 \vee S^1 \vee S^2)$. Wedge sums are trivial pushouts of pointed spaces, and $\Sigma^\infty$ preserves pushouts, so $\Sigma^\infty T^2$ is a direct sum/trivial pushout  of spectra $\Sphere^1,\Sphere^1$, and $\Sphere^2$. We conclude
  \eqnn
  \pi_k\left(\Sigma^\infty(T^2)\right) \cong \pi_{k}(\Sphere^1) \oplus \pi_{k}(\Sphere^1) \oplus\pi_{k}(\Sphere^2)
  \eqnnd
 \end{ex}

 \vfill
 \pagebreak

\section{Identifying a normal bundle}
In this section, we identify the negative index-bundle associated to a Morse-Bott difference function (Corollary~\ref{cor:neg-index-trivial}) to a trivial vector bundle. This is a necessary step (see proof of Lemma~\ref{lem:neg-end}) in proving the spectral lift of the Seidel isomorphism.

First, fix a Legendrian immersion $X \immto J^1B$ (or a Lagrangian immersion $X \immto T^*B$). We will say that this immersion {\em admits a generating family} if there exists a generic (Assumption~\ref{assumption. f generic}) generating family $f: B \times \RR^N \to \RR$ and a diffeomorphism $X \cong \Sigma_f$ such that this diffeomorphism intertwines the given immersions of $X$ with the immersions~\eqref{eqn. immersions of sigma f}.

\begin{rem}
Independently, Giroux \cite{Giroux:gf-exist} and Latour \cite{Latour} showed that, for a closed manifold $X$, a Legendrian immersion $X \immto J^1B$ admits a generating family if and only if the associated stable Gauss map (for the immersion $X \immto T^*B$) is trivial. (See also Proposition~2.6 of~\cite{acgk}.) Our aims here are more modest, instead classifying the stable normal class of the natural projection $X \to B$.
\end{rem}

Suppose $X \immto J^1B$ (or $X \immto T^*B$) is an immersed Legendrian (or Lagrangian) and consider the map
	$X \to B$
obtained as the composition $X \immto J^1B \to B$ (or the composition $X \immto T^*B \to B$). For all $N \geq 0$, composition with the inclusion $B \cong B \times \{ {\bf 0}\} \subset B \times \RR^N$ induces the smooth map
	\eqn\label{eqn. X to B times R^N fiber constant map}
	X \to B \times \RR^N.
	\eqnd

 \begin{prop}\label{prop. trivial normal}
Fix a smooth (possibly non-compact) manifold $X$. If $X \immto J^1B$ admits a generating family, then for some $N \geq 0$, the induced map~\eqref{eqn. X to B times R^N fiber constant map} is homotopic rel $B$ (that is, homotopic through maps that respect the projection to $B$) to a smooth embedding with trivial normal bundle. 
 \end{prop}

\begin{rem} 
Proposition~\ref{prop. trivial normal} strongly constrains whether an immersed Lagrangian $X \immto  T^*B$
(or an immersed Legendrian $X \immto  J^1B$)  admits a generating family. Indeed, the proposition yields a short exact sequence of vector bundles
	\eqnn
	0 \to TX \to T(B \times \RR^N)|_{X} \to \underline{\RR^N}  \to 0,
	\eqnnd
where $\underline\rr^N$ is the trivial bundle of rank $N$ over $X$. Thus the pullback of $TB$ to $X$ must be stably isomorphic to $TX$.  In particular, all characteristic classes of $X$ must match those of $B$ under the map $X\to B$. 
\end{rem}

\begin{ex}
Suppose $B$ is Euclidean space (or a smooth manifold with stably trivial tangent bundle -- a sphere of any dimension, a Riemann surface, a Lie group, et cetera). If there exists
a Lagrangian immersion $X \looparrowright T^*B$  admitting a generating family, then $X$ must have stably trivial tangent bundle. In particular, all the characteristic classes of  $X$ must
vanish.
\end{ex}

\begin{ex}
We also note that
Proposition~\ref{prop. trivial normal} does not assume compactness of $B$ or $X$. In particular, one can conclude that there does not exist a Lagrangian immersion 
  $\CC P^n \setminus  \{pt\}  \looparrowright T^*\RR^{2n}$ admitting  a generating family. \end{ex}

\begin{rem}
\label{remark. tangential structures on branes are natural}
It is known that many tangential obstructions must vanish for an $\Sphere$-linear Floer theory to be unambiguously defined on, say, an embedded Lagrangian in $T^*B$. We view the topological necessities for the existence of generating families as a similar constraint.
\end{rem}

\begin{proof}[Proof of Proposition~\ref{prop. trivial normal}.] 
We will prove the statement for the Lagrangian case. The argument for the Legendrian situation is similar. Recall we may identify $\Sigma_f$ with a subset of $B \times \RR^N$ via~\eqref{eqn. sigma f as a subset}. Consider the composition
	\eqnn
	\Sigma_f \subset B \times \RR^N \xrightarrow{df} \Gamma_{df} \subset T^*B \times T^*\RR^N \to T^*\RR^N \to (\RR^N)^\vee
	\eqnnd
where the last map is the projection to the cotangent coordinate of $T^*\RR^N$. By the definition of $\Sigma_f$ (as the locus where the derivatives of $f$ in the $\RR^N$ direction vanish) the image of the above composition is contained in the origin of $(\RR^N)^\vee$. In particular, $T\Sigma_f \to T((\RR^N)^\vee)$ is the zero map. By definition of normal bundle -- as the quotient of $T(B \times \RR^N)|_{\Sigma_f}$ by $T\Sigma_f$ -- we thus witness the induced map
	\eqnn
	N \Sigma_f \to T( (\RR^N)^\vee)
	\eqnnd
from the normal bundle. By genericity (Assumption~\ref{assumption. f generic}) the above map must be a fiber-wise surjection. A dimension count therefore shows $N{\Sigma_f} \to T\left( (\RR^N)^\vee\right)$ is a fiberwise isomorphism. Triviality of $N{\Sigma_f}$ follows.

Next we observe that the embedding $\Sigma_f \subset B \times \RR^N$ is, via a contracting homotopy of $\RR^N$ to its origin, homotopic to a map from $\Sigma_f$ with image contained in $B \cong B \times \{ {\bf 0}\} \subset B \times \RR^N$. This map is, by definition of $\Sigma_f$, identified with the composition $\Sigma_f \to T^*B \to B$.

Now, by applying the hypothesis that one has a diffeomorphism $X \cong \Sigma_f$ respecting the map to $T^*B$ (and hence the map to $B$) we obtain the desired result.
\end{proof}

Recall from Proposition~\ref{prop:leg-crit-point} that the critical locus of the difference function $\delta_f$ contains a copy of $\Sigma_f$;  similar arguments show 
that when $f$ is a generating family for an immersed   Lagrangian $L \subset T^*B$, the critical locus of $\delta_f$ likewise contains a copy of $\Sigma_f$ inside its $0$-level set.
 	Explicitly, the identification is given by the composition
	\eqn\label{eqn. identifying L with critical locus}
	\Sigma_f \subset B \times \RR^N
	\buildrel{\Delta_B}\over{\longrightarrow} B \times \RR^N \times \RR^N,
	\eqnd
where $\Delta_B$ denotes the diagonal map, $\Delta_B(x, \e) = (x, \e, \e)$. Thus, it makes sense to restrict vector bundles to $\Sigma_f$ (by pulling back along the above composition).
 
Now let us further assume that one can choose a Riemannian metric on $B \times \RR^N \times \RR^N$ so that the difference function $\delta_f$ is Morse-Bott.  It is natural to study
the Hessian of $\delta_f$. We  have the following:

\begin{prop}\label{prop. negative index is normal bundle}
The  negative index bundle of the Hessian of $\delta_f$, restricted to $\Sigma_f$,
is isomorphic to $N\Sigma_f$ (the normal bundle of the inclusion
$\Sigma_f \to B \times \RR^N$).
\end{prop}

Combining Propositions~\ref{prop. negative index is normal bundle} and \ref{prop. trivial normal}, we find:
\begin{cor} \label{cor:neg-index-trivial} The  negative index bundle of the Hessian of $\delta_f$, restricted to $\Sigma_f$, is a trivial vector bundle. 
\end{cor}

\begin{rem}
As we will see in the proof, the positive index bundle, restricted to $\Sigma_f$, is identified with the negative index bundle, so Proposition~\ref{prop. negative index is normal bundle} holds for the positive index bundle as well.
\end{rem}

\begin{proof}[Proof of Proposition~\ref{prop. negative index is normal bundle}.] The domain of $\delta_f$, 
$B \times \RR^N \times \RR^N$, has an order 2 diffeomorphism:
\eqnn
	\swap: (x,\e, \te) \mapsto (x,\te,\e).
\eqnnd
The diagonal $\Delta:= \{(x,\eta,\eta)\}$ is in the fixed point locus, so the derivative induces an order-2 automorphism $d(\swap)$ of the tangent bundle of $T(B \times \RR^N \times \RR^N)|_\Delta$.
Moreover, the composition of $\delta_f$ with the swap map produces $-\delta_f$:
\eqnn
\delta_f \circ \swap (x, \e, \te) = \delta_f (x, \te, \e) 
= f(x,\e) - f(x,\te)
= -\delta_f (x, \e, \te).
\eqnnd
Thus along the diagonal $\Delta$  the swap map has the effect of swapping the negative and positive eigenspaces of the Hessian of $\delta_f$  
and of preserving the zero eigenspace of the Hessian. This proves that the negative and positive eigenspaces have equal dimension. On the other hand, we know $\dim X_\leg = \dim B$, and we have assumed $f$ is Morse-Bott; so we conclude that the negative eigenspace has dimension $N$.

Now choose $(x, \e) \in \Sigma_f$, and choose a non-zero tangent vector 
\eqnn
	v \in T_{\Delta_B (x, \e)}  (\Delta_B (\Sigma_f)
\eqnnd
in the negative eigenspace of the Hessian.   
 
{\em Claim.}  $$v + d(\swap)(v) \notin T_{\Delta_B(x, \e)} (\Delta_B (\Sigma_f)).$$  
By the choice of $v$,   $v + d(\swap)(v)$ is in the span of the negative and positive eigenspaces. By the Morse-Bott Lemma,  $T_{\vartriangle(x, \e)} (B \times \rr^N \times \rr^N)$ is a  direct sum of the negative, positive, and null eigenspaces, and, since we are in a  Morse-Bott situation, the null eigenspace of the Hessian at $\Delta_B(x, \e)$ is precisely the tangent space to $\Delta_B(\Sigma_f)$. The Claim follows.
 
 It is immediate that $\id + d(\swap)$  
has image contained in the diagonal tangent space of $B \times \RR^N \times \RR^N$ -- i.e., the vector bundle of vectors of the form $(u,v,v)$. Restrict $\id + d(\swap)$ to the negative eigenbundle over $\Delta_B(\Sigma_f)$ and for every $\Delta_B(x, \e)$, $(x, \e) \in \Sigma_f$, consider the composition
	\begin{align}
	\text{Negative eigenspace}
	& \xrightarrow{\id + d(\swap)}
	\text{Diagonal tangent space of $B \times \RR^N \times \RR^N$} \nonumber \\
	& \cong 
	\text{Tangent space of $B \times \RR^N$} \nonumber \\
	& \to
	N\Sigma_f \subset B \times \RR^N. \nonumber
	\end{align}
  By the Claim, the above composition is an injection. A dimension count shows it must be an isomorphism.
\end{proof}

\bibliographystyle{amsplain} 
\bibliography{main}

\end{document}